\documentclass[11pt, twoside]{amsart}

\title{Indecomposable objects in Khovanov--Sazdanovic's generalizations of Deligne's interpolation categories}
\date{\today}

\author{Johannes Flake}
\address{Max Planck Institute for Mathematics,
Vivatsgasse 7, 53111 Bonn, Germany}
\email{flake@mpim-bonn.mpg.de}

\author{Robert Laugwitz}
\address{School of Mathematical Sciences,
University of Nottingham, University Park, Nottingham, NG7 2RD, UK}
\email{robert.laugwitz@nottingham.ac.uk}

\author{Sebastian Posur}
\address{University of Münster,
Fachbereich Mathematik und Informatik,
Einsteinstraße 62,
48149 Münster,
Germany}
\email{sebastian.posur@uni-muenster.de}

\usepackage{import}
\usepackage{amsmath,amsfonts,amsthm,amssymb}
\usepackage[alphabetic, initials]{amsrefs}
\usepackage[english]{babel}
\usepackage{url}
\usepackage{fancyhdr}
\usepackage{graphicx}
\usepackage{microtype}
\usepackage{wrapfig, caption}
\usepackage[
colorlinks=true,
linkcolor=black, %
anchorcolor=black,%
citecolor=black, %
urlcolor=black, %
]{hyperref}
\usepackage{cleveref} %
\usepackage{geometry}

\usepackage[all]{xy}
\usepackage{tikz}
\usetikzlibrary{positioning}
\usetikzlibrary{shapes.geometric,tqft}
\usetikzlibrary{calc}

\tikzset{
 mytikzlink/.style={
  line width=0.6pt,
  preaction={draw, {}-{}, line width=3pt, white}
 }
}

\makeatletter
\newcommand{\superimpose}[2]{%
  {\ooalign{$#1\@firstoftwo#2$\cr\hfil$#1\@secondoftwo#2$\hfil\cr}}}
\makeatother

\newcommand\sfO{\mathsf{O}}
\newcommand\Sp{\mathsf{Sp}}
\newcommand\SL{\mathsf{SL}}

\newcommand{\ov}[1]{\overline{#1}}
\newcommand{\un}[1]{\underline{#1}}
\newcommand{\lmod}[1]{#1\text{-}\mathbf{Mod}}
\newcommand{\rmod}[1]{\mathbf{Mod}\text{-}#1}

\newcommand{\lfmod}[1]{#1\text{-}\mathbf{mod}}
\newcommand{\rfmod}[1]{\mathbf{mod}\text{-}#1}
\newcommand{\lvec}[1]{#1\text{-}\mathbf{vec}}
\newcommand{\lVec}[1]{#1\text{-}\mathbf{Vec}}
\newcommand{\rfproj}[1]{\mathbf{proj}\text{-}#1}

\newcommand{\Aut}{\operatorname{Aut}}

\newcommand{\cha}{\operatorname{char}}

\newcommand{\morker}[1]{\mathcal{I}_{#1}}%
\newcommand{\objker}[1]{\mathrm{ker}({#1})}

\newcommand{\End}{\operatorname{End}}

\newcommand{\Gal}{\operatorname{Gal}}

\newcommand{\gr}{\operatorname{gr}}

\newcommand{\Hom}{\operatorname{Hom}}

\newcommand{\mult}{\operatorname{mult}}

\newcommand{\one}{\mathbf{1}}
\newcommand{\ord}{\operatorname{ord}}
\newcommand{\OSp}{\mathsf{OSp}}

\newcommand{\Rep}{\operatorname{Rep}}
\newcommand{\uRep}{\protect\underline{\mathrm{Re}}\!\operatorname{p}}

\newcommand{\Stab}{\operatorname{Stab}}

\newcommand{\Sym}{\mathrm{Sym}}

\newcommand{\rad}{\operatorname{rad}}
\newcommand{\Xk}[1]{\operatorname{Irrep}_\kk(#1)}
\newcommand{\Indec}{\operatorname{Indec}}
\newcommand{\Mg}{\Omega}
\newcommand{\Mgr}{\Omega'}
\newcommand{\Gr}{\operatorname{K}_0}
\newcommand{\Kar}{\operatorname{Kar}}
\newcommand{\Symu}[1]{\cS_{u_{#1}}}

\newcommand{\g}{\mathfrak{g}}

\newcommand{\mC}{\mathbb{C}}

\newcommand{\mR}{\mathbb{R}}
\newcommand{\mQ}{\mathbb{Q}}

\newcommand{\mZ}{\mathbb{Z}}
\newcommand{\mN}{\mathbb{N}}
\newcommand{\mH}{\mathbb{H}}

\newcommand{\mF}{\mathbb{F}}

\newcommand{\cA}{\mathcal{A}}
\newcommand{\cC}{\mathcal{C}}
\newcommand{\cD}{\mathcal{D}}

\newcommand{\cI}{\mathcal{I}}
\newcommand{\cJ}{\mathcal{J}}

\newcommand{\cN}{\mathcal{N}}

\newcommand{\cS}{\mathcal{S}}

\newcommand{\isomorph}{\stackrel{\sim}{\longrightarrow}}

\newcommand{\DCob}[1]{\mathrm{DCob}_{#1}}
\newcommand{\SCob}[1]{\mathrm{SCob}_{#1}}
\newcommand{\Cob}[1]{\mathrm{Cob}_{#1}}

\newcommand\eps{\varepsilon}
\newcommand\RepSt{\uRep(S_T)}
\newcommand\RepOt{\uRep(\sfO_T)}

\newcommand{\Grel}[1]{[V_{#1}]}

\newcommand\K{K} %
\newcommand\KK{\mathbb{K}} %
\newcommand\kk{\Bbbk} %

\newcommand\id{{\operatorname{id}}}
\newcommand\s\sigma
\newcommand\p\pi
\newcommand\scap{s_{\mathrm{cap}}}
\newcommand\scup{s_{\mathrm{cup}}}
\newcommand\sdel{s_{\Delta}}
\newcommand\sm{s_{m}}

\newcommand\sx{s_{\mathrm{X}}}

\renewcommand\o\otimes

\newcommand\h{V_n}
\newcommand\osp{\mathfrak{osp}}

\newcommand\im{\operatorname{im}}

\renewcommand\l\lambda

\tikzset{
	partition/.style={
      scale=0.5,
      yscale=-1,
      baseline={([yshift=-0.5ex]current bounding box.center)}
    }
}

\tikzset{
    bend/.cd,
    0/.style={},
    1/.style={bend right},
    -1/.style={bend left}
}
\newcommand\makePartPt[1]{({Mod(#1,10)},{(#1-Mod(#1,10))*.1})}
\newcommand\makePartLn[2]{%
\pgfmathtruncatemacro\bend{%
(int(#1/10)==int(#2/10)) ?
(#1<10 ? 1 : -1)*(#1>#2 ? 1 : -1)
: 0%
}
\draw ({mod(#1,10)},{int(#1/10}) to[bend/\bend] ({mod(#2,10)},{int(#2/10)});
}
\newcommand\tp[1] {%
\tikz[partition] {
\draw[white,opacity=0] (1,0)--(1,1); %
\def\j{0}
\foreach \i [remember=\i as \j] in {#1} {
  \ifnum \i>0
    \ifnum \j>0
      \makePartLn{\i}{\j};
    \fi
    \draw[fill] \makePartPt\i circle (2.5pt);
  \fi
} %
}}

\newtheoremstyle{mystyle}%
  {0.5cm}                   %
  {0.5cm}                   %
  {\normalfont}           %
  {}                      %
  {\itfont\bfseries}  %
  {:}                     %
  {0.3cm}              %
  {\thmname{#1}}

\newtheoremstyle{defstyle}%
  {0.5cm}                   %
  {0.5cm}                   %
  {\normalfont}           %
  {}     %
  {\normalfont\bfseries}  %
  {:}                     %
  {0.3cm}              %
  {\thmname{#1}\thmnumber{ #2}\thmnote{ (#3)}}
\numberwithin{equation}{section}
\newtheorem*{rep@theorem}{\rep@title}
\newcommand{\newreptheorem}[2]{%
\newenvironment{rep#1}[1]{%
 \def\rep@title{#2 \ref{##1}}%
 \begin{rep@theorem}}%
 {\end{rep@theorem}}}
\makeatother

\newtheorem{introtheorem}{Theorem}

\newtheorem{theorem}{Theorem}[section]

\newtheorem{proposition}[theorem]{Proposition}
\newreptheorem{proposition}{Proposition}
\newtheorem{corollary}[theorem]{Corollary}
\newreptheorem{corollary}{Corollary}
\newtheorem{lemma}[theorem]{Lemma}

\newtheorem*{theorem*}{Theorem}
\newreptheorem{theorem}{Theorem}

\theoremstyle{definition}
\newtheorem{definition}[theorem]{Definition}

\newtheorem{construction}[theorem]{Construction}

\theoremstyle{remark}
\newtheorem{example}[theorem]{Example}
\newtheorem{remark}[theorem]{Remark}

\renewcommand{\sectionmark}[1]		%
	{
	\markboth{\small\it \thesection{} #1}{}
	}

\makeatletter
\@namedef{subjclassname@2020}{%
  \textup{2020} Mathematics Subject Classification}
\makeatother


\begin{document}

\begin{abstract}
    Khovanov and Sazdanovic recently introduced symmetric monoidal categories parameterized by rational functions and given by quotients of categories of two-dimensional cobordisms. 
    These categories generalize Deligne's interpolation categories of representations of symmetric groups. In this paper, we classify indecomposable objects and identify the associated graded Grothendieck rings of Khovanov--Sazdanovic's categories through sums of representation categories over crossed products of polynomial rings over a general field. To obtain these results, we introduce associated graded categories for Krull--Schmidt categories with filtrations as a categorification of the associated graded Grothendieck ring, and study field extensions and Galois descent for Krull--Schmidt categories.
\end{abstract}

\maketitle 

\section{Introduction}

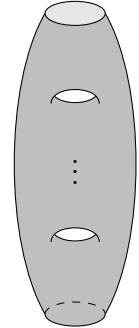
\begin{wrapfigure}{r}{0.25\textwidth}
  \begin{center}
  \vspace{-10pt}
  \begin{tikzpicture}[scale=.8]

\def\rx{.5}\def\ry{.2}
\coordinate (a) at (0,5);

\begin{scope}[looseness=.7]
\draw[fill=black!25] (a) to[bend right] (0,0)
  arc[start angle=-180, end angle=0, x radius=\rx, y radius=\ry]
  to[bend right] ($(a)+(2*\rx,0)$) -- cycle;
\draw[fill=black!10] (a)
  arc[start angle=-180, end angle=180, x radius=\rx, y radius=\ry];
\draw[ultra thin, dashed] (0,0) 
  arc[start angle=180, end angle=0, x radius=\rx, y radius=\ry];
\end{scope}

\newcommand\drawhole[2]{%
\def\x{#1}\def\y{#2}\def\w{.8}\def\h{0.25}\def\ang{90}
\def\arcA{(\x,\y) to[out=\ang,in=180-\ang] (\x+\w,\y)}
\def\arcB{(\x,\y+\h) to[out=-\ang,in=-180+\ang] (\x+\w,\y+\h) -- (\x+.5*\w,\y+3*\h) -- cycle}
\begin{scope} \clip \arcA; \draw[fill=white] \arcB; \end{scope}
\draw \arcA;
}

\drawhole{.1}{3.5}
\node at (.5,2.5) {$\vdots$};
\drawhole{.1}{1.2}

\end{tikzpicture}
  \end{center}
  \captionsetup{width=.8\linewidth}
  \caption{A cylinder with $n$ handles.}
  \label{fig:genus-g-cylinder}
    \vspace{10pt}
\end{wrapfigure} \par

In \cite{KS}, Khovanov--Sazdanovic introduced a class of symmetric monoidal categories $\DCob{\alpha}$ dependent on a rational series 
$$\alpha = \alpha(t)=p(t)/q(t) =\sum_{i\geq 0}\alpha_i t^i\in \kk[[t]]
\hspace{0.3\textwidth}$$ over a field $\kk$. Briefly, $\DCob{\alpha}$ is the quotient of the $\kk$-linear category of two-dimensional cobordisms by the monoidal ideal generated by requiring that a closed genus $i$ surface is evaluated to $\alpha_i$, the $i$-th coefficient of $\alpha(t)$, and that the polynomial recursion 
$$
u_\alpha(x)=0 \hspace{0.3\textwidth}
$$
holds. Here, $u_\alpha(t)$ is the polynomial $t^kq(t^{-1})$ with $k=\max\{\deg p(t)+1,\deg q(t)\}$ and the powers $x^n$ denote cylinders with $n$ handles (see \Cref{fig:genus-g-cylinder}).
As a special case, Deligne's \emph{interpolation categories} $\uRep(S_T)$ \cite{Del}, for $T\in \kk\setminus \{0\}$,  are recovered as $\DCob{\alpha}$ for $\alpha(t)$ being the series $T/(1-t)$.

\smallskip

 Subsequently, Khovanov--Kononov--Ostrik \cite{KKO} further investigated the categories $\DCob{\alpha}$. They proved that $\DCob{\alpha}$ decomposes as an external tensor product of such tensor categories based on a partial fraction decomposition of $\alpha$ and identified semisimple quotients of these categories. For $\alpha=c\in \kk\setminus \{0\}$ these are given by a category of representations over $\mathfrak{osp}(1|2)$ and for $\alpha=\beta_0+\beta_1t$ by certain categories of representations of orthosymplectic and orthogonal groups.

\smallskip

The monoidal categories $\DCob{\alpha}$ were motivated from universal constructions of \emph{topological theories} in \cite{Kho2}, inspired by \cite{BHMV}. Here, the question is posed whether given a series $\alpha$ as above one can find a two-dimensional topological theory, i.e.~a (lax) symmetric monoidal functor from two-cobordisms to an algebraic tensor category $\cC$. If the functor is strong monoidal, this recovers Atiyah's axiomatization of a \emph{topological quantum field theory}. The categories $\DCob{\alpha}$ provide such categories $\cC$ with finite-dimensional morphism spaces if $\alpha$ is rational. Moreover, they are in a certain sense universal among topological theories which assign $\alpha_i$ to the genus $i$ surface, see \Cref{thm:universal-prop}.

\smallskip

The main goals of the present paper are
\begin{enumerate}
    \item to construct all indecomposable objects of $\DCob{\alpha}$ and
    \item to describe the graded Grothendieck ring of $\DCob{\alpha}$.
\end{enumerate}
Similar to the case of Deligne's interpolation categories,
the additive Grothendieck ring $\Gr(\DCob{\alpha})$ comes with a natural filtration, which in turn gives rise to a graded ring $\gr\Gr(\DCob{\alpha})$.
The main idea how to describe $\gr\Gr( \DCob{\alpha} )$ is as follows: from $\DCob{\alpha}$, we construct a new category $\gr\DCob{\alpha}$, that we call its \emph{associated graded category}. We can think of $\gr\DCob{\alpha}$ as a categorification of $\gr\Gr(\DCob{\alpha})$, since we have
$\Gr(\gr\DCob{\alpha}) \cong \gr\Gr(\DCob{\alpha})$.
The upshot is that $\gr\DCob{\alpha}$ is much easier to analyze than the original category $\DCob{\alpha}$, but nevertheless provides enough information to deduce the structure of $\gr\Gr(\DCob{\alpha})$. 

For the purpose of this introduction, we state the main result in three increasingly general setups.

\begin{introtheorem} Let $\kk$ be a splitting field for $u_\alpha$ of characteristic zero. Then
$$\gr\Gr( \DCob{\alpha} ) \cong \bigotimes_{z \in Z} \Sym,$$
where $\Sym$ is the ring of symmetric functions and $Z$ is the set of distinct zeros of $u_\alpha$.
In particular, indecomposable objects in $ \DCob{\alpha}$ are parametrized by $Z$-indexed tuples of Young diagrams.
\end{introtheorem}

In order to generalize the above theorem to arbitrary characteristic $p$, we investigate characteristic $p$-analogues of the ring of symmetric functions in \Cref{sec:Symp}.
We pick a field $\kk$ of characteristic $p$ and define $\Sym^p$ to be the ring 
$$\Sym^p=\bigoplus_{n\geq 0} \Gr( \rfproj{\kk[S_n]} ),$$ 
with the product given by induction, which in fact does not depend on the choice of $\kk$. Here, $\rfproj{\kk[S_n]}$ denotes the category of finite-dimensional projective $\kk[S_n]$-modules. Computing products in $\Sym^p$ uses the embedding $\Sym^p\hookrightarrow \Sym$ which is determined by the decomposition matrices of symmetric groups in characteristic $p$, see \Cref{theorem:symp}, and thus relates to open questions in modular representation theory. A basis for $\Sym^p$ is labeled by $p$-regular Young diagrams and embeds non-trivially into $\Sym$. To encompass all characteristics, we set $\Sym^0=\Sym$.

\begin{introtheorem} Let $\kk$ be a splitting field for $u_\alpha$ of arbitrary characteristic $p$, possibly zero. Then
 $$ \gr\Gr( \DCob{\alpha} ) \cong \bigotimes_{z \in Z} \Sym^p ,
$$
where $Z$ is the set of distinct zeros of $u_\alpha$.
In particular, indecomposable objects in $ \DCob{\alpha}$ are parametrized by $Z$-indexed tuples of ($p$-regular) Young diagrams.
\end{introtheorem}

Finally, in order to generalize this result to the case where $u_\alpha$ does not split over the field $\kk$, we have to take invariants with respect to a Galois action.

\begin{introtheorem}[\Cref{theorem:main_classification}]
Let $\kk$ be any field of characteristic $p$ (possibly zero) and $\KK'$ a splitting field of $u_\alpha$. Then
$$ \gr\Gr( \DCob{\alpha} ) \cong \bigg(\bigotimes_{z \in Z} \Sym^p\bigg)^G,
$$
where $G=\Aut(\KK'|\kk)$ is the automorphism group of $\KK'$ fixing $\kk$ and $Z$ is the $G$-set of (distinct) zeros of $u_{\alpha}$ in $\KK'$.
In particular, indecomposable objects in $ \DCob{\alpha}$ are parametrized by orbit sums of $Z$-indexed tuples of ($p$-regular) Young diagrams.
\end{introtheorem}

To derive the main theorem \Cref{theorem:main_classification} we develop more general results of associated graded categories (\Cref{section:categorification_of_gr}) and field extensions and Galois descent  (Sections~\ref{sec::field-extensions}, \ref{sec:galoisdecent}) that can be applied to the much wider class of \emph{Krull--Schmidt categories}---i.e.~categories in which, in particular, direct sum decomposition are unique up to reordering the summands, see e.g.~\cite{Krause}.  

First, in \Cref{section:categorification_of_gr} we study Krull--Schmidt categories $\cC$ with filtrations and define the associated graded category $\gr \cC$. 
The category $\gr \cC$ has the same indecomposable objects as $\cC$ (\Cref{corollary:indec_for_gr}), inherits monoidal structures and braidings from $\cC$, and $\gr(\cC\boxtimes_\kk \cC')\simeq \gr \cC\boxtimes_\kk \gr \cC'$ (\Cref{cor:tensorfiltrations}). Taking the associated graded commutes with taking the additive Grothendieck ring in the following sense, see  \Cref{theorem:iso_graded_rings}:
$$\Gr( \gr(\cC )) \cong \gr( \Gr( \cC ) ).$$
As a special case of interest, we consider the case where $\cC$ has a tensor generator. Then there is a natural filtration on $\cC$ with respect to which $\gr\cC$ is a direct sum of categories of projective modules for a tower of algebras (\Cref{sec::tensor-generator}).

In another general section (\Cref{sec::field-extensions}), we study the scalar extension $\cC^\KK$ of a $\kk$-linear Krull--Schmidt category $\cC$ with respect to a field extension $\kk\subseteq \KK$. If $\cC$ has finite-dimensional morphism spaces, then sending $X\mapsto X^\KK$ induces a monomorphism 
$\Gr (\cC)\hookrightarrow \Gr(\cC^\KK)$
on additive Grothendieck rings. This monomorphism is an isomorphism if $\kk$ is a \emph{splitting field} for the category $\cC$ (see \Cref{def::splitting-field-for-C}). If $\kk$ is a splitting field for either $\cC$ or $\cD$, then we prove that 
$$\Gr (\cC\boxtimes_\kk \cD)\cong \Gr(\cC)\otimes_\mZ \Gr(\cD),$$
see \Cref{corollary:grothendieck_dec}. This follows by proving that indecomposable objects in $\cC\boxtimes_\kk \cD$ are given by $X\boxtimes Y$ for indecomposables $X$ of $\cC$, $Y$ of $\cD$ in this case (see \Cref{theorem:grothendieck_dec}). If $\cC$, $\cD$ are monoidal categories then the above isomorphism is one of rings.

In \Cref{sec:galoisdecent}, we study Krull--Schmidt categories with strict actions of a group $G$ (\Cref{sec:groupactions}) and Galois descent. Given a strict $G$-action on $\cC$, denote by $\cC^G$ the category of $G$-equivariant objects. In particular, the group  $G=\Aut(\KK|\kk)$ of a field extension $\kk\subseteq \KK$ naturally acts on $\cC^\KK$ and Galois descent (see e.g. \cite{Bourbaki}*{Chapter V.62,~Proposition 7}) extends to the following result: If $\kk \subseteq \KK$ is a finite Galois extension with Galois group $G$ and $\cC$ is a $\kk$-linear hom-finite Krull--Schmidt category, then by \Cref{theorem:galois_descent}, we have an equivalence
$$ \cC \xrightarrow{~\simeq~} ( \cC^\KK )^G .$$

With these tools at hand, the main theorem \Cref{theorem:main_classification} is proved by exhibiting an equivalence of monoidal categories
$$\gr \DCob{\alpha}\simeq \bigoplus_{n\geq 0} \rfproj{(P_n\rtimes\kk[S_n])},$$
where $P_n$ are certain commutative rings with $n$ generators. We explicitly construct the primitive idempotents in the crossed product rings on the right-hand side over a suitable splitting field $\KK'$. This implies an explicit classification of the indecomposable objects in $\DCob{\alpha}$ (see \Cref{prop::indecomposables}).
Finally, we upgrade these idempotents to equivariant objects representing orbit sums over $\Aut(\KK'|\kk)$ in order to apply the categorical Galois descent theorem above.

\medskip

Our results specify to particularly interesting rational series $\alpha$ relating to interpolation categories of classical representation categories, see \Cref{sec:grRepSt} and \Cref{sec:examples}. 
\begin{enumerate}
    \item As a special case, we extend the description of the associated graded of the Grothendieck ring of $\uRep_\kk(S_T)$ to arbitrary fields. Namely,
$$
\gr\Gr(\uRep(S_T)) \cong \Sym^p,
$$
see \Cref{cor::gr-RepSt}. This generalizes the result over an algebraically closed field $\kk$ of characteristic zero of \cite{Del}*{Proposition~5.11} for $T\notin\mZ_{\geq 0}$ (and \cite{CO1}*{Proposition~3.12}, \cite{Har2}*{Theorem~3.3} for general $T$).
    \item In the case $\alpha\in \kk[t]$, we derive that $\gr\Gr (\DCob{\alpha})\cong \Sym^p$, see \Cref{sec:polynomialcase}. The case $\alpha=c\in \kk^\times$, see \Cref{sec:constantcase}, leads to a family of categories interpolating the Karoubian tensor subcategory of $\Rep \mathfrak{osp}(1|2)$ generated by the defining representation. Note that in the case $\deg \alpha=1$, $\DCob{\alpha}$ interpolates certain representation categories of orthogonal or orthosymplectic groups by \cite{KKO}*{Section~7.1}.
    \item \Cref{sec:nonsplitting examples} contains  examples involving  irreducible polynomials of degree two in order to demonstrate the main theorem \Cref{theorem:main_classification} in case $u_\alpha$ does not split into linear factors over $\kk$.
    \item In \Cref{sec:inflations} we explain how to inflate the categories $\DCob{\alpha}$ by an additional polynomial factor in the relation $u_\alpha(t)=0$. This way, the families $\DCob{c}$ or $\DCob{T/(1-t)}$ can be extended at the special values $c=0$ or $T=0$ having the same partition bases for morphism spaces. For example, this way, $\uRep_\kk(S_0)$ can be interpreted in the general class of cobordism categories. 
\end{enumerate}

\subsection*{Acknowledgments} 
The authors like to thank M.~Khovanov and V.~Ostrik for helpful exchanges regarding the categories $\DCob{\alpha}$ studied in this paper, T.~Heidersdorf for helpful discussions about the representations of orthosymplectic supergroups, and the referee for helpful comments improving the exposition of the paper.
The research of J.~F. was supported by the Deutsche Forschungsgemeinschaft (DFG, German Research Foundation) -- Project-ID 286237555 -- TRR 195.
The research of R.~L. is supported by a Nottingham Research Fellowship.

\tableofcontents

\section{Preliminaries}

\subsection{Conventions and notations}
We always denote by $\kk$ a field and by $\K$ a commutative ring.
A category $\cC$ is called \emph{$K$-linear} if it is enriched over $\K$-modules.
Note that a \emph{pre-additive} category is simply a $\mZ$-linear category.
A $\K$-linear category $\cC$ is called \emph{additive} if it has all finite direct sums, and we denote by $\cC^\oplus$ the \emph{additive closure} of $\cC$.
If $\cC$ is the full subcategory of an additive $\K$-linear category $\cD$, then $\cC^\oplus$ can be identified with the full subcategory of $\cD$ spanned by all finite direct sums $(\bigoplus_{i=1}^n X_i) \in \cD$ with $X_i \in \cC$ and $n \geq 0$.

We denote by $\Kar(\cC)$ the \emph{Karoubian envelope} of a $\K$-linear category $\cC$, i.e., its idempotent completion.
Its objects are given by pairs $(X,e)$ with $X \in \cC$ and $e \in \End_{\cC}(C)$ an idempotent.
A morphism $(X,e) \rightarrow (X',e')$ between objects in $\Kar(\cC)$ is given by $\alpha \in \Hom_{\cC}( X, X' )$ such that $e' \circ \alpha \circ e = \alpha$. 
If $\cC$ is a full subcategory of an idempotent complete $\K$-linear category $\cD$, then $\Kar( \cC )$ can be identified with the full subcategory of $\cD$ spanned by all summands of objects in $\cC$.
For any $\K$-linear category $\cC$, the category $\Kar( \cC^{\oplus} )$ is always $\K$-linear, additive, and idempotent complete.

Given a family $\{\cC_i\}_{i\in I}$ of $\K$-linear categories for an index set $I$, the \emph{direct sum category} $\bigoplus_{i\in I}\cC_i$ is the full subcategory of the product category $\prod_{i \in I} \cC_i$ spanned by those families of objects $(C_i)_{i \in I}$ such that all but finitely many of the $C_i$ are isomorphic to zero (cf. \cite{EGNO}*{Section~1.3}).
Clearly, the direct sum is again $\K$-linear and inherits properties such as being additive, idempotent complete or Krull--Schmidt (see \Cref{sec::Krull-Schmidt}) from the $\cC_i$.

For a ring $R$, we denote by $\lmod{R}$ ($\rmod{R}$, respectively)
the category of left (right, respectively) $R$-modules.
We denote by $\rfproj{R}$ the category of finitely generated projective right $R$-modules. If $R$ is a finite-dimensional $\kk$-algebra, we denote by
$\lfmod{R}$ ($\rfmod{R}$, respectively)
the category of \emph{finite-dimensional} left (right, respectively) $R$-modules.
We denote the category of $\kk$-vector spaces by
$\lVec{\kk}$, and by $\lvec{\kk}$ the category of finite-dimensional $\kk$-vector spaces.

\begin{lemma}[{\cite{Krause}*{Proposition 2.3}}]\label{remark:KS_and_proj} 
For an object $X$ of an additive and idempotent
complete $\K$-linear category $\cC$, the functor
\[
Y \mapsto \Hom_{\cC}( X, Y )\colon \Kar( \{ X \}^{\oplus} ) \rightarrow \rfproj{(\End_{\cC}(X))} 
\]
is an equivalence of categories.
\end{lemma}

Let $\cC$ and $\cD$ be $\K$-linear categories.
We set $\cC \times_\K \cD$ as the  tensor product of $\K$-linear categories in the sense of \cite{Kelly}*{Section 1.4}, i.e., as the $\K$-linear category whose objects consist of pairs $(X,Y)$ with $X \in \cC$, $Y \in \cD$, and morphisms are given by
\[
\Hom_{\cC \times_{\K} \cD}( (X,Y), (X',Y') )
:=
\Hom_{\cC}(X,X') \otimes_{\K} \Hom_{\cD}(Y,Y').
\]
We also write $X \boxtimes Y$ for the object $(X,Y)$.
Note that even if $\cC$ and $\cD$ are additive and idempotent complete, $\cC \times_{\K} \cD$ is neither in general.
We set $\cC\boxtimes_\K \cD := \Kar((\cC \times_{\K} \cD)^{\oplus})$ and note that if $\cC$ and $\cD$ are both additive, then it can be verified that $\cC\boxtimes_\K \cD \simeq \Kar(\cC \times_{\K} \cD)$.
A \emph{monoidal structure} (i.e., a bilinear functor together with an associator and unit that satisfy the usual pentagon equation and triangle identity) given on both $\cC$ and $\cD$ induces a monoidal structure on $\cC\boxtimes_\K \cD$.

A $\kk$-linear category is called \emph{hom-finite} if all of its homomorphism spaces are finite-dimensional over $\kk$.
By a \emph{Karoubian tensor category} we mean an additive, idempotent complete, and hom-finite $\kk$-linear category which comes equipped with a rigid monoidal structure such that the tensor product is a $\kk$-bilinear functor and $\End(\one)=\kk$. A Karoubian tensor category that is, in addition, abelian is a \emph{tensor category} in the sense of \cite{EGNO}.

\subsection{Young tableaux and representations of symmetric groups} \label{sec::Young}

A \emph{Young diagram} $\l=(\l_1\geq \l_2\geq \ldots)$ is a sequence of integers $\l_i$, such that $\l_n=0$ for $n\gg0$. We denote its \emph{length} by $|\l|=\sum_{i\geq 0}\l_i$. The Young diagrams of length $n$ are in bijection with the irreducible representations of the symmetric group $S_n$ over a field $\kk$ of characteristic zero. 

To make this more precise, let $\l$ be a Young diagram of size $|\l|=n$. We label the boxes of the diagram with the numbers $1$ to $n$ (say, in ascending order along the rows of the diagram), and let $P$ and $Q$ be the subgroups of $S_n$ which preserve each row and each column, respectively. Then each element in $\sigma\in S_n$ can be written uniquely as a product $pq$ where $p\in P$ and $q\in Q$, and we consider $p'_\l:=\sum_{pq=\sigma\in S_n}(-1)^q \sigma$ in $\kk[S_n]$, where $(-1)^q$ is the sign of the permutation $q$. It can be shown that this element is a quasi-idempotent, i.e., there is a non-zero scalar multiple $p_\l$ of $p'_\l$ which is an idempotent. In fact, the idempotent is primitive, and the set $(p_\l)_\l$ is a complete list of idempotents in $\kk[S_n]$ up to conjugation, whose associated cyclic left ideals form a complete set of irreducible representations of $S_n$ over $\kk$ up to isomorphism.

To discuss the case of a field of positive characteristic $p$, we say that a Young diagram is \emph{$p$-regular} if no part $\l_i$ is repeated consecutively $p$ or more times. The set of $p$-regular Young diagrams of size $n$ is in bijection with the indecomposable projective modules over $\kk[S_n]$ for $\cha \kk=p$,  see \cite{J78}*{Theorem 11.5}.

Hence, in any case, we have a bijection between ($p$-regular, if $\cha\kk>0$) Young diagrams $\l$ of size $n$ and conjugacy classes of primitive idempotents in the group algebra $\kk[S_n]$. Let us fix a system of representatives of these conjugacy classes for the rest of the paper, so let $(e_\l)_\l$ be a list of primitive idempotents, complete up to conjugation, in the group algebra $\kk[S_n]$ for $n\geq0$ labeled by the ($p$-regular) Young diagrams, and we choose $e_\l$ to be the Young symmetrizer $p_\l$ if $\cha\kk=0$.

Recall that $\kk$ is a \emph{splitting field} for a finite group $G$ if
every irreducible $\kk[G]$-module is absolutely irreducible, i.e., remains irreducible over the algebraic closure of $\kk$.

\begin{remark}\label{remark:sn_splitting_field}
  Every field is a splitting field
  of the symmetric groups $S_n$ for $n \geq 0$,
  see \cite{J78}*{Theorem 11.5}.
  In particular, if $\cha\kk = p$, then the idempotents $e_{\lambda} \in \kk[S_n]$ can be chosen with coefficients in the prime field $\mF_p$.
\end{remark}

For any $\kk$-linear idempotent complete symmetric monoidal category $\cC$ with braiding $c_{-,-}$ and any object $X\in\cC$, we have an algebra embedding $\iota_X\colon\kk[S_n]\hookrightarrow\End(X^{\o n})$ which sends the transposition $(i,i+1)\in S_n$ to the endomorphism $\id_X^{\o(i-1)}\o c_{X,X}\o\id_X^{\o(n-i-1)}$. This embedding produces an idempotent endomorphism $\iota_X(e_\l)$ for any Young diagram $\l$. Let us define a functor $F^\l\colon\cC\to\cC$ by $F^\l(X):=\im\iota_X(p_\l)$, where the image object is given by a universal epi-mono factorization. In characteristic $0$, where the idempotents used in the process are Young symmetrizers, the resulting functors are usually called \emph{Schur functors} (see \cite{FultonHarris}*{Sec.~6.1}, or \cite{BMT} for a more recent categorical construction of Schur functors).

\subsection{Deligne's interpolation categories} \label{sec::Delignes-categories} Let $K$ be a commutative ring in this section. Pierre Deligne constructed a family of Karoubian tensor categories depending on a parameter $T\in K$ which in a certain precise sense ``interpolate'' the representation categories of all symmetric groups. We review this construction, referring to \cites{Del, CO1} for a more detailed account.

Deligne's category $\RepSt=\uRep_K(S_T)$ has objects $[n]$ labeled by non-negative integers $n\geq0$. For $m,n\geq0$, the space of morphisms from $[m]$ to $[n]$ is defined as $K P_{m,n}$, the free $K$-module over the set $P_{m,n}$ of set partitions of the set $\{1,\dots,m,1',\dots,n'\}$ with $m+n$ elements. An element in $P_{m,n}$ can be represented by a string diagram with $m$ upper and $n$ lower points labeled $1,\dots,m$ and $1',\dots,n'$, respectively, with strings connecting some of the points, where the connected components give a partition of the $m+n$ points. Different string diagrams can represent the same partition, and all subsequent definitions and construction only depend on the set partition the string diagrams define. A typical string diagram representing a partition and, hence, a morphism in $\RepSt$ looks as follows:
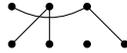
\[
\tp{1,3,14,0,11,2,12,0,13}
\]

To compute the composition of two partitions viewed as morphisms in $\RepSt$, we pick string diagrams representing them and stack those vertically, identifying the lower points of the first diagram with the upper points of the second diagram. Now we remove all connected components not containing upper or lower points, effectively reducing the total number of connected components by a number $\ell\geq0$. The composition is then given by the partition the remaining string diagram represents, multiplied by the factor $T^\ell\in K$. Similarly, the tensor product of two morphisms is given by horizontal stacking. Now $\RepSt$ is the additive closure of the Karoubian envelope of the category constructed so far; in particular, every object is a direct sum of pairs $([n],e)$, with $n\geq0$ and $e$ an idempotent in $\End([n])=KP_{n,n}$.

If $K=\kk$ is a field of characteristic zero, $\RepSt$ can be viewed as an interpolation category in the following sense: $\RepSt$ is a Karoubian tensor category for all $T\in K$; it is semisimple if and only if $T\in\mZ_{\geq0}$, while for $T=d\in\mZ_{\geq0}$, it has the group-theoretic category $\Rep(S_d)$ as its unique semisimple quotient, the \emph{semisimplification} in the sense of \cite{EO}.

Besides this interpolation property, $\RepSt$ can be understood as the universal Karoubian tensor category with a special commutative Frobenius algebra. Recall that a \emph{Frobenius algebra} in a monoidal category with tensor unit $\one$ is given by a tuple
$$
(A,\quad
\eta\colon\one\to A,\quad
\mu\colon A\o A\to A,\quad
\eps\colon A\to\one,\quad
\delta\colon A\to A\o A) ,
$$
such that $(A,\eta,\mu)$ defines an algebra, $(A,\eps,\delta)$ defines a coalgebra, and 
$$
(\id_A\o\mu)(\delta\o\id_A) = \delta\mu = (\mu\o\id_A)(\id_A\o\delta)
 .
$$
The Frobenius algebra is called \emph{special} if additionally $\mu\delta=\id_A$. In case the monoidal category we are working in has a symmetric braiding, the Frobenius algebra is called \emph{commutative} if the underlying algebra is commutative, that is $\mu c=\mu$, where $c\in\End(A\o A)$ is the symmetric braiding for the object $A$.

In $\RepSt$, whose tensor unit is $[0]$ and which has a unique symmetric braiding that is
$$
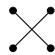
\tp{1,12,0,2,11}
$$
on $[1]\o[1]$, a special commutative Frobenius algebra is given by
$$
([1],\quad
\tp{11}\ ,\quad
\tp{1,11,2}\ ,\quad
\tp{1}\ ,\quad
\tp{11,1,12}) .
$$
In fact, $\RepSt$ is generated as a Karoubian tensor category by the object $[1]$ and the five morphisms defining the braiding and the Frobenius algebra structure, so the symmetric monoidal functors from $\RepSt$ to any other symmetric Karoubian tensor category are given exactly by the special commutative Frobenius algebras of dimension $T$ in that category \cite{Del}*{Proposition~8.3}.

As a subcategory of $\RepSt$, we obtain the \emph{Brauer category}, where instead of allowing all partitions $P_{m,n}$ we consider those consisting of blocks of size $2$, or equivalently, instead of using arbitrary string diagrams, we restrict our attention to those in which each point is connected to exactly one other point. The Karoubian tensor category obtained in this way is denoted by $\RepOt$. As a symmetric Karoubian tensor category, $\RepOt$ is generated by the object $[1]$ and the morphisms
$$
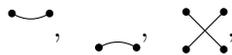
\tp{1,2},\quad\tp{11,12},\quad\tp{1,12,0,2,11} ,
$$
and the symmetric monoidal functors from $\RepOt$ to any other symmetric Karoubian tensor category are given exactly by objects $X$ of dimension $T$ with a symmetric self-duality, i.e.~a pair of morphisms $f\colon X\o X\to\one$, $g\colon\one\to X\o X$ satisfying $(g\o\id_X)(\id_X\o f)=\id_X=(\id_X\o g)(f\o\id_X)$ together with $f=f\circ c$ and $g=c\circ g$, where $c$ is the symmetric braiding on $X\o X$ \cite{Del}*{Proposition~9.4}.

\section{The Khovanov--Sazdanovic interpolation categories}

This section contains background material summarizing results from \cites{KS,KKO} and setting up notation for later use.

\subsection{Cobordism categories}\label{sec:cob}

Let $\Cob{2}$ denote the category of oriented two-dimensional cobordisms, bounding disjoint unions of the closed $1$-manifold $S^1$, up to isotopy. That is, the objects of $\Cob{2}$ are denoted by $[n]$ for non-negative integers $n$ representing a disjoint union of $n$ copies of $S^1$. A morphism $c\colon [m]\to [n]$ is given by an oriented two-dimensional cobordism with boundary given by a disjoint union of $n+m$ copies of $S^1$, with $m$ copies as incoming boundary components, and $n$ copies as outgoing boundary components. 

The datum of such a cobordism can be described combinatorially as a triple $(\pi_c,g_c,f_c)$, where $\pi_c$ is a partition of the set $\{1,\ldots, m,1',\ldots,n'\}$ together with a function 
$$g_c\colon \{1,\ldots, m,1',\ldots,n'\}\to \mN_0,$$
which is constant on the parts of $\pi_c$. The function $g_c$ encodes the genus of each connected component with non-empty boundary of the cobordism $c$. Given such a connected component $c'$ of $c$, the intersection $c'\cap \{1,\ldots, m\}$ denotes the incoming boundary circles, and $c'\cap \{1',\ldots, m'\}$ denotes the outgoing boundary circles. The third piece of data $f_c$ is a sequence of non-negative integers, where $(f_c)_i$ denotes the number of closed components of $c$ of genus $i$. Note that $f_c$ is eventually constantly zero. 

For example, a morphism $[5]\to [6]$ in $\Cob{2}$ is given by the cobordism in \Cref{cobordismexpl}.
\begin{figure}
$$
\vcenter{\hbox{\begin{tikzpicture}[
  tqft,
  every outgoing boundary component/.style={draw=black,fill=black!20},
  every incoming boundary component/.style={draw=black, fill=black!20},
  every lower boundary component/.style={draw},
  every upper boundary component/.style={draw},
  cobordism/.style={fill=black!25},
  cobordism edge/.style={draw},
  view from=incoming,
  cobordism height=2.5cm,
]
\begin{scope}[every node/.style={rotate=0}]
\pic[name=a,
  tqft,
  incoming boundary components=3,
  skip incoming boundary components=2,
  outgoing boundary components=0
  ];
\pic[name=b,
  tqft,
  incoming boundary components=3,
  skip incoming boundary components=2,
  outgoing boundary components=3,
  skip outgoing boundary components=2,
  offset=1,
  genus=2,
  anchor=incoming boundary 1,
  at=(a-incoming boundary 2)
];
\pic[name=c,
  tqft,
  incoming boundary components=2,
  outgoing boundary components=0,
  genus=1,
  anchor={(0,0)},
  at=(b-incoming boundary 3)
];
\pic[name=d,
  tqft,
  incoming boundary components=0,
  outgoing boundary components=2,
  at=(a-incoming boundary 1)
];
\pic[name=e,
  tqft,
  incoming boundary components=0,
  outgoing boundary components=1,
  at=(b-incoming boundary 3),
];
\draw node[label=$1'$] at (0,0) {};
\draw node[label=$1$] at (0,-3.5) {};
\draw node[label=$2'$] at (2,0) {};
\draw node[label=$2$] at (2,-3.5) {};
\draw node[label=$3'$] at (4,0) {};
\draw node[label=$3$] at (4,-3.5) {};
\draw node[label=$4'$] at (6,0) {};
\draw node[label=$4$] at (6,-3.5) {};
\draw node[label=$5'$] at (8,0) {};
\draw node[label=$5$] at (8,-3.5) {};
\draw node[label=$6'$] at (10,0) {};
\end{scope}
{
\draw[fill=black!25] (12,-1.5) ellipse (1cm and 0.8cm);
\def\x{11.5}\def\y{-1.6}\def\w{1}\def\h{0.3}\def\ang{80}
\def\arcA{(\x,\y) to[out=\ang,in=180-\ang] (\x+\w,\y)}
\def\arcB{(\x,\y+\h) to[out=-\ang,in=-180+\ang] (\x+\w,\y+\h) -- (\x+-.5\w,\y+3\h) -- cycle}
\begin{scope} \clip \arcA; \draw[fill=white] \arcB; \end{scope}
\draw \arcA;
}
\end{tikzpicture}}}\ .$$
\caption{An morphism $[5]\to [6]$ in $\Cob{2}$.}
\label{cobordismexpl}
\end{figure}
In this example, $\pi_c=\{\{1,2\},\{3,5,2',4'\},\{4\},\{1',3'\},\{5',6'\}\}$, setting
$$g_c(1)=0,\quad g_c(3)=2,\quad g_c(1')=0,\quad g_c(4)=0,\quad g_c(5')=1,$$
determines the function $g_c$, and $f_c=(0,1,0,\ldots)$ is constantly zero apart from the second entry counting the one closed component of genus $1$.

The composition $dc\colon [k]\to [m]$ of two cobordisms $c\colon [k]\to [l]$, $d\colon [l]\to [m]$ is given by connecting, in order, the outgoing circles of $c$ to the corresponding incoming circles of $d$.

\begin{definition}[$C_n^m$, $C_n^m(<k)$]
We denote the set of all cobordisms with $n$ incoming and $m$ outgoing circles by $C_n^m$. The \emph{weight} of such a cobordism is the maximal genus of a connected component with non-empty boundary, i.e., the maximum values of the function $g_c$. The subset of $C_n^m$ consisting of all cobordisms with maximal weight $k-1$ and no closed components, i.e., $f_c=0$, is denoted by $C_n^m(<k)$. For example, if we remove the single closed component from the cobordism in Equation~\eqref{cobordismexpl} we obtain an element of $C_5^6(<3)$.
\end{definition}

The category $\Cob{2}$ is a monoidal category where each object $[n]$ is self-dual. The tensor product is given by $[n]\otimes [m]=[n+m]$ and disjoint union of cobordisms.

We fix notation for some basic morphisms
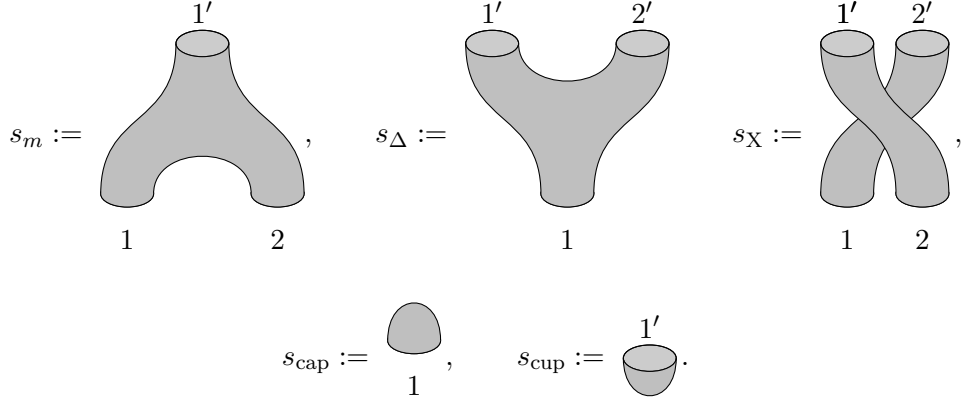
\begin{gather*}
\sm:=\vcenter{\hbox{
\begin{tikzpicture}[
  tqft,
  every outgoing boundary component/.style={draw=black,fill=black!20},
  every incoming boundary component/.style={draw=black, fill=black!20},
  every lower boundary component/.style={draw},
  every upper boundary component/.style={draw},
  cobordism/.style={fill=black!25},
  cobordism edge/.style={draw},
  view from=incoming,
  cobordism height=2cm,
]
\begin{scope}[every node/.style={rotate=0}]
\pic[tqft/pair of pants,
  at={(0,0)}
  ];
\draw node[label=$1'$](incoming boundary 1) at (0,0) {};
\draw node[label=$1$](outgoing boundary 1) at (-1,-3) {};
\draw node[label=$2$](outgoing boundary 1) at (1,-3) {};
\end{scope}
\end{tikzpicture}}}
, \qquad
\sdel:=\vcenter{\hbox{
\begin{tikzpicture}[
  tqft,
  every outgoing boundary component/.style={draw=black,fill=black!20},
  every incoming boundary component/.style={draw=black, fill=black!20},
  every lower boundary component/.style={draw},
  every upper boundary component/.style={draw},
  cobordism/.style={fill=black!25},
  cobordism edge/.style={draw},
  view from=incoming,
  cobordism height=2cm,
]
\begin{scope}[every node/.style={rotate=0}]
\pic[tqft/reverse pair of pants,
  at={(0,0)}
  ];
\draw node[label=$1'$](incoming boundary 1) at (0,0) {};
\draw node[label=$2'$](incoming boundary 1) at (2,0) {};
\draw node[label=$1$](outgoing boundary 1) at (1,-3) {};
\end{scope}
\end{tikzpicture}}}\, \qquad 
\sx:=\vcenter{\hbox{
\begin{tikzpicture}[
  tqft,
  every outgoing boundary component/.style={draw=black,fill=black!20},
  every incoming boundary component/.style={draw=black, fill=black!20},
  every lower boundary component/.style={draw},
  every upper boundary component/.style={draw},
  cobordism/.style={fill=black!25},
  cobordism edge/.style={draw},
  view from=incoming,
  cobordism height=2cm,
]
\begin{scope}[every node/.style={rotate=0}]
\pic[tqft/cylinder to prior,
  at={(1,0)}
  ];
\draw node[label=$1'$](incoming boundary 1) at (0,0) {};
\draw node[label=$2'$](incoming boundary 2) at (1,0) {};
\draw node[label=$1$](outgoing boundary 1) at (0,-3) {};
\draw node[label=$2$](outgoing boundary 2) at (1,-3) {};
\pic[tqft/cylinder to next,
  at={(0,0)}
  ];
\draw node[label=$1'$](incoming boundary 1) at (0,0) {};
\end{scope}
\end{tikzpicture}}}\, ,\\
\scap:=\vcenter{\hbox{
\begin{tikzpicture}[
  tqft,
  every outgoing boundary component/.style={draw=black,fill=black!20},
  every incoming boundary component/.style={draw=black, fill=black!20},
  every lower boundary component/.style={draw},
  every upper boundary component/.style={draw},
  cobordism/.style={fill=black!25},
  cobordism edge/.style={draw},
  view from=incoming,
  cobordism height=2cm,
]
\begin{scope}[every node/.style={rotate=0}]
\pic[tqft/cap,
  at={(0,0)}
  ];
\draw node[label=$1$](incoming boundary 1) at (0,-3) {};
\end{scope}
\end{tikzpicture}}}\,
, \qquad
\scup:=\vcenter{\hbox{
\begin{tikzpicture}[
  tqft,
  every outgoing boundary component/.style={draw=black,fill=black!20},
  every incoming boundary component/.style={draw=black, fill=black!20},
  every lower boundary component/.style={draw},
  every upper boundary component/.style={draw},
  cobordism/.style={fill=black!25},
  cobordism edge/.style={draw},
  view from=incoming,
  cobordism height=2cm,
]
\begin{scope}[every node/.style={rotate=0}]
\pic[tqft/cup,
  at={(0,0)}
  ];
\draw node[label=$1'$](incoming boundary 1) at (0,0) {};
\end{scope}
\end{tikzpicture}}}\, .
\end{gather*}
The following lemma is well known.

\begin{lemma}\label{generators}
The monoidal category $\Cob{2}$ is generated under composition and tensor product by the morphisms
$\sm,\sdel, \sx,\scup, \scap.$
\end{lemma}

For a field $\kk$, the \emph{linearization} $\kk\Cob{2}$ is the $\kk$-linear category that has the same objects as $\Cob{2}$ and whose morphisms are formal linear combinations of morphisms in $\Cob{2}$. Specifically, $\Hom_{\kk \Cob{2}}([n],[m])$ consists of formal linear combinations of elements from $D_n^m$. By construction, the morphisms from Lemma \ref{generators} generate $\kk \Cob{2}$ under tensor product, composition, and $\kk$-linear combinations.

\subsection{The monoidal categories \texorpdfstring{$\DCob{\alpha}$}{DCob(alpha)}}
\label{sect:DCob}

In \cite{KS}, Khovanov--Sazdanovic define quotients of $\kk\Cob{2}$, based on a sequence $\alpha=(\alpha_0,\alpha_1,\ldots)$ of elements in a field $\kk$ that generalize the interpolation categories introduced by Deligne. We briefly recall the construction here.

Assume from now on that the sequence $\alpha$ is given by a rational function
\begin{align}\label{eq:alphaconventions}
    \alpha(t)&=\sum_{i\geq 0}\alpha_i t^i=p(t)/q(t), & p(t),q(t)\in \kk[t] \text{ coprime}, \text{ and } q(0)=1.
\end{align}
In particular, we use the convention that $0=0/1$.
We write 
\begin{align*}
q(t)=1-q_1 t\pm\ldots +(-1)^m q_mt^m,\qquad q_i\in \kk, q_m\neq 0,
\end{align*}
and for 
$k=\max \{\deg p(t)+1,\deg q(t)\}$ denote
\begin{align}\label{eq:ualpha}
    u_\alpha(t)
    := t^k q(t^{-1})
    = t^k-q_1 t^{k-1}\pm\ldots+(-1)^m q_m t^{k-m}.
\end{align}
We use the convention that $\deg 0=-1$. Hence, $u_0(t)=1$ in the case $\alpha=0$. 

We further denote
\begin{align}
    x&:= \sm\sdel\in \End_{\Cob{2}}([1]),&
    s_i&:=\scap x^i \scup\in \End_{\Cob{2}}([0]), &\text{for $i\geq 0$,}
\end{align}
and observe that $s_i$ is a closed genus $i$ surface.

\begin{definition}\label{def:DCob}
The category $\SCob{\alpha}$ is defined as the quotient category of $\kk\Cob{2}$ by the ideal generated under $\kk$-linear combinations, two-sided composition, and tensor product, by the relations
\begin{align}
s_i&=\alpha_i \id_{[0]},\quad \forall i\geq0, & u_\alpha(x)&=0.
\end{align}
The category $\DCob{\alpha}$ is defined as the Karoubian envelope of $\SCob{\alpha}^\oplus$, i.e., the idempotent completion of the closure of $\SCob{\alpha}$ under finite direct sums.
\end{definition}

Recall that $u_{\alpha}(t)$ is a polynomial of degree $k$, thus, $u_{\alpha}(x) = 0$ implies that any cylinder with $k$ handles or more can be written as a linear combination of cylinders with strictly less than $k$ handles (see \Cref{example::special-DCob} below).

\begin{example}If $\alpha=0$, the object $[1]$ becomes isomorphic to the zero object $0$ of the additive category. Therefore, $\SCob{\alpha}$ (and hence $\DCob{\alpha}$) are equivalent to $\lvec{\kk}$.
\end{example}

\begin{theorem}[Khovanov--Sazdanovic, \cite{KS}*{Theorem~1}]\label{thm:universal-prop}
The category $\DCob{\alpha}$ is a symmetric monoidal category with finite-dimensional morphism spaces. A basis for $\Hom_{\DCob{\alpha}}([n],[m])$ is given by the morphisms corresponding to elements of the set $C_n^m(<k)$.
\end{theorem}

\begin{example} \label{example::special-DCob}
The following three examples are worth highlighting:
\begin{enumerate}
    \item[(i)] The case $\alpha=c\in \kk$, $\alpha \neq 0$, gives a category with semisimplification $\Rep^+ \mathfrak{osp}(1|2)$ \cite{KKO}*{Section~5}, where $\Rep^+ \mathfrak{osp}(1|2)$ denotes a certain subcategory of the category of representations of the Lie superalgebra $\mathfrak{osp}(1|2)$ (see \Cref{sect:Liesup}).
    Here, $u_{\alpha}(t) = t$, that is, any cobordism involving handles represents a zero morphism in $\DCob{c}$.
    \item[(ii)] In the case $\alpha=\beta/(1-\gamma t)$ with $\beta,\gamma \neq 0$, it follows that $\DCob{\alpha}\simeq\uRep S_{\beta\gamma}$ \cite{KKO}*{Section~6.1}. Note that as a consequence of this equivalence, we may assume $\gamma=1$.
    Here, $u_{\alpha}(t) = t-1$, so adding or removing handles does not change the morphism represented by any given cobordism in $\DCob{\alpha}$.
    \item[(iii)] The case $\alpha=\beta_0+\beta_1 t$, for $\beta_1\neq 0$, is related to the oriented Brauer category $\uRep \sfO_{\beta_1-2}$ \cite{KKO}*{Section~7.1}.
    Here, $u_{\alpha}(t) = t^2$, i.e., cobordisms featuring two or more handles are $0$.
\end{enumerate}
\end{example}

The categories $\DCob{\alpha}$ have the following universal property, derived from \cite{KKO}*{Section~2.3}. To state this property, for a commutative Frobenius algebra object $(A,\eta,\mu,\varepsilon,\Delta)$ in a symmetric monoidal category $\cS$, we denote $x:=\mu\Delta$ and $s_{n}:= \varepsilon x^n \eta$.

\begin{proposition}\label{prop:universal}
Let $\cS$ be a $\kk$-linear additive idempotent complete symmetric monoidal category. Then there is a bijection between $\kk$-linear symmetric monoidal functors 
$$F\colon \DCob{\alpha}\to \cS$$
and commutative Frobenius algebra objects $A$ in $\cS$ satisfying 
\begin{enumerate}
    \item[(i)] $u_\alpha(x)=0$,
    \item[(ii)] $s_n=\alpha_n$, for $n\leq \max \{\deg p(t)+1,\deg q(t)\}$.
\end{enumerate}
\end{proposition}

If $A$ satisfies (i)--(ii) then it satisfies $s_n=\alpha_n$ for all $n\geq 0$.
The correspondence is given by sending a symmetric monoidal functor $F$ to $A:=F([1])$, and the morphisms $\mu,\Delta,\eta,\varepsilon$ are the images of the morphisms $\sm,\sdel,\scup,\scap$ with the notation of \Cref{generators}.

\section{A categorification of the graded Grothendieck ring}\label{section:categorification_of_gr}

In examples of interest (e.g., for Deligne's interpolation categories of the symmetric groups), the Grothendieck ring of a category comes equipped with a natural filtration, and the associated graded ring turns out to be more manageable. In this section, we lift this idea to a categorical level: we introduce filtrations on Krull--Schmidt categories, define the associated graded category to such a filtration (\Cref{definition:ass_gr_cat}), and prove its compatibility with passing to Grothendieck groups (\Cref{theorem:iso_graded_rings}).

\subsection{Indecomposables in quotients of Krull--Schmidt categories} \label{sec::Krull-Schmidt}
Let $K$ denote a commutative ring.
We recall standard results on Krull--Schmidt categories
that turn out to be useful for the classification problem of indecomposable objects.
An additive $\K$-linear category $\cC$ is a \emph{Krull--Schmidt category} 
if every object $X \in \cC$ has a decomposition $X \cong \bigoplus_{i = 1}^n X_i$
with $\End_{\cC}(X_i)$ a local ring\footnote{
A ring $R$ is local if $1 \neq 0$ in $R$ and the sum of two non-units is a non-unit.} for $n \geq 0$, $i = 1, \dots, n$.
This is equivalent to $\cC$ being idempotent complete and $\End_{\cC}(X)$ being 
semi-perfect\footnote{A ring $R$ is semi-perfect if $R \cong \bigoplus_{i = 1}^n M_i$ as right modules with $\End_R(M_i)$ local rings for $n \geq 0$, $i = 1, \dots, n$.} for all $X \in \cC$ (\cite{Krause}*{Corollary 4.4}).

Let $\cC$ be a (skeletally small\footnote{A category is \emph{skeletally small} if the isomorphism classes of objects form a set. This assumption ensures that the Grothendieck group has an underlying set of elements rather than a proper class. Whenever we will speak about Grothendieck groups, we will tacitly assume the underlying category to be skeletally small.}) Krull--Schmidt category.
Clearly, an object $X \in \cC$ is indecomposable if and only if $\End_{\cC}(X)$ is local.
Moreover, let $\Gr( \cC )$ denote the
\emph{additive Grothendieck group} of $\cC$, i.e., the group spanned by
symbols $[X]$ for $X \in \cC$ subject to the relations
$[X] + [Y] = [Z]$ whenever $Z\cong X\oplus Y$.
We denote the set of isomorphism classes of indecomposable objects in $\cC$ by $\Indec( \cC )$.
Then $\Gr( \cC )$ is freely generated by the objects in $\Indec( \cC )$,
i.e., $\Gr( \cC ) \cong \mZ^{\oplus \Indec{\cC}}$,
which follows from the fact that each object in $\cC$ decomposes into a finite direct sum of indecomposables in an essentially unique way, that is, unique up to permutation of summands (\cite{Krause}*{Theorem 4.2}).

\begin{lemma}\label{lem::Krull-Schmidt} 
Let $\cC$ be a $\kk$-linear hom-finite category for a field $\kk$.
Then $\cC$ is Krull--Schmidt if and only if
it is additive and idempotent complete.
\end{lemma}
\begin{proof}
  The endomorphism rings are finite-dimensional $\kk$-algebras and thus semi-perfect.
\end{proof}

Recall that an ideal $\cI$ of a $\K$-linear category $\cC$ is given by a family of $\K$-modules $\cI( X, Y ) \subseteq \Hom_{\cC}( X, Y )$ for $X,Y \in \cC$ which is closed under composition from the left and right.
We denote the corresponding $\K$-linear quotient category by $\cC/\cI$.

\begin{remark}\label{remark:equality_of_ideals}
  There are two useful tests for equality of ideals $\cI, \cJ$ whenever $\cC$ is additive:
  \begin{enumerate}
    \item We have $\cI = \cJ$ if and only if $\cI( X, X ) = \cJ( X, X )$ for all $X \in \cC$.
    \item Let $\cS \subseteq \cC$ be a class of objects such that $\{ \cS \}^{\oplus} = \cC$. Then $\cI = \cJ$ if and only if $\cI( X, Y ) = \cJ( X, Y )$ for all $X, Y \in \cS$.
  \end{enumerate}
  Both tests follow from the isomorphism $\cI( X \oplus Y, Z \oplus W ) \cong \begin{pmatrix}
    \cI(X,Z) & \cI(Y,Z)\\\cI(X,W) & \cI(Y,W)
    \end{pmatrix}$ for all $X,Y,Z,W \in \cC$.
\end{remark}

If $F\colon \cC \rightarrow \cD$ is a $\K$-linear functor between $\K$-linear categories, 
then there are two useful notions of its kernel.
First, we have the \emph{kernel ideal} $\morker{F} \subseteq \cC$ given by those morphisms in $\cC$ which are sent to zero via $F$.
By the homomorphism theorem, $F$ induces a faithful functor $\cC/\morker{F} \rightarrow \cD$.
Second, we have the \emph{kernel subcategory} $\objker{F} \subseteq \cC$ which we define as the full subcategory spanned by
$X \in \cC$ such that $F(X) \cong 0$, i.e., such that $\id_X \in \morker{F}$. Clearly, if $\cC$ is Krull--Schmidt, so is $\objker{F}$.

\begin{lemma}\label{lemma:ks_indec}
  Let $\cI$ be an ideal of a Krull--Schmidt category $\cC$, and let $\cC \xrightarrow{\pi} \cC/\cI$ denote the canonical quotient functor.
  Then:
  \begin{enumerate}
    \item $\cC/\cI$ is a Krull--Schmidt category.
    \item Let $X \in \cC$ be indecomposable such that $\pi(X) \ncong 0$.
    Then $\pi(X)$ is indecomposable and every indecomposable in $\cC/\cI$ arises in this way.

    \item Let $X,Y \in \cC$ be indecomposable such that $\pi( X ) \cong \pi( Y ) \ncong 0$. Then $X \cong Y$.
  \end{enumerate}
  In particular, $\pi$ induces a bijection
  \[
  \Indec({ \cC }) \cong \Indec({ \cC/ \cI }) \sqcup \Indec({ \objker{\pi} }).
    \]
\end{lemma}
\begin{proof}
  Let $X \in \cC$ be indecomposable.
  Then $\End_{\cC/\cI}(\pi(X)) = \End_{\cC}(X)/\cI(X,X)$ is either zero or local.
  It follows that a decomposition of an arbitrary non-zero object $X \in \cC/\cI$
  into summands having local endomorphism rings
  can be obtained by applying $\pi$ to such a decomposition of $X$ regarded as an object in $\cC$.
  Thus, $\cC/\cI$ is Krull--Schmidt, and all its indecomposables arise from images of indecomposables under $\pi$.
  Last, let $X,Y \in \cC$ be indecomposable such that $\pi( X ) \cong \pi( Y ) \ncong 0$.
  Then we have isomorphisms
  $\pi(\alpha): \pi( X ) \rightarrow \pi( Y )$
  and
  $\pi(\beta): \pi( Y ) \rightarrow \pi( X )$
  such that
  $\pi( \alpha \circ \beta ) \in \End_{\cC}( Y )/\cI( Y, Y )$
  and
  $\pi( \beta \circ \alpha ) \in \End_{\cC}( X )/\cI( X, X )$
  are units.
  Thus, $\alpha \circ \beta$ and $\beta \circ \alpha$ are units in
  $\End_{\cC}(X)$ and $\End_{\cC}(Y)$, respectively, since these endomorphism rings are local.
\end{proof}

Since in a Krull--Schmidt category, every object has an essentially unique decomposition as a direct sum of indecomposable objects, there are well-defined multiplicities for any indecomposable object in any object of such a category. 

\begin{corollary} \label{lem::indec-direct-summand} In the situation of the previous lemma, let $m$ be the multiplicity of an indecomposable object $X\in\cC$ in an object $Y\in\cC$. Then either $\pi(X)\cong 0$ or $\pi(X)$ is an indecomposable object with multiplicity $m$ in $\pi(Y)$.
\end{corollary}

Let $\cD$ be a full subcategory of a $\K$-linear category $\cC$.
We denote by $\cI_{\cD}$ the ideal spanned by $\id_X$ for $X \in \cD$.
We also set
$\cC/\cD := \cC/\cI_{\cD}$.
On objects $X, Y \in \cC$, the ideal $\cI_{\cD}$ can be described by
\[
  \cI_{\cD}( X, Y ) = \{ \alpha \in \Hom_{\cC}( X, Y ) \mid \text{ $\alpha$ factors through an object in $\cD^{\oplus} \subseteq \cC^{\oplus}$} \}.
\]
We always have $\cD \subseteq \objker{\pi: \cC \rightarrow \cC/{\cD} }$
as subcategories of $\cC$. Equality holds if and only if $\cD$ is closed under taking direct sums and summands within $\cC$.
This proves the following corollary.

\begin{corollary}\label{corollary:indec_obj_ideal}
  Let $\cD \subseteq \cC$ be a full Krull--Schmidt subcategory of a Krull--Schmidt category $\cC$.
  Then \[\Indec( \cC ) \cong \Indec( \cC/{\cD} ) \sqcup \Indec(\cD).\]
\end{corollary}

For any object $X$ in a Krull--Schmidt category $\cC$, let us denote by $\Indec_X(\cC)$ the isomorphism classes of those indecomposable object which are isomorphic to direct summands of $X$. Then with \Cref{lem::indec-direct-summand} we even obtain:
\begin{corollary} \label{cor::indec-x}
  Let $\cD \subseteq \cC$ be a full Krull--Schmidt subcategory of a Krull--Schmidt category $\cC$. Let $\pi\colon\cC\to\cC/\cD$ be the quotient functor and assume $X\in\cC$.
  Then \[\Indec_X( \cC ) \cong \Indec_{\pi(X)}( \cC/{\cD} ) \sqcup (\Indec(\cD)\cap\Indec_X(\cC)).\]
\end{corollary}

\subsection{The radical of a \texorpdfstring{$K$}{K}-linear category}

Let $K$ be a commutative ring.
The radical of a $K$-linear category $\cC$ is given by the family of $K$-submodules
\[
  \rad_\cC ( X, Y ) := \{ \alpha \in \Hom_{\cC}( X, Y ) \mid \text{ $\alpha \circ \beta \in \rad( \End_{\cC}( Y ))$ for all $\beta: Y \rightarrow X$} \}
\]
for $X, Y \in \cC$.

\begin{lemma}\label{lemma:characterize_rad}
  Let $\cC$ be a $K$-linear category.
  The family $\rad_\cC$ forms an ideal in $\cC$
  and we have $\rad_{\cC}(X,X) = \rad( \End_{\cC}(X) )$ for all $X \in \cC$.
  Moreover, if $\cC$ is additive, then $\rad_{\cC}$ can be characterized as the ideal spanned by
  the morphisms $\rad( \End_{\cC}( X ) )$ for all $X \in \cC$.
\end{lemma}
\begin{proof}
  This follows from the proof given in \cite{Krause}*{Proposition 2.9}, where only the stated characterization of $\rad_{\cC}$ needs the stronger assumption of $\cC$ being additive rather than being merely $K$-linear.
\end{proof}

In the context of the classification of indecomposables, the radical owes its significance to the following lemma.

\begin{lemma}\label{lemma:radical_KS}
  Let $\cI \subseteq \rad_\cC $ be an ideal of a Krull--Schmidt category $\cC$.
  Then \[\Indec( \cC ) \cong \Indec( \cC/\cI ).\]
\end{lemma}
\begin{proof}
  By \Cref{lemma:ks_indec} it suffices to show that $\objker{ \pi: \cC \rightarrow \cC/\rad_\cC  } = 0$,
  which is clear since $\rad_\cC ( X, X ) = \rad( \End_{\cC}( X ) )$ 
  and thus $\pi( \id_X ) \neq 0$
  for all indecomposables $X \in \cC$.
\end{proof}

The following is a direct consequence.
\begin{corollary}\label{cor:rad-quotient-K0}
Let $\cI\subseteq \rad_{\cC}$ be an ideal of a Krull--Schmidt category. Then the quotient functor 
$\cC\to \cC/\cI$ induces an isomorphism 
$$\Gr (\cC)\stackrel{\cong}{\longrightarrow} \Gr(\cC/\cI).$$
\end{corollary}

Next, we show compatibilities of the radical with additive closures and Karoubian envelopes.

\begin{lemma}[Radicals and additive closures]\label{lemma:radicals_and_additive_closures}
  Let $\cC$ be a $K$-linear category. Then
  \[
    \cC^{\oplus}/\rad_{ \cC^{\oplus} } \simeq (\cC/\rad_{\cC})^{\oplus}.
  \]
\end{lemma}
\begin{proof}
  Let $\pi:\cC \rightarrow \cC/\rad_{ \cC }$ be the canonical quotient functor.
  By the universality of additive closures,
  we obtain a full and essentially surjective functor $\pi^{\oplus}: \cC^{\oplus} \rightarrow (\cC/\rad_{\cC})^{\oplus}$.
  We have $\morker{ \pi^{\oplus} }( X, Y ) = \rad_{\cC}( X, Y ) = \rad_{ \cC^{\oplus} }( X, Y )$ for all $X, Y \in \cC$.
  From \Cref{remark:equality_of_ideals}, we get $\morker{\pi^{\oplus}} = \rad_{\cC^{\oplus}}$.
  Thus, $\pi^{\oplus}$ induces the desired equivalence.
\end{proof}

\begin{lemma}[Radicals and Karoubian envelopes]\label{lemma:radicals_and_karoubi_env}
  Let $\cC$ be a $K$-linear category. Then we have a fully faithful functor
  \[
    \Kar( \cC )/\rad_{\Kar( \cC )} \rightarrow \Kar(\cC/\rad_{\cC} )
  \]
  which is an equivalence if for all $X \in \cC$, idempotents in $\End_{\cC}( X )$ can be lifted modulo $\rad( \End_{\cC}( X ) )$.
\end{lemma}
\begin{proof}
  Let $\pi:\cC \rightarrow \cC/\rad( \cC )$ be the canonical quotient functor.
  By the universality of Karoubian envelopes,
  we obtain a full functor $\Kar(\pi): \Kar( \cC ) \rightarrow \Kar(\cC/\rad_{\cC})$.
  For $X \in \cC$ and $e \in \End_{\cC}( X )$ an idempotent,
  we have 
  \begin{align*}
    \morker{\Kar(\pi)}( (X,e), (X,e) ) &= \{ e\alpha e \mid \alpha \in \End_{\cC}( X ), e \alpha e \in \rad_{\cC}( X, X )\} \\
    &= e\rad( \End_{\cC}( X ) )e \\
    &= \rad( e\End_{\cC}( X )e ) \\
    &= \rad( \End_{\Kar( \cC )}( X, e ))= \rad_{\Kar( \cC )}( (X,e), (X,e) )
  \end{align*}
  by \cite{Lam}*{Theorem 21.10}.
  If $\cC$ is additive, then so is $\Kar( \cC )$, and we may apply \Cref{remark:equality_of_ideals} from which we get $\morker{\Kar( \pi )} = \rad_{\Kar( \cC )}$.
  Thus, in the case where $\cC$ is additive, $\Kar( \pi )$ induces the desired fully faithful functor.
  If $\cC$ is not additive, then we may apply our above argument to the additive closure $\cC^{\oplus}$ and obtain our desired result by a restriction of the resulting fully faithful functor.
  
  Last, if idempotents can be lifted modulo the radical,
  then any object in $\Kar(\cC/\rad_{\cC})$ is of the form $(\pi(X), \pi(e))$ for $X \in \cC$ and $e \in \End_{\cC}( X )$ an idempotent.
  This yields essential surjectivity.
\end{proof}

We conclude that taking the radical commutes with passing to the additive and idempotent completion in the $\kk$-linear hom-finite case.

\begin{corollary}\label{corollary:radicals_and_KS_homfinite}
  Let $\cC$ be a $\kk$-linear hom-finite category for a field $\kk$. Then
  \[
    \big(\Kar( \cC^{\oplus })/\rad_{\Kar( \cC^{\oplus })}\big) \simeq \Kar( ( \cC/\rad_{\cC})^{\oplus } ).
  \]
\end{corollary}
\begin{proof}
  Idempotents in finite dimensional $\kk$-algebras can always be lifted modulo the radical.
  Thus, the claim follows from \Cref{lemma:radicals_and_additive_closures} and \Cref{lemma:radicals_and_karoubi_env}.
\end{proof}

\subsection{Filtrations of Krull--Schmidt categories and the associated graded category} \label{sec::filtrations}

Let $\cC$ be a Krull--Schmidt category.
We call an ascending chain $\cD_{\bullet}$ of full Krull--Schmidt subcategories 
\[
  \cD_{-1} := \{0 \} \subseteq \cD_0 \subseteq \cD_1 \subseteq \cD_2 \subseteq \dots \subseteq \cC
\]
with $\cC = \bigcup_{i \geq 0} \cD_i$ an \emph{(exhaustive) filtration} of $\cC$.
Such a filtration induces an ascending filtration on the Grothendieck group
\[
0 \subseteq \Gr(\cD_0) \subseteq \Gr(\cD_1) \subseteq \Gr(\cD_2) \subseteq \dots \subseteq \Gr(\cC).
\]

\begin{lemma}\label{lemma:classify_via_filt}
  Let $\cD_{\bullet}$ be a filtration of a Krull--Schmidt category $\cC$.
  Then
  \[
    \Indec( \cC ) \cong \bigsqcup_{i \geq 0} \Indec( \cD_i/{\cD_{i-1}} ).
  \]
\end{lemma}
\begin{proof}
We have an ascending chain of subsets
\[
  \emptyset = \Indec( \cD_{-1} ) \subseteq \Indec( \cD_0 ) \subseteq \Indec( \cD_1 ) \subseteq \Indec( \cD_2 ) \subseteq \dots
\]
such that
\[
  \Indec( \cC ) = \bigcup_{i \geq 0 } \Indec( \cD_i ) =  \bigsqcup_{i \geq 0 } (\Indec( \cD_i ) \setminus \Indec( \cD_{i-1} )).
\]
Now, the claim follows from \Cref{corollary:indec_obj_ideal}:
\[
  \Indec( \cD_i ) \cong \Indec( \cD_{i-1} ) \sqcup \Indec( \cD_i/{\cD_{i-1}} )
\]
for all $i \geq 0$.
\end{proof}

\begin{definition}\label{definition:ass_gr_cat}
  Let $\cD_{\bullet}$ be a filtration of a Krull--Schmidt category $\cC$.
  Then we define the \emph{associated graded category} of $\cC$ as:
  \[
    \gr( \cC ) := \gr( \cC, \cD_{\bullet} ) := \bigoplus_{i \geq 0} \cD_i/{\cD_{i-1}}.
  \]
\end{definition}

We say that a \emph{grading} $\cC_\bullet=\{\cC_i\}_{i\geq 0}$ for $\cC$ is a sequence of full Krull--Schmidt subcategories $\cC_i\subseteq \cC$ satisfying $\cC=\bigoplus_{i\geq 0}\cC_i$. Given a grading, we have an associated filtration $\cD_\bullet$, with $\cD_i=\bigoplus_{0\leq j\leq i}\cC_i$, for $\cC$. Clearly, given a filtration on $\cC$, then $\cC_i:=\cD_i/\cD_{i-1}$ is a grading for the associated graded category $\gr(\cC)$.

\begin{corollary}\label{corollary:indec_for_gr}
  Let $\cD_{\bullet}$ be a filtration of a Krull--Schmidt category $\cC$. Then
  \[
    \Indec( \cC ) \cong \Indec( \gr( \cC ) ).
  \]
\end{corollary}

\begin{remark}
We regard $\Gr( \gr(\cC )) = \bigoplus_{i \geq 0} \Gr(\cD_i/{\cD_{i-1}})$
as a graded abelian group.
\end{remark}

Let $\cC$ be a Krull--Schmidt category (over $\K$)
equipped with
a $\K$-bilinear functor $\otimes: \cC \times \cC \rightarrow \cC$.
We call a filtration $\cD_{\bullet}$ of $\cC$
\emph{compatible}
if $\otimes$ restricts to functors
\begin{equation}\label{eq:tensorcomp}
    \cD_m \times \cD_{n} \rightarrow \cD_{m+n}
\end{equation}
for all $m,n \geq 0$.
In that case, we get well-defined bilinear functors
\[
(\cD_m/\cD_{m-1}) \times (\cD_{n}/\cD_{n-1}) \rightarrow (\cD_{m+n}/\cD_{{m+n-1}})
\]
that assemble to a bilinear functor
\[
\otimes_{\gr}: \gr( \cC ) \times \gr( \cC ) \rightarrow \gr( \cC ).
\]
A unit\slash associator\slash braiding on $(\cC, \otimes)$
turns $\Gr( \cC )$ into a unital\slash associative\slash commutative ring.
Moreover, a compatible filtration $\cD_{\bullet}$ turns $\Gr( \cC )$ into a filtered ring, 
whose associated graded ring we denote by
$\gr( \Gr( \cC ) )$.

\begin{remark}
Every unit\slash associator\slash(symmetric) braiding on $(\cC, \otimes)$ induces a unit\slash associator\slash(symmetric) braiding on $(\gr( \cC ), \otimes_{\gr})$ which turns $\gr( \Gr( \cC ) )$
into a graded unital\slash associative\slash commutative ring.
\end{remark}

We note that passing to the associated graded category is functorial in the following sense. Assume given a $\K$-linear functor $F\colon \cC\to \cC'$, where both $\cC$ and $\cC'$ have filtrations $\cD_\bullet$, respectively, $\cD_\bullet'$, which is \emph{compatible with filtrations} in the sense that 
$$F(\cD_i)\subseteq \cD_i', \qquad \text{for all $i\geq 0$}.$$
Similarly, if $\cC=\bigoplus_{i\geq 0}\cC_i$ and $\cC'=\bigoplus_{i\geq 0}\cC_i'$ are gradings, we say $F$ is \emph{compatible with gradings} if $F(\cC_i)\subseteq \cC_i'$ for all $i$.
If $F$ is compatible with filtrations, we obtain an induced $\K$-linear functor 
\begin{equation}\label{eq:functorialitygr}
    \gr F\colon \gr (\cC)\longrightarrow \gr (\cC').
\end{equation}
which is compatible with gradings.
This assignment is (strictly) functorial in the sense that $\gr GF=\gr G\gr F$ and, in fact, $2$-functorial with respect to natural transformations. We say that an equivalence of filtered categories $\cC$ and $\cC'$ is \emph{compatible with filtrations} if both functors defining the equivalence are compatible with filtrations. We will later apply the following straightforward lemma.

\begin{lemma}\label{lem:filt-comp-equiv} Let $\cC=\cD_\bullet$, $\cC'=\cD'_\bullet$ be Krull--Schmidt categories with filtrations  and $F\colon \cC\to \cC'$ be a functor.
\begin{enumerate}
    \item  The functor $F$ is part of an equivalence that is compatible with filtrations if and only if $F$ is part of an equivalence and for any object $X$ of $\cC$, 
    $$X\in \cD_i \quad \Longleftrightarrow \quad FX\in \cD'_i.$$
    \item     
    If $F$ is part of an equivalence compatible with filtrations, then $\gr F\colon \gr (\cC)\to \gr (\cC')$ is part of an equivalence (compatible with gradings).
\end{enumerate}
\end{lemma}

The next theorem reveals the main idea behind the associated graded category: it provides a categorification of the graded Grothendieck ring.

\begin{theorem}\label{theorem:iso_graded_rings}
We have an isomorphism of graded rings
$\Gr( \gr(\cC )) \cong \gr( \Gr( \cC ) )$.
\end{theorem}
\begin{proof}
The bijection provided in \Cref{corollary:indec_for_gr}
induces a group isomorphism
\[
\Gr( \gr(\cC )) \cong \mZ^{\Indec(\gr(\cC))}
\cong \mZ^{\Indec(\cC)} \cong \gr( \Gr( \cC ) ).
\]
To compare multiplications,
let $I_m := \Indec({\cD_m}) \setminus \Indec({\cD_{m-1}})$
denote the set of indecomposables in $\cC$ of degree $m$.
Let $X \in I_m$ and $X' \in I_n$.
Then
\[
X \otimes X'
\cong
\bigg(\bigoplus_{Y \in I_{m+n}} Y^{n_Y}\bigg) \oplus Z
\]
where $Z$ is a direct sum of indecomposables
of degree $< m + n$, and $n_Y \geq 0$.
It follows that
$[X] \cdot [X'] = (\sum_{Y \in I_{m+n}} n_Y[Y])$
in $\gr( \Gr( \cC ) )$.
It also follows that
\[
X \otimes_{\gr} X'
\cong
\bigoplus_{Y \in I_{m+n}} Y^{n_Y}
\]
and thus the multiplication in $\Gr( \gr(\cC ))$
coincides with the one in $\gr( \Gr( \cC ) )$.
\end{proof}

\subsection{Krull--Schmidt categories with a tensor generator} \label{sec::tensor-generator}

Let $\K$ be a commutative ring.
Any $\K$-bilinear functor $\otimes: \cC \times \cC \rightarrow \cC$ for $\cC$ a Krull--Schmidt category
can be interpreted in terms of external tensor products and induction of projective modules.

\begin{lemma}\label{lemma:tp_in_terms_of_proj}
Let $X$, $Y$
be objects of a $\K$-linear Krull--Schmidt category $\cC$
equipped with a bilinear functor $\otimes: \cC \times \cC \rightarrow \cC$.
We set 
$R := \End_{\cC}( X )$,
$S := \End_{\cC}( Y )$,
and $T := \End_{\cC}( X \otimes Y )$.
Then the diagram
\begin{center}
        \begin{tikzpicture}
            \coordinate (r) at (9,0);
            \coordinate (d) at (0,-1.5);
            \node (A) {$\cC \times \cC$};
            \node (B) at ($(A)+(r)$) {$\cC$};
            \node (C) at ($(A)+(d)$) {$\rfproj{R} \times \rfproj{S}$};
            \node (D) at ($(C)+(r)$) {$\rfproj{T}$};
            \draw[->,thick] (A) to node[above]{$\otimes$} (B);
            \draw[->,thick] (C) to node[below]{$(P,Q) \mapsto (P \boxtimes Q) \otimes_{(R \otimes_{\K} S)} T$} (D);
            \draw[->,thick] (C) to (A);
            \draw[->,thick] (D) to (B);
        \end{tikzpicture}
    \end{center}
commutes up to natural isomorphism,
where the vertical functors are given as described in \Cref{remark:KS_and_proj},
$P \boxtimes Q \in \rfproj{(R \otimes_{\K} S)}$ is the external tensor product of $P$ and $Q$,
and $T$ is considered as a left $(R \otimes_{\K} S)$-module via the map $R \otimes_{\K} S \rightarrow T$ that comes from the functor $\otimes: \cC \times \cC \rightarrow \cC$.
\end{lemma}
\begin{proof}
By the universality of additive closures and Karoubian envelopes, it suffices to show that the diagram commutes (up to natural isomorphism) for the pair $(R,S) \in \rfproj{R} \times \rfproj{S}$.
But this is clear since
\[
(R \boxtimes S) \otimes_{(R \otimes_{\K} S)}
T \cong T. \qedhere
\]
\end{proof}

\begin{definition} \label{def:tower} A \emph{tower} of $K$-algebras is a family of $K$-algebras $(R_n)_{n\geq0}$ together with algebra maps $(\phi_{m,n}\colon R_m\o_K R_n\to R_{m+n})_{m,n\geq0}$ such that $R_0=K$ and
\begin{equation} \label{eq::associativity}
\phi_{m_1+m_2,m_3}(\phi_{m_1,m_2}\o_K\id_{R_{m_3}})=
\phi_{m_1,m_2+m_3}(\id_{R_{m_1}}\o_K\phi_{m_2,m_3})
\end{equation}
for all $m_1,m_2,m_3\geq0$.
\end{definition}
Towers of algebras involving more conditions in the definition were studied in \cite{SY}*{Section~3.1}.

\begin{construction}\label{construction:graded_cat}
Suppose given a tower of $K$-algebras $(R_m)_m$ with maps $(\phi_{m,n})_{m,n}$. 
Regarding the $R_m$ as $\K$-linear categories with a single object whose endomorphism ring is $R_m$,
we can form the category $\cC := \bigoplus_{m \geq 0} {R_m}$, 
and the $\phi_{m,n}$ encode a bilinear functor $\otimes_{\cC}: \cC \times \cC \rightarrow \cC$.
By the universality of additive closures and Karoubian envelopes, we may extend $\otimes_{\cC}$
to a bilinear functor $\o_\cD$ on $\cD := (\bigoplus_{m \geq 0} \rfproj{R_m})$.
\Cref{lemma:tp_in_terms_of_proj} shows that this bifunctor on $\cD$ is given by 
\[
  (P,Q) \mapsto (P \boxtimes Q) \otimes_{(R_m \otimes_{\K} R_n)} R_{m+n}
\]
for $P \in \rfproj{R_i}$, $Q \in \rfproj{R_j}$. We set $\one_\cD:=K\in\rfproj{R_0}\subset\cD$.
\end{construction}

\begin{lemma} \label{lem::tensor-product-graded-cat} $(\cD,\o_\cD,\one_\cD)$ as in \Cref{construction:graded_cat} is a $K$-linear monoidal Krull--Schmidt category.
\end{lemma}

\begin{proof} As $\cD$ is constructed as an additive closure and Karoubian envelope, it suffices to verify that $(\cC,\o_\cC,\one_\cC)$ is a strict monoidal category, with $\cC=\bigoplus_{m\geq0} R_m$ and $\o_\cC$ as in \Cref{construction:graded_cat}, and with $\one_\cC$ the object with identity morphism $1\in R_0$. But this follows from \Cref{eq::associativity}.
\end{proof}

Now let us consider a $K$-linear monoidal Krull--Schmidt category $\cC$ whose tensor unit has endomorphism ring $K$. Then an object $X\in\cC$ is called a \emph{tensor generator} if $\cC=\Kar(\{X^{\o m}\}_{m\geq0}^\oplus)$. Note that if $\cC$ is generated by any finite number of objects under tensor products, direct sums, and direct summands, then their direct sum is a tensor generator in this sense.

The choice of a generator can be used to define the following filtration of $\cC$. We set
$$
\cD_m := \Kar(\{X^{\o i}\}_{0\leq i\leq m}^\oplus) ,
$$
which is automatically compatible with the monoidal structure in the sense of \eqref{eq:tensorcomp} in \Cref{sec::filtrations}, i.e., $\o(\cD_m\times\cD_n)\subset\cD_{m+n}$. Hence, the associated graded Krull--Schmidt category $\gr(\cC)$ is naturally monoidal with a tensor product preserving direct sums.

To describe its structure, let us define $K$-algebras
\begin{align*}
R_m &:= \End_{\gr(\cC)}(X^{\o m}) 
= \End_{\cD_m/\cD_{m-1}}(X^{\o m}) ,
\end{align*}
in particular, $R_0=K$. Now for any $m,n\geq0$, the tensor product induces an algebra homomorphism $\phi_{m,n}:R_m\o_K R_n\to R_{m+n}$. If we assume $\cC$ to be strictly monoidal, then associativity of the tensor product in $\cC$ ensures that $(R_m)_{m\geq0}$ together with $(\phi_{m,n})_{m,n}$ form a tower of algebras in the sense of \Cref{def:tower}. Thus, by \Cref{construction:graded_cat} and \Cref{lem::tensor-product-graded-cat}, the category
$$
\cD := \bigoplus_{m\geq0} \rfproj{R_m} 
$$
can be equipped with a tensor product given by induction along the algebra maps.

\begin{proposition} \label{prop:tensor-generator} For any $K$-linear (strict) monoidal Krull--Schmidt category $\cC$ with a tensor generator whose tensor unit has endomorphism ring $K$, we have an equivalence $\gr(\cC)\simeq\cD$ of monoidal Krull--Schmidt categories which is compatible with the gradings, and in particular, an isomorphism $\gr K_0(\cC)\cong K_0(\cD)$ of graded rings.
\end{proposition}

\begin{proof} The categories $\gr(\cC)$ and $\cD$ are equivalent Krull--Schmidt categories with gradings, since $\cD_m/\cD_{m-1}\simeq \rfproj{R_m}$ by \Cref{remark:KS_and_proj}. Then the equivalence $\gr(\cC)\simeq \cD$ is monoidal by \Cref{lemma:tp_in_terms_of_proj}. This implies the claimed isomorphism of Grothendieck rings with \Cref{theorem:iso_graded_rings}.
\end{proof}

In case the monoidal structure on $\cC$ is not strict, $\cC$ is equivalent to its strictification $\cC'$, which induces an equivalence $\gr(\cC)\simeq\gr(\cC')$, and we may apply \Cref{prop:tensor-generator} to $\cC'$. In particular, the tower of algebras relevant for the description of the associated graded is obtained from the strictified category.

\section{Field extensions for Krull--Schmidt categories}
\label{sec::field-extensions}

\subsection{Field extensions for categories}

Let $\kk \subseteq \KK$ be a field extension.
Scalar extension defines a monoidal functor
\[
(\KK \otimes_\kk -) : (\lVec{\kk},\otimes_\kk) \rightarrow (\lVec{\KK},\otimes_{\KK}).
\]
In particular, any $\kk$-linear category $\cC$
can be transformed via $(\KK \otimes_\kk -)$
into a $\KK$-linear category,
which we denote by $\KK \otimes_\kk \cC$.
Concretely, the objects in $\KK \otimes_\kk \cC$ are the same as in $\cC$,
and $\Hom_{ \KK \otimes_\kk \cC}( X, Y ) = \KK \otimes_{\kk}\Hom_{\cC}( X, Y )$
for objects $X, Y$. 

If $\cC$ is additive, so is $\KK \otimes_\kk \cC$.
But if $\cC$ is idempotent complete, the same
does not necessarily hold for $\KK \otimes_\kk \cC$.
Thus, we define $\cC^\KK := \Kar( \KK \otimes_\kk \cC )$.
We observe that $\cC^\KK \simeq \lvec{\KK}\boxtimes_{\kk}\cC$, where the right-hand side can be viewed as a $\KK$-linear category.
We denote the image of an object $X \in \cC$ under the canonical functor
$\cC \rightarrow \cC^\KK$ by $X^{\KK}$. 

\begin{remark}
  If $\cC$ is a $\kk$-linear hom-finite Krull--Schmidt category,
  then $\cC^\KK$
  is a $\KK$-linear hom-finite Krull--Schmidt category.
  If such a category $\cC$ comes equipped with
  a monoidal structure $\otimes: \cC \times \cC \rightarrow \cC$,
  then this naturally extends to a monoidal structure $\otimes^{\KK}: \cC^{\KK} \times \cC^{\KK} \rightarrow \cC^{\KK}$,
  and the canonical functor $\cC \rightarrow \cC^\KK$ becomes monoidal.
  In particular, we get a ring homomorphism
  \[
    \Gr( (-)^{\KK} ):\Gr( \cC ) \rightarrow \Gr( \cC^\KK ): [X] \mapsto [X^\KK].
  \]
  \end{remark}

\begin{example}\label{example:field_extension_sn}
  For any $\kk$-algebra $R$, we have $(\rfproj{R})^{\KK} \simeq \rfproj{(R \otimes_{\kk} \KK)}$.
\end{example}

We need the following standard lemma.

\begin{lemma}\label{lemma:tensor_prod_of_locals}
Let $A$, $B$ be finite-dimensional algebras over a field $\kk$. Then 
\[
\rad( A \otimes_\kk B ) \supseteq \rad(A) \otimes_{\kk} B + A \otimes_{\kk} \rad(B)
\]
with equality if $\kk$ is algebraically closed or $\dim A/\rad(A)=1$ or $\dim B/\rad(B)=1$.
In particular, if $A$ and $B$ are local and $\kk$ is algebraically closed, then
$A \otimes_\kk B$ is local. Or if $\dim A/\rad(A)=\dim B/\rad(B)=1$, then $A\o_\kk B$ is local and $\dim A\o_\kk B/\rad(A\o_\kk B)=1$.
\end{lemma}
\begin{proof}
Let $I$ denote the kernel of the canonical surjective morphism
\[\epsilon: A \otimes_\kk B \rightarrow {A}/{\rad(A)} \otimes_{\kk} {B}/{\rad(B)}.\]
Then clearly $I \supseteq \rad(A) \otimes_{\kk} B + A \otimes_{\kk} \rad(B)$ and we get equality by comparing dimensions.
Since $\rad(A)$ and $\rad(B)$
are nilpotent ideals, so is $I$:
\[
I^l = (\rad(A) \otimes_{\kk} B + A \otimes_{\kk} \rad(B))^l
=
\sum_{i+j = l} \rad(A)^i \otimes_{\kk} \rad(B)^j = 0
\]
for $l$ big enough, which proves the inclusion
$I \subseteq \rad( A \otimes_\kk B )$.
If $\kk$ is algebraically closed or if $\dim A/\rad(A)=1$ or $\dim B/\rad(B)=1$, then
${A}/{\rad(A)} \otimes_{\kk} {B}/{\rad(B)}$ is semisimple and hence $\rad( A \otimes_\kk B ) \subseteq I$,
proving equality.
It follows that
\[
(A \otimes_\kk B)/\rad( A \otimes_\kk B)
\cong {A}/{\rad(A)} \otimes_{\kk} {B}/{\rad(B)}
\cong \kk
\]
for both cases of the premise.
\end{proof}

\begin{lemma}\label{lemma:mono_Gro}
Let $\cC$ be a $\kk$-linear hom-finite Krull--Schmidt category.
If $X, Y \in \cC$ are non-isomorphic indecomposables, then every common (isomorphic) summand of $X^{\KK}, Y^{\KK} \in \cC^{\KK}$ is zero.
In particular, we have a monomorphism between Grothendieck groups:
\[
  \Gr( (-)^{\KK} )\colon \Gr( \cC ) \rightarrow \Gr( \cC^\KK ), \quad [X] \mapsto [X^\KK].
\]
\end{lemma}
\begin{proof}
  We remark that for two objects $X, Y$ in any Krull--Schmidt category,
  $\Hom( X, Y ) \subseteq \rad(X, Y)$ implies that
  every common summand of $X$ and $Y$ is zero.
  Since suppose given such a summand $S$, then we have
  morphisms $X \xrightarrow{\epsilon_X} S \xrightarrow{\iota_X} X$
  and $Y \xrightarrow{\epsilon_Y} S \xrightarrow{\iota_Y} Y$
  such that $\id_S = \epsilon_X \circ \iota_X = \epsilon_Y \circ \iota_Y$.
  But since $\iota_Y \circ \epsilon_X \in \rad( X, Y )$,
  we have that the idempotent 
  $\iota_X \circ \epsilon_X = (\iota_X \circ \epsilon_Y) \circ (\iota_Y \circ \epsilon_X)$ lies in $\rad( X, X )$.
  Since any idempotent in the radical must be $0$,
  it follows that $\iota_X \circ \epsilon_X = 0$ and hence $S = 0$.
  
  Now, let $X,Y \in \cC$ be non-isomorphic indecomposables.
  Then $\Hom_{\cC}( X, Y ) = \rad_{\cC}( X, Y )$, and
  \begin{align*}
    \Hom_{\cC^\KK}( X^\KK, Y^\KK ) = \KK \otimes_\kk \Hom_{\cC}( X, Y )
    = \KK \otimes_\kk \rad_{\cC}( X, Y ) \subseteq  \rad_{\cC^\KK}( X, Y )
  \end{align*}
  where the last inclusion follows from the first part of \Cref{lemma:tensor_prod_of_locals}.
\end{proof}

\begin{definition} \label{def::splitting-field-for-C} Let $\cC$ be a $\kk$-linear hom-finite Krull--Schmidt category, let $\kk\subseteq\KK$ be a field extension. We say $\KK$ is a \emph{splitting field} of $\cC$ if $\dim_{\KK} \End(X)/\rad(\End(X))=1$ for every indecomposable object $X\in\cC^{\KK}$.
\end{definition}

\begin{remark}\label{remark:alternative_char_of_splitting_field}
  In other words, $\kk \subseteq \KK$ is a splitting field if each indecomposable in $\cC^{\KK}/\rad_{\cC^{\KK}}$ has endomorphism ring equal to $\KK$.
\end{remark}

\begin{lemma} The algebraic closure $\KK$ of $\kk$ is a splitting field for any $\kk$-linear hom-finite Krull--Schmidt category.
\end{lemma}

\begin{proof} In such a category, $\End(X)/\rad(\End(X))$ is a finite field extension of $\KK$, for every indecomposable object $X$.
\end{proof}

Recall that a field $\KK$ is called \emph{splitting field} of a finite-dimensional $\kk$-algebra $A$ if $A^\KK/\rad(A^\KK)$ is a direct product of matrix algebras over $\KK$, with $A^\KK:=A\o_\kk \KK$ \cite{Lam}*{Chapter~3, Theorem~7.7}.

\begin{lemma} \label{lem::splitting-field-algebra} If $\KK$ is a splitting field for a finite-dimensional $\kk$-algebra $A$, then it is a splitting field for $\rfproj A$.
\end{lemma}

\begin{proof}  Set $R = A^\KK$.
  By \Cref{example:field_extension_sn}, we have $(\rfproj A)^{\KK} \simeq \rfproj R$.
  By \Cref{remark:alternative_char_of_splitting_field}, we need to prove that each indecomposable
  in $(\rfproj R/\rad_{\rfproj R})$ has $\KK$ as its endomorphism ring.
  By \Cref{remark:KS_and_proj} and \Cref{corollary:radicals_and_KS_homfinite},
  we have an equivalence of categories
\[
  (\rfproj R/\rad_{\rfproj R}) \simeq \rfproj (R/\rad( R )).
\]
As $R/\rad( R )$ is a direct product of matrix algebras over $\KK$, the endomorphism rings of indecomposable objects in $\rfproj (R/\rad( R ))$ are given by $\KK$.
\end{proof}

\begin{lemma} \label{prop::Grothendieck-ring-splitting-field}
Assume $\kk$ is a splitting field for a $\kk$-linear hom-finite Krull--Schmidt category $\cC$ and $\kk\subseteq \KK$ a field extension. Then the monomorphism $\Gr((-)^{\KK})\colon \Gr( \cC ) \rightarrow \Gr( \cC^\KK )$ of \Cref{lemma:mono_Gro} is an isomorphism.
\end{lemma}

\begin{proof} We prove surjectivity. 
Since any object in $\KK \otimes_\kk \cC$ is of the form $X^\KK$ for some object $X\in\cC$, every indecomposable object in $\cC^\KK = \Kar( \KK \otimes_\kk \cC )$ occurs as a direct summand of $X^\KK$ for some indecomposable $X\in\cC$.
Let $X$ be an indecomposable object in $\cC$. Then, as $\kk$ is a splitting field for $\cC$, the quotient $\End(X)/\rad(\End(X))$ is one-dimensional over $\kk$. So $\End(X^\KK)/\rad(\End(X^\KK))$ is one-dimensional over $\KK$, and in particular, $X^\KK$ is indecomposable in $\cC^\KK$.  This proves surjectivity.
\end{proof}

\subsection{Tensor product decomposition and the Grothendieck ring}

Let $\kk$ be a field.

\begin{remark}
If $\cC$ and $\cD$ are $\kk$-linear Krull--Schmidt categories, then $\cC \boxtimes_{\kk} \cD$ is not necessarily a Krull--Schmidt category since the tensor product of semi-perfect $\kk$-algebras is not necessarily semi-perfect
(e.g., $R \otimes_{\mQ} \mC$ for $R$ the localization of $\mQ[x]$ at the prime ideal spanned by $x^2 + 1 $).
However, if both $\cC$ and $\cD$ are in addition hom-finite, then so is $\cC \boxtimes_{\kk} \cD$ and thus it is a Krull--Schmidt category by \Cref{lem::Krull-Schmidt}.
\end{remark}

\begin{theorem}\label{theorem:grothendieck_dec}
Let $\cC$ and $\cD$ be
$\kk$-linear hom-finite Krull--Schmidt categories, and assume $\kk$ is a splitting field for $\cC$ or $\cD$.
Then the following holds:
\begin{enumerate}
    \item We have $\cC \boxtimes_\kk \cD \simeq (\cC \times_{\kk} \cD)^{\oplus}$, i.e.,
the additive closure of
$\cC \times_{\kk} \cD$ is already idempotent complete. Thus, it
is a hom-finite Krull--Schmidt category.
\item All indecomposables in $\cC \boxtimes_\kk \cD$
are given by
$X\boxtimes Y$ with $X \in \cC$, $Y \in \cD$
indecomposables.
\item If $X, X' \in \cC$ and $Y, Y' \in \cD$ are indecomposables, then $X\boxtimes Y \cong X'\boxtimes Y'$ in $\cC\boxtimes_\kk \cD$ implies $X \cong X'$ and $Y \cong Y'$.
\item If $\kk$ is a splitting field for both $\cC$ and $\cD$, then $\kk$ is a splitting field for $\cC\boxtimes_\kk \cD$.
\end{enumerate}
\end{theorem}
\begin{proof}
Assume $\kk$ is a splitting field for $\cC$ (the other case is handled similarly).
Let $X \in \cC$ and $Y \in \cD$
be indecomposable objects.
Since the categories $\cC$ and $\cD$
are Krull--Schmidt categories,
the endomorphism algebras of $X$ and $Y$
are local.
Since $\kk$ is a splitting field for $\cC$,
we have
$\End_{\cC}(X)/\rad(\End_{\cC}(X)) \cong \kk$.
By \Cref{lemma:tensor_prod_of_locals},
it follows that for the object
$X\boxtimes Y \in \cC \times_{\kk} \cD$, we have
$\End_{\cC \times_{\kk} \cD}(X\boxtimes Y)/\rad(\End_{\cC \times_{\kk} \cD}(X\boxtimes Y)) \cong \kk \otimes_{\kk} \End_{\cD}(Y)/\rad(\End_{\cD}(Y))$, and thus $X\boxtimes Y$ is indecomposable since its endomorphism algebra is local.
Moreover, every object in 
$(\cC \times_{\kk} \cD)^{\oplus}$
can be written as a direct sum of objects $X\boxtimes Y$, for pairs $(X,Y)$
whose entries are indecomposables.
Thus, $(\cC \times_{\kk} \cD)^{\oplus} = \Kar( (\cC \times_{\kk} \cD)^{\oplus} )$
is a hom-finite Krull--Schmidt category. This proves (1) and (2).

For (3), let $X' \in \cC$ and $Y' \in \cD$
be indecomposables such that $X \not\cong X'$. Then $\Hom_{\cC}( X, X') = \rad_{\cC}( X, X' )$, and it follows that
\begin{align*}
    \Hom_{\cC\boxtimes_\kk\cD}( X \boxtimes Y, X' \boxtimes Y' ) &= \Hom_{\cC}( X, X' ) \otimes \Hom_{\cD}( Y, Y' )  \\
    &= \rad_\cC( X, X' ) \otimes \Hom_{\cD}( Y, Y' ) \subseteq \rad_{\cC\boxtimes_\kk\cD}( X \boxtimes Y, X' \boxtimes Y' ),
\end{align*}
which implies $X \boxtimes Y \not\cong X' \boxtimes Y'$ in $\cC\boxtimes_\kk \cD$ (and similarly for $Y\not\cong Y'$).

Last, if $\kk$ is a splitting field for both $\cC$ and $\cD$,
then $\End_{\cC\boxtimes_\kk\cD}(X\boxtimes Y)/\rad(\End_{\cC\boxtimes_\kk\cD}(X\boxtimes Y)) \simeq \kk$ for all indecomposables $X \in \cC$, $Y \in \cD$. Thus, $\kk$ is a splitting field for $\cC\boxtimes_\kk \cD$.
\end{proof}

\begin{corollary}\label{corollary:grothendieck_dec}
Let $\cC, \cD$ be as in \Cref{theorem:grothendieck_dec}.
Then we have a group isomorphism
\[
\Gr( \cC \boxtimes_\kk \cD )
\cong
\Gr ( \cC ) \otimes_{\mZ} \Gr ( \cD ).
\]
which is a ring isomorphism if
$\cC$ and $\cD$ come equipped with monoidal structures.
\end{corollary}

\begin{remark}\label{rem::natural-filtration}
Given two hom-finite $\kk$-linear Krull--Schmidt categories $\cC=\bigcup_{i\geq 0}\cC_i$ and $\cD=\bigcup_{i\geq 0}\cD_i$ with filtrations, we obtain a filtration on $\cC\boxtimes_\kk \cD$ with the $i$-th filtration piece given by the full Krull--Schmidt subcategory 
\begin{align*}
    (\cC\boxtimes \cD)_i=\Kar(\{\cC_j\boxtimes_\kk \cD_k \mid j +k = i\}^{\oplus}).
\end{align*}
\end{remark}

\begin{corollary}\label{cor:tensorfiltrations}
 Let $\cC, \cD$ be hom-finite $\kk$-linear Krull--Schmidt categories. 
 \begin{enumerate}
     \item The above filtration on $\cC\boxtimes_\kk\cD$ induces an equivalence of categories with gradings 
     $$\gr(\cC\boxtimes_\kk\cD)\simeq\gr \cC\boxtimes_\kk \gr \cD.$$
     \item  In particular, if $\kk$ is a splitting field for $\cC$ or $\cD$, then the group isomorphism from \Cref{corollary:grothendieck_dec} induces a group isomorphism
 $$\gr\Gr(\cC\boxtimes_\kk \cD)\cong \gr\Gr(\cC)\otimes_\mZ \gr\Gr(\cD).$$
     \item If both $\cC$ and $\cD$ have monoidal structures compatible with the filtration as in \eqref{eq:tensorcomp} in \Cref{sec::filtrations}, then the equivalence in \emph{(1)} is one of monoidal categories and the group isomorphism in \emph{(2)} is a ring isomorphism.
 \end{enumerate}
\end{corollary}
\begin{proof}
It suffices to prove the equivalence for each graded part, i.e., we need to show
\[
\cA_i := (\cC\boxtimes_\kk\cD)_i/(\cC\boxtimes_\kk\cD)_{i-1} \simeq  \bigoplus_{j+k=i} (\cC_j/\cC_{j-1}) \boxtimes_\kk (\cD_k/\cD_{k-1})
\]
for all $i \geq 0$.
Let $j,l,k,m \geq 0$ be indices such that $j+k = i = l + m$, and let $X \in \cC_j$, $Y \in \cD_k$, $Z \in \cC_l$, $W \in \cD_m$.
Let $\alpha \in \Hom_{\cC}( X, Z )$ and $\beta \in \Hom_{\cD}( Y, W )$.
Then $\alpha \otimes \beta = (\alpha \otimes \id) \circ (\id \otimes \beta) = (\id \otimes \beta) \circ (\alpha \otimes \id)$ in $(\cC\boxtimes_\kk\cD)_i$, which implies that $\alpha \otimes \beta$ factors over an object in $(\cC\boxtimes_\kk\cD)_{i-1}$ provided $j \neq l$, and thus we have $\Hom_{\cA_i}( X \boxtimes Y, Z \boxtimes W ) = 0$ in that case.
It follows that we may assume $j = l$ and consequently $k = m$.
Then we have
\begin{align*}
    \Hom_{\cA_i}( X \boxtimes Y, Z \boxtimes W ) &= \big(\Hom_{\cC}( X, Z ) \otimes_{\kk} \Hom_{\cD}( Y, W )\big)/\langle \alpha \otimes \beta \mid \text{$\alpha \in \cI_{\cC_{j-1}}$ or $\beta \in \cI_{\cD_{k-1}}$}\rangle_{\kk} \\
    &\cong \big( \Hom_{\cC}( X, Z )/\cI_{\cC_{j-1}}(X,Z) \big) \otimes_{\kk} \big( \Hom_{\cD}( Y, W )/\cI_{\cD_{k-1}}(Y,W) \big) \\
    &\cong \Hom_{\cC_j/\cC_{j-1}}( X, Z ) \otimes_{\kk} \Hom_{\cD_k/\cD_{k-1}}( Y, W ). 
\end{align*}
It follows that 
\[
\cA_i \supseteq \Kar( \{ X \boxtimes Y \mid X \in \cC_j, Y \in \cD_k \}^{\oplus} ) \simeq (\cC_j/\cC_{j-1}) \boxtimes_\kk (\cD_k/\cD_{k-1})
\]
which yields the claim. This equivalence is the identity assignment on objects. Hence, we note that if $\cC$ and $\cD$ have monoidal structures compatible with filtrations then
the above equivalence is compatible with the monoidal structure. This implies the first claim in Part (3).
To derive Part (2) and the Grothendieck ring statement in Part (3) from \Cref{corollary:grothendieck_dec} we use \Cref{theorem:iso_graded_rings}.
\end{proof}

\section{The graded Grothendieck ring of \texorpdfstring{$\RepSt$}{Rep(St)} over an arbitrary field}\label{sec:grRepSt}

As an application of the general theory presented in \Cref{section:categorification_of_gr} and \Cref{sec::field-extensions},
we discuss indecomposables in $\RepSt$
and its associated graded Grothendieck ring
over an arbitrary field $\kk$. We show that the associated graded Grothendieck ring is a ring of symmetric functions depending only on the characteristic of $\kk$.

\subsection{The associated graded category of \texorpdfstring{$\uRep (S_T)$}{Rep(S\_T)}}

Let $\kk$ be a field and $T \in \kk$. Deligne's interpolation category $\uRep_\kk(S_T)$ is a $\kk$-linear hom-finite Krull--Schmidt category with a tensor generator in the sense of \Cref{sec::tensor-generator} given by the object $[1]$. Hence, it  admits a filtration $\cD_{\bullet}$ defined by
\[
\cD_{n} := \Kar(  \big\{ [0],[1], \dots, [n] \big\}^{\oplus} ) \subseteq \uRep_{\kk}(S_T)
\]
that is compatible with the tensor product $\otimes$ on $\uRep_{\kk}(S_T)$.

Consider the tower of algebras (in the sense of \Cref{def:tower}) given by the group algebras $(\kk[S_n])_{n\geq0}$ with the natural embeddings
\begin{equation} \label{eq::embedding-Sn}
    \kk[S_n]\otimes \kk[S_m]\hookrightarrow \kk[S_{n+m}],
\end{equation}
where $S_n$ acts on the first $n$ elements, and $S_m$ acts on the last $m$ elements of $\{1,\ldots, n+m\}$. Using \Cref{construction:graded_cat} we associate to it the category $\bigoplus_{n \geq 0} \rfproj{\kk[S_n]}$, a $\kk$-linear symmetric monoidal category with a grading, whose tensor product is given by induction.

\begin{corollary} \label{cor::gr-RepSt} There is an equivalence of $\kk$-linear symmetric monoidal categories compatible with gradings
\[
\gr( \uRep_{\kk}(S_T) )
\simeq
\bigoplus_{n \geq 0} \rfproj{\kk[S_n]} ,
\]
and hence an isomorphism of graded rings
\[
\gr \Gr( \uRep_{\kk}(S_T) ) 
\cong
\bigoplus_{n \geq 0} \Gr( \rfproj{\kk[S_n]} ),
\]
\end{corollary}

    \begin{proof} It is shown in greater generality in \Cref{prop::quotient-is-crossed-product} below that the embedding of $\kk S_n$ into $\End_{\uRep_\kk S_T}([n])$ which sends a permutation $\sigma$ to the partition with parts $\{i,\sigma(i)'\}$ for $1\leq i\leq n$ induces an algebra isomorphism $\End_{\cD_n/\cD_{n-1}}([n])\cong\kk S_n$ for each $n\geq0$. Using these isomorphisms, the tensor product in $\uRep_\kk S_T$ induces exactly the embeddings in \Cref{eq::embedding-Sn}. Hence, the assertion follows from \Cref{prop:tensor-generator}.
\end{proof}

Assume that $\kk\subseteq \KK$ is a field extension. Using \Cref{example:field_extension_sn}, we obtain an equivalence of monoidal categories
  \[
    \Big(\bigoplus_{n \geq 0} \rfproj{\kk[S_n]}\Big)^{\KK} \simeq \bigoplus_{n \geq 0} \rfproj{\KK[S_n]}.
  \]
\begin{corollary}
\label{theorem:sym_only_dep_on_p}
  The graded ring $\bigoplus_{n \geq 0} \Gr( \rfproj{\kk[S_n]} )$ %
  only depends on the characteristic of~$\kk$.
\end{corollary}
\begin{proof} This follows from \Cref{prop::Grothendieck-ring-splitting-field}, as any field is a splitting field for $\rfproj{\kk[S_n]}$ by \Cref{remark:sn_splitting_field} and \Cref{lem::splitting-field-algebra}. 
\end{proof}

\subsection{The ring of symmetric functions in the modular case}\label{sec:Symp}

\begin{definition}\label{def:Sym}
Let $\kk$ be a field of characteristic $p$
(possibly zero).
We denote the graded ring of \Cref{cor::gr-RepSt} by
\[
\Sym_{\bullet}^p :=\Sym^p := \bigoplus_{n \geq 0} \Gr( \rfproj{\kk[S_n]} ),
\]
since it only depends on $p$ by \Cref{theorem:sym_only_dep_on_p}.
\end{definition}

The ring $\Sym^p$ has a basis labeled by the symbols $[V]$ of indecomposable projective $\kk[S_n]$-modules $V$, $n\geq 0$. This basis is in bijection with $p$-regular Young diagrams $\l$, see \Cref{sec::Young}, while if $\cha \kk=0$ $p$-regularity is not required.

If $\cha \kk=0$, given $\l$ with $|\l|=n$, the  symbol $[V_\l]$ of the associated irreducible $\kk[S_n]$-module $V_\l$ can be identified with the Schur function $s_\l$ and 
$$[V_\l\otimes V_\mu]=s_\l s_\mu=\sum_\nu c_{\l,\mu}^\nu s_\nu=\sum_\nu c_{\l,\mu}^\nu [V_\nu],$$
where $ c_{\l,\mu}^\nu$ is the Littlewood--Richardson coefficient \cite{Mac}*{Sections~I.7,~I.9}.

We set $\Sym := \Sym^0$. Our goal of this subsection is to prove the following description of $\Sym^p$, for $p$ a prime, as a subring of $\Sym$.

\begin{theorem}\label{theorem:symp}
  We have an injective homomorphism of graded rings
  \[
    \Sym^p \hookrightarrow \Sym
  \]
  that splits as a homomorphism of abelian groups and
  whose degree $n \geq 0$ part is given by
  \[
    \Sym^p_n \hookrightarrow \Sym_n: [V_{\l}] \mapsto \sum_{\mu}e_{\l, \mu} \cdot [V_{\mu}]
  \]
  where $\l$ ($\mu$, respectively) ranges over $p$-regular (all, respectively) Young diagrams of length $n$,
  and $(e_{\l, \mu})_{\l, \mu}$ is the transposed 
  decomposition matrix\footnote{For the definition of the decomposition matrix of a group $G$ (with respect to a $p$-modular system), see \cite{Serre}*{Chapter 15.2}} of $S_n$.
\end{theorem}

For the proof, we collect some useful facts from modular representation theory.
Let $G$ be a finite group and let $(\KK, R, \kk)$ be a $p$-modular system, i.e., $\kk$ a field of characteristic $p$,
$R$ a complete discrete valuation ring of characteristic $0$ such that $R/\rad( R ) \cong \kk$, and $\KK := \mathrm{Quot}( R )$ 
the field of fractions of $R$.

\begin{lemma}\label{lemma:modular_rep}
  The canonical ring maps $\kk \leftarrow R \rightarrow \KK$ induce functors
  \[
    \rfproj{\kk[G]} \xleftarrow{(\kk \otimes_R -)} \rfproj{R[G]} \xrightarrow{(\KK \otimes_R -)} \rfmod{\KK[G]}
  \]
  with the following properties:
  \begin{enumerate}
    \item The induced map on Grothendieck groups
    \[ \Gr( \rfproj{\kk[G]} ) \xleftarrow{\Gr(\kk \otimes_R -)} \Gr( \rfproj{R[G]} ) \]
    is an isomorphism.
    \item The induced map on Grothendieck groups
    \[ \Gr( \rfproj{R[G]} ) ) \xrightarrow{\Gr(\KK \otimes_R -)} \Gr( \rfmod{\KK[G]} ) \]
    is a split monomorphism.
    \item If $\KK$ and $\kk$ are splitting fields of $G$, then the map
    \[
      \Gr( \rfproj{\kk[G]} ) \xrightarrow{\Gr(\KK \otimes_R -) \circ \Gr(\kk \otimes_R -)^{-1}} \Gr( \rfmod{\KK[G]} )
    \]
    written as a matrix with respect to a basis of indecomposable objects in $\rfproj{\kk[G]}$ and with respect to simple objects in $\rfmod{K[G]}$
    is given by the transpose of the decomposition matrix of $G$.
  \end{enumerate}
\end{lemma}
\begin{proof}
  The first statement is \cite{Serre}*{Chapter 14, Corollary 3}.
  The second statement is \cite{Serre}*{Chapter 16, Theorem 34}.
  The third statement is given in \cite{Serre}*{Chapter 15.4.(c)}
  in the context of $\KK$ sufficiently large\footnote{This means that $\KK$ contains all $m$-th roots of unity for $m$ the l.c.m.\ of orders of elements in $G$.}.
  But note that in the case where $\kk$ and $\KK$ are both splitting fields of $G$,
  the matrix in question does not change for any extensions $\kk' \supseteq \kk$ and $\KK' \supseteq \KK$
  by \cite{Serre}*{Chapter 14.6}. Thus, the statement of \cite{Serre}*{Chapter 15.4.(c)}
  also holds in this case. 
\end{proof}

Note that \Cref{remark:sn_splitting_field} ensures that we can apply \Cref{lemma:modular_rep} in the context of symmetric groups.
Moreover, \Cref{lemma:modular_rep} gives us a split group monomorphism $\Sym^p \hookrightarrow \Sym$.
The next lemma clarifies its compatibility with the multiplication.

\begin{lemma}\label{lemma:scalar_ext_box}
  Let $S \rightarrow T$ be a morphism of commutative rings,
  and let $A \otimes_S B \rightarrow C$ be an $S$-algebra morphism for $S$-algebras
  $A, B, C$. Let $(-)^T := (T \otimes_S -)$ denote the scalar extension functor.
  Then we have an isomorphism of right $C^T$-modules
  \[
    \left( (M \boxtimes_S N){\otimes}_{(A \otimes_S B)} C \right)^T \cong
    (M^T \boxtimes_T N^T){\otimes}_{(A^T \otimes_T B^T)} C^T
  \]
  for all $M \in \rmod{A}$ and $N \in \rmod{B}$.
\end{lemma}
\begin{proof}
  We have an exact sequence
  \[
  M \otimes_S N \otimes_S A \otimes_S B \otimes_S C
  \xrightarrow{
  \delta
  }
  M \otimes_S N \otimes_S C
  \rightarrow
  (M \boxtimes_S N){\otimes}_{(A \otimes_S B)} C
  \rightarrow
  0
  \]
  where
  \[
  \delta( m \otimes n \otimes a \otimes b \otimes c)
  :=
  m \otimes n \otimes ( (a \otimes b)\cdot c )
  -
  (m\cdot a) \otimes (n\cdot b) \otimes c
  \]
  for $m \in M$, $n \in N$, $a \in A$, $b \in B$, $c \in C$.
  Now, the statement follows from the fact that
  extension by scalars
  $(-)^T: (\lmod{S}, \otimes_S) \rightarrow (\lmod{T},\otimes_T)$ is a monoidal functor
  (see, e.g., \cite{BourbakiI}*{Chapter II.5.1, Proposition 3}) and right exact.
\end{proof}

\begin{proof}[Proof of \Cref{theorem:symp}]
  By \Cref{lemma:modular_rep} the map $\Sym^p\hookrightarrow \Sym$ is a split group monomorphism.
  Moreover, if we equip
  \[ \Sym^R := \bigoplus_{n \geq 0} \Gr( \rfproj{R[S_n]} ) \]
  with the ring multiplication
\[
[P]\cdot[Q] := [(P \boxtimes Q) \otimes_{R[S_{n} \times S_{m}]} R[S_{n+m}] ]
\]
for $P \in \rfproj{R[S_n]}$ and $Q \in \rfproj{R[S_m]}$, then \Cref{lemma:modular_rep} and \Cref{lemma:scalar_ext_box}
imply that scalar extension $(R \otimes_{\kk} -)$ defines a graded ring isomorphism
\[
  \Sym^R \cong \Sym^p
\]
and scalar extension $(\KK \otimes_{R} -)$ defines an injective ring homomorphism
\[
  \Sym^R \hookrightarrow \Sym
\]
as desired.
\end{proof}

\begin{example}
  The determination of decomposition matrices of $S_n$ is a hard and (to the knowledge of the authors)
  open problem.
  From \cite{J78}*{Appendix}, we can write down the embedding $\Sym^2 \hookrightarrow \Sym$
  for low degrees (up to $4$)
  \begin{align*}
    \Grel{(0)} &\mapsto \Grel{(0)} &\Grel{(2,1)} &\mapsto \Grel{(2,1)} \\
    \Grel{(1)} &\mapsto \Grel{(1)} & \Grel{(4)} &\mapsto \Grel{(4)} + \Grel{(3,1)} + \Grel{(2,1^2)} + \Grel{(1^4)}\\
    \Grel{(2)} &\mapsto \Grel{(2)} + \Grel{(1^2)}  & \Grel{(3,1)} &\mapsto \Grel{(3,1)} + \Grel{(2^2)} + \Grel{(2,1^2)}\\
    \Grel{(3)} &\mapsto \Grel{(3)} + \Grel{(1^3)}
  \end{align*}
  and the embedding $\Sym^3 \hookrightarrow \Sym$
  for low degrees (up to $4$)
  \begin{align*}
    \Grel{(0)} &\mapsto \Grel{(0)} & \Grel{(4)} &\mapsto \Grel{(4)} + \Grel{(2^2)}  \\
    \Grel{(1)} &\mapsto \Grel{(1)} & \Grel{(3,1)} &\mapsto \Grel{(3,1)} \\
    \Grel{(2)} &\mapsto \Grel{(2)} & \Grel{(2^2)} &\mapsto \Grel{(2^2)} + \Grel{(1^4)}  \\
    \Grel{(1^2)} &\mapsto \Grel{(1^2)} &  \Grel{(2,1^2)} &\mapsto \Grel{(2,1^2)}\\
    \Grel{(3)} &\mapsto \Grel{(3)} + \Grel{(2,1)}.
  \end{align*}
  
These embeddings allow us to compute products in the modular rings of symmetric functions.\footnote{We used the online calculator  \cite{Gib} for the Littlewood--Richardson coefficients in $\cha \kk=0$.} In general $\Grel{(0)}=1$. For example, in $\Sym^2$, we compute all remaining products in $\Sym^2_i\cdot \Sym^2_j$, for $i+j\leq 4$, as
\begin{align*}
    \Grel{(1)}\cdot \Grel{(1)}&= \Grel{(2)} \\
    \Grel{(1)}\cdot \Grel{(2)}&= \Grel{(3)}+\Grel{(2,1)} \\
    \Grel{(2)}\cdot \Grel{(2)}&= \Grel{(4)}+2\Grel{(3,1)} \\
    \Grel{(1)}\cdot \Grel{(3)}&= \Grel{(4)} \\
    \Grel{(1)}\cdot \Grel{(2,1)}&= \Grel{(3,1)}.
\end{align*}
Similar, in $\Sym^3$, we compute
\begin{align*}
    \Grel{(1)}\cdot  \Grel{(1)}&=  \Grel{(2)}+ \Grel{(1,1)}\\
    \Grel{(1)}\cdot  \Grel{(2)}&=  \Grel{(3)}\\
    \Grel{(1)}\cdot  \Grel{(1,1)}&=  \Grel{(2,1)}\\
    \Grel{(2)}\cdot  \Grel{(2)}&=  \Grel{(4)}+ \Grel{(3,1)}\\
    \Grel{(2)}\cdot  \Grel{(1,1)}&=  \Grel{(3)}+ \Grel{(2,1^2)}\\
    \Grel{(1)}\cdot  \Grel{(3)}&=  \Grel{(4)}+ 2\Grel{(3,1)}+ \Grel{(2,1^2)}\\
    \Grel{(1)}\cdot  \Grel{(2,1)}&=  \Grel{(3,1)}+ \Grel{(2,2)}+ 2\Grel{(2,1^2)}.
\end{align*}
These products are different from the products in $\Sym$, where, for example,
\begin{align*}
    \Grel{(1)}\cdot \Grel{(1)}&= \Grel{(2)}+\Grel{(1,1)} \\
    \Grel{(1)}\cdot \Grel{(2)}&= \Grel{(3)}+\Grel{(2,1)} \\
    \Grel{(2)}\cdot \Grel{(2)}&= \Grel{(4)}+\Grel{(3,1)}+\Grel{(2,2)} \\
    \Grel{(1)}\cdot \Grel{(3)}&= \Grel{(4)}+\Grel{(3,1)} \\
    \Grel{(1)}\cdot \Grel{(2,1)}&= \Grel{(3,1)}+\Grel{(2,2)}+\Grel{(2,1,1)}.
\end{align*}
  
\end{example}

\section{Classification of indecomposable objects in \texorpdfstring{$\DCob{\alpha}$}{DCob(alpha)}}\label{sec:classification}

\subsection{The endomorphism rings of \texorpdfstring{$[n]$}{[n]}}
\label{sec::endo-of-n}

In order to classify indecomposable objects in the Karoubian tensor category $\DCob{\alpha}$, for $\alpha$ a rational function, introduced in \Cref{sect:DCob}, we study the endomorphism rings of the objects $[n]=[1]^{\otimes n}$ modulo the categorical ideal $\cI_{<n}$ of morphisms generated by $\id_{[m]}$ for objects $[m]$ with $m<n$. To this end, let $\kk$ be an arbitrary field.

Recall the generating morphisms $\sm,\sdel,\scap,\scup, \sx$ of $\DCob{\alpha}$ from Lemma \ref{generators}. We denote $x:=\sm\sdel$ and
\begin{align}
x_i:=\id_{[i-1]}\otimes x \otimes \id_{[n-i]}\colon [n]\longrightarrow[n], &&i=1, \ldots, n.
\end{align}
The $x_i$ generate a subalgebra of $\End_{\DCob{\alpha}}([n])$. Since they commute, satisfy $u_\alpha(x_i)=0$, and using \Cref{thm:universal-prop}, this subalgebra can be described by the following  quotient polynomial ring:
$$P_n := \kk[x_1,\ldots, x_{n}]/(u_\alpha(x_1), \ldots, u_\alpha(x_{n}))\cong P_1^{\otimes n}.$$
Here, we use the convention that $P_0=\kk$ for all $u_\alpha$.

Further, given a permutation $\s$ of $n$ points, we denote by $s_{\s}$ the isomorphism of $[n]$ that permutes $n$ strands according to $\s$. That is, $s_\s$ corresponds to the partition
$$\big\{\{1,\s(1)'\},\ldots, \{n,\s(n)'\}\big\} ,$$
where each connected component has genus zero (i.e., $g_{s_\s}=0$ and $f_{s_\s}=0$ in the combinatorial description of cobordisms from \Cref{sec:cob}). The assignment $\s\mapsto s_\s$ defines an embedding of $\kk$-algebras $\kk[S_n]\hookrightarrow \End_{\DCob{\alpha}}([n])$.

\begin{lemma}\label{lem:actiononPn}
The group algebra $\kk[S_{n}]$ admits a left action on $P_n$ by $\kk$-algebra automorphisms. The action is defined by
\begin{align*}
\s\cdot x_i &:= x_{\s(i)}, & \s\cdot(x y)&=(\s\cdot x)(\s \cdot y), &\text{for }x,y\in P_n.
\end{align*}
\end{lemma}
In other words, $P_n$ is  an $S_n$-equivariant $\kk$-algebra, cf. \Cref{sec::idempotents-crossed-product}.

\begin{definition}\label{def:crossedproduct}
Assume given a group $G$ and a left $G$-action on a $\kk$-algebra $A$ by $\kk$-algebra automorphisms.
Define the \emph{crossed product} $A \rtimes G$ (also denoted $A\rtimes \kk G$) to be the $\kk$-algebra defined on the $\kk$-vector space $A \otimes_\kk \kk[G]$ with the product
$(a \otimes g) \cdot (b \otimes h) = ag(b) \otimes gh$ for $a,b \in A$, $g,h \in G$.
\end{definition}
A left $A \rtimes G$-module is the same as a $G$-equivariant $A$-module,
i.e., an $A$-module $M$ with a left $\kk$-linear $G$-action such that
$g(am) = g(a)g(m)$ for all $a \in A$, $g \in G$.

\begin{lemma}\label{lem:subalg_crossed_prod}
    The subalgebra of $\End_{\DCob{\alpha}}([n])$ generated by the subalgebras $\kk[S_n]$ and $P_n$ is isomorphic to the crossed product algebra $P_n\rtimes \kk[S_n]$ with respect to the action of \Cref{lem:actiononPn}.
\end{lemma}
\begin{proof}
Given a permutation $\s\in S_{n}$, by definition of $\DCob{\alpha}$, the relation 
$$x_i s_\s=s_\s x_{\s (i)}$$
holds in $\End_{\DCob{\alpha}}([n])$. In fact, this relation comes from the construction of the cobordism category $\Cob2$. This is precisely the relation of the crossed product $P_n\rtimes \kk[S_{n}]$. This implies that the subalgebra of $\End_{\DCob{\alpha}}([n])$ generated by $\kk[S_n]$ and $P_n$ is isomorphic to the crossed product as claimed.
\end{proof}

We are now able to describe the endomorphism rings of $[n]$ in the quotient category by the ideal $\cI_{<n}$. %

\begin{proposition} \label{prop::quotient-is-crossed-product} For $\alpha = \frac{p}{q}$ a rational function as in \Cref{eq:alphaconventions} and $n\geq 0$ we have isomorphisms of $\kk$-algebras
\begin{align*}
\End_{\DCob{\alpha}/\cI_{<n}}([n])\cong P_n \rtimes \kk[S_{n}].
\end{align*}
\end{proposition}

\begin{proof} We recall that a basis for $\End_{\DCob{\alpha}}([n])$ is given by the cobordisms in $C_n^n(<k)$, with $k=\max\{\deg p+1,\deg q\}$, where $C_n^n(<k)$ is the set of those cobordisms $c$ with $n$ incoming and $n$ outgoing boundary circles, with no closed components, and of genus strictly less than $k$. Consider such a cobordism $c$ and let $r$ be the number of its components which have both incoming and outgoing boundary circles. Then $c$ factors through $[r]$, so $c$ lies in $\cI_{<n}$, unless $r=n$. As a consequence, $\cI_{<n}$ is spanned as a vector space by all cobordisms $c$ with $r<n$, while a complement is spanned by the cobordisms $c$ with $r=n$, which are exactly those cobordisms whose components connect exactly one incoming boundary circle with exactly one outgoing boundary circle. Hence, the composition 
$$P_n\rtimes \kk[S_n] \hookrightarrow \End_{\DCob{\alpha}}([n])\twoheadrightarrow \End_{\DCob{\alpha}/\cI_{<n}}([n])$$
is an isomorphism, using \Cref{lem:subalg_crossed_prod}.
\end{proof}

We remark that in the case when $\alpha=0$, $\End_{\DCob{0}}([n])=0$ for all $n\geq 1$. In this case, we have $P_1=\kk[x_1]/(1)=0$ and hence $P_n=0$ for all $n\geq1$, while always $P_0=\kk$.

We conclude this section with the additional observation that for all $n\geq 1$, $[n]$ is a direct summand of $[n+1]$ and if $\alpha\neq 0$, then $[0]$ is a direct summand of $[1]$.
To this end, we consider the morphisms 
\begin{align}
\iota_n\colon [n]\to [n+1],\quad  \iota_n&=\id_{[n-1]}\otimes \sdel,&&
\pi_n\colon [n+1]\to [n], \quad \pi_n=\id_{[n]}\otimes \scap,
\end{align}
defined for $n\geq 1$, in the category $\DCob{\alpha}$. Independent of choice of $\alpha$, we see that
\begin{align*}
\pi_n\iota_n&=\id_{[n]}.
\end{align*}
This proves the following observation.

\begin{lemma}\label{lem:en}
The morphism 
\begin{align}
e_n:=\iota_n\pi_n
\end{align}
is an idempotent endomorphism of $[n+1]$, which corresponds to the cobordism
\begin{align}
e_n=\vcenter{\hbox{
\begin{tikzpicture}[
  tqft,
  every outgoing boundary component/.style={draw=black,fill=black!20},
  outgoing boundary component 3/.style={fill=none,draw=black},
  every incoming boundary component/.style={draw=black, fill=black!20},
  every lower boundary component/.style={draw},
  every upper boundary component/.style={draw},
  cobordism/.style={fill=black!25},
  cobordism edge/.style={draw},
  view from=incoming,
  cobordism height=2cm,
]
\begin{scope}[every node/.style={rotate=0}]
\pic[name=a,
  tqft,
  incoming boundary components=1,
  outgoing boundary components=1,
  at={(0,0)}
  ];
\pic[name=b,
  tqft,
  incoming boundary components=1,
  outgoing boundary components=1,
  at={(2,0)}
  ];
\pic[name=c,
  tqft,
  incoming boundary components=1,
  outgoing boundary components=1,
  at={(4,0)}
  ];
\draw node[label=$1'$](incoming boundary 1) at (0,0) {};
\draw node[label=$1$](outgoing boundary 1) at (0,-3) {};
\draw node[label=$2'$](incoming boundary 2) at (2,0) {};
\draw node[label=$2$](outgoing boundary 2) at (2,-3) {};
\draw node[label=$3'$](incoming boundary 3) at (4,0) {};
\draw node[label=$3$](outgoing boundary 3) at (4,-3) {};
\end{scope}
\end{tikzpicture}}}
\quad\,\ldots
\vcenter{\hbox{
\begin{tikzpicture}[
  tqft,
  every outgoing boundary component/.style={draw=black,fill=black!20},
  outgoing boundary component 3/.style={fill=none,draw=black},
  every incoming boundary component/.style={draw=black, fill=black!20},
  every lower boundary component/.style={draw},
  every upper boundary component/.style={draw},
  cobordism/.style={fill=black!25},
  cobordism edge/.style={draw},
  view from=incoming,
  cobordism height=2cm,
]
\begin{scope}[every node/.style={rotate=0}]
\pic[name=b,
  tqft,
  incoming boundary components=1,
  outgoing boundary components=1,
  at={(4,0)}
  ];
\pic[name=e,
  tqft,
  incoming boundary components=2,
  outgoing boundary components=1,
  at={(6,0)}
];
\pic[name=f,
  tqft,
  incoming boundary components=0,
  outgoing boundary components=1,
  at={(8,0)}
];
\draw node[label=$(n-1)'$](incoming boundary 3) at (4,0) {};
\draw node[label=$n-1$](outgoing boundary 3) at (4,-3) {};
\draw node[label=$n'$](incoming boundary 4) at (6,0) {};
\draw node[label=$n$](outgoing boundary 4) at (6,-3) {};
\draw node[label=$(n+1)'$](incoming boundary 5) at (8,0) {};
\draw node[label=$n+1$](outgoing boundary 5) at (8,-3) {};
\end{scope}
\end{tikzpicture}}},
\end{align}
Hence, for $n\geq 1$, $[n]=([n+1],e_n)$ is a direct summand of $[n+1]$ in $\DCob{\alpha}$.
\end{lemma}

The following lemma completes this picture by viewing $[0]$ as a direct summand of $[1]$ for $\alpha\neq 0$.

\begin{lemma}\label{lemma:summands_of_one} If $\alpha\neq 0$, take $k$ minimal such that $\alpha_k\neq 0$ and consider the morphisms 
\begin{align}
\iota_0&:=\alpha_k^{-1}x^k\scup, & \pi_0&:= \scap.
\end{align}
Then $\pi_0\iota_0=\id_{[0]}$ and hence $e_0:=\iota_0\pi_0$ is an idempotent endomorphism of $[1]$. In particular, $[0]$ is a direct summand of $[1]$ in $\DCob{\alpha}$.
\end{lemma}

\subsection{The associated graded category of \texorpdfstring{$\DCob{\alpha}$}{DCob(alpha)}}
\label{sec:K0DCob}

In this subsection, $\kk$ is an arbitrary field.
We want to study the category $\DCob{\alpha}$ with the tools developed in \Cref{sec::Krull-Schmidt} and \Cref{sec::filtrations}. We first note that $\DCob{\alpha}$ is a Krull--Schmidt category by \Cref{lem::Krull-Schmidt}. By definition, the object $[1]\in\DCob{\alpha}$ is a tensor generator in the sense of \Cref{sec::tensor-generator}, so we can study $\DCob{\alpha}$ using a suitable filtration.

\begin{definition}\label{def::filtration-DCob} Let $\cD_n$ be the additive idempotent complete subcategory of $\DCob{\alpha}$ generated by the objects $[0],\dots,[n]$ for every $n\geq0$.
\end{definition}

Then $\cD_\bullet$ defines a filtration of $\DCob{\alpha}$ in the sense of \Cref{sec::filtrations}. In fact, since the object $[n-1]$ always occurs as a direct summand in the object $[n]$ as long as $\alpha\neq0$, as described in \Cref{sec::endo-of-n}, we can characterize $\cD_n$ alternatively as the additive idempotent complete subcategory generated by the object $[n]$ in this case.
The filtration $\cD_\bullet$ allows us to describe the graded Grothendieck ring of $\DCob{\alpha}$.
To this end, recall that $P_n = \kk[x_1,\ldots, x_{n}]/(u_\alpha(x_1), \ldots, u_\alpha(x_{n}))$.

\begin{definition}\label{definition:natural_alg_homs}
We define a family of $\kk$-algebra homomorphisms
\[
\mu_{ij}\colon
(P_i \rtimes \kk[S_{i}])
\otimes_{\kk}
(P_j \rtimes \kk[S_{j}])
\rightarrow
(P_{i+j} \rtimes \kk[S_{i+j}])
\]
using the natural embeddings 
\begin{align*}
\mathrm{emb}^P_l\colon P_i \hookrightarrow P_{i+j}:&~ p(x_1, \dots, x_i) \mapsto p(x_1, \dots, x_i),
\\
\mathrm{emb}^P_r\colon P_j \hookrightarrow P_{i+j}:&~ p(x_1, \dots, x_j) \mapsto p(x_{i + 1}, \dots, x_{i + j})
\end{align*}
as follows: let $p \in P_i$, $\sigma \in S_{i}$, $p'\in P_j$,
$\sigma' \in S_j$. Then
\begin{align*}
\mu_{ij}( p \otimes \sigma, p' \otimes \sigma' ) :=&
\big(\mathrm{emb}^P_l(p) \cdot \mathrm{emb}^P_r(p')\big) \otimes (\s,\s')
\\
=&
p(x_1,\dots,x_i)p'(x_{i+1},\dots,x_{i+j})\o (\s,\s')
,
\end{align*}
with the natural embedding $S_i\times S_j\to S_{i+j}$.
\end{definition}

\begin{definition}\label{definition:symu}
We define the $\kk$-linear Krull--Schmidt category (with a grading)
\[
\Symu{\alpha} := \bigoplus_{i \geq 0} \rfproj{(P_i \rtimes \kk[S_{i}])}.
\]
Moreover, we equip $\Symu{\alpha}$ with a tensor product
by applying \Cref{construction:graded_cat} to the natural $\kk$-algebra homomorphisms of \Cref{definition:natural_alg_homs}.
\end{definition}

\begin{theorem}\label{thm:KoDCob}
  We have an equivalence of monoidal categories with gradings
  \[
  \gr( \DCob{\alpha } ) \simeq \Symu{{\alpha}}.
  \]
  In particular, 
  $
  \gr( \Gr( \DCob{\alpha } ) ) \cong \Gr( \Symu{{\alpha}} )
  $
  as graded rings.
\end{theorem}

\begin{proof} By \Cref{prop::quotient-is-crossed-product},
  \begin{align*}
    \End_{\cD_n/\cD_{n-1}}([n]) \cong P_n\rtimes S_n
  \end{align*}
  for all $n \geq 0$, and these algebras form a tower with the algebra embeddings from \Cref{definition:natural_alg_homs} induced by the tensor product in $\DCob{\alpha}$. Hence, the assertion follows from  \Cref{prop:tensor-generator}.
\end{proof}

\begin{remark}
The idempotents $e_{n-1}$ constructed in \Cref{lem:en} for $n> 1$, and in \Cref{lemma:summands_of_one} for $\alpha\neq 0$ and $n=1$, imply that 
$$\End_{\DCob{\alpha}}([n])\cong e_{n-1} \End_{\DCob{\alpha}}([n]) e_{n-1} \oplus \End_{\DCob{\alpha}}([n])/(e_{n-1})\cong \End_{\DCob{\alpha}}([n-1])\oplus P_n\rtimes S_n.$$
\end{remark}

\subsection{Radicals of crossed products}
\label{sec::idempotents-crossed-product}

Let $A$ be a finite-dimensional $G$-equivariant
$\kk$-algebra for a finite group $G$ and a field $\kk$, i.e., $A$ is a $\kk$-algebra 
equipped with a left $G$-action
via $\kk$-algebra automorphisms.
Since its Jacobson radical $\rad(A)$ is preserved by all automorphisms, $G$ also acts on $B := \frac{A}{\rad(A)}$
via $\kk$-algebra automorphisms. Hence, we can form the crossed products
$R:=A\rtimes G$ and $B\rtimes G$, cf. \Cref{def:crossedproduct}.

\begin{lemma}\label{lemma:compute_rad}
Let $M$ be a finite-dimensional $R$-module.
Then $\rad(_{R}M)  \supseteq \rad(_{A}M)$, i.e.,
the radical of $M$ as an $A$-module
is contained in the radical of $M$
as an $R$-module.
\end{lemma}
\begin{proof}
First, we observe that $G$ acts on the set of $A$-submodules of $M$:
if $N \leq_{A} M$, then
$(1 \otimes g)N$ is an $A$-module,
since $(a \otimes 1)(1 \otimes g)N = (1 \otimes g)(g^{-1}(a) \otimes 1)N \subseteq (1 \otimes g)N$ for all $g \in G$, $a \in A$.
Moreover, this action respects inclusions, and thus it respects maximal submodules.

Let $U \lneq_{R} M$ be a maximal $R$-submodule.
Since $M$ is a finite-dimensional $\kk$-vector space,
there exists a maximal $A$-submodule $T$ such that
$U \leq_{A} T \lneq_{A} M$.
But then, as $(1\otimes g)U=U$ for all $g\in G$, we have
$U \leq_{A} \bigcap_{g \in G} (1 \otimes g)T \lneq_{A} M$,
and $\bigcap_{g \in G} (1 \otimes g)T$ is an $A$-submodule stable under the action of $G$. Thus, $U = \bigcap_{g \in G} (1 \otimes g)T$.
We compute
\[
\rad(_{R}M) = \bigcap_{\substack{U \lneq_R M\\U \text{max.}}} U
= \bigcap_{\substack{U \lneq_R M\\U \text{max.}}} 
\bigcap_{\substack{U \leq T \lneq_A M\\T \text{max.}}} T \supseteq \rad(_{A}M). \qedhere
\]
\end{proof}

The canonical surjective morphism of $\kk$-algebras
$A \rightarrow B$
induces a surjective morphism between the crossed product rings
$\epsilon:R \rightarrow (B \rtimes G)$.

\begin{corollary}\label{corollary:j_inclusion}
We have $\rad(R) \supseteq \ker( \epsilon )$.
In particular, primitive idempotents (up to conjugation) in $R$
are in bijection with
primitive idempotents (up to conjugation) in $B \rtimes G$.
\end{corollary}
\begin{proof}
If we consider $R$ merely as a left $A$-module, we have $R \cong \bigoplus_{g \in G} A$.
Using \Cref{lemma:compute_rad} and the additivity of
the radical, we compute
\[
\rad(R) \supseteq \rad(_{A}R) \cong \bigoplus_{g \in G} \rad(A)
\]
and thus obtain the description
\[
\rad(_{A}R) = \langle a \otimes g \mid a \in \rad(A), g \in G \rangle_\kk
\]
which equals $\ker( \epsilon )$.
\end{proof}

\begin{corollary}\label{corollary:family_mod_rad}
    The family $\mu_{ij}$ of $\kk$-algebra homomorphisms of \Cref{definition:natural_alg_homs} gives rise to a well-defined family of $\kk$-algebra homomorphisms
    \[
    \overline{\mu_{ij}}\colon
    \left( \frac{P_i}{\rad( P_i )} \rtimes \kk[S_{i}] \right)
    \otimes_{\kk}
    \left( \frac{P_j}{\rad( P_j )} \rtimes \kk[S_{j}] \right)
    \rightarrow
    \left( \frac{P_{i+j}}{\rad( P_{i+j} )} \rtimes \kk[S_{{i+j}}] \right)
    \]
    for $i,j \geq 0$.
    In particular, the ideal $\mathcal{I} \subseteq \Symu{\alpha}$ spanned by the endomorphisms $\rad( P_i ) \subseteq P_i \rtimes \kk[S_i] = \End_{\Symu{\alpha}}([i])$ for all $i \geq 0$ is a tensor ideal. Moreover, we have an inclusion $\mathcal{I} \subseteq \rad( \Symu{\alpha} )$.
\end{corollary}
\begin{proof}
In the notation of \Cref{definition:natural_alg_homs},
we have $\mathrm{emb}^P_l( \rad( P_i ) ) \subseteq \rad( P_{i+j} )$
and $\mathrm{emb}^P_r( \rad( P_j ) ) \subseteq \rad( P_{i+j} )$
for all $i,j \geq 0$ since homomorphisms between finite dimensional commutative $\kk$-algebras respect radicals.
Thus, we have 
\begin{align*}
    \mu_{ij}( \rad( P_i ) \otimes 1, 1 ) &= \mathrm{emb}^P_l( \rad( P_i ) ) \otimes 1 \\
    &\subseteq \rad( P_{i+j} ) \rtimes \kk[S_{i+j}]
\end{align*}
which gives the well-definedness of the family $\overline{\mu_{ij}}$ and the fact that $\mathcal{I}$ is a tensor ideal.
Last, by \Cref{corollary:j_inclusion}, we have
$\rad( P_{i+j} ) \rtimes \kk[S_{i+j}] \subseteq \rad( P_{i+j} \rtimes \kk[S_{i+j}])$ and thus the desired inclusion $\mathcal{I} \subseteq \rad( \Symu{\alpha} )$.
\end{proof}

\begin{corollary}\label{cor:removemulti}
Let $u_{\alpha}'$ be the product of all distinct irreducible factors of $u_{\alpha}$ (in particular, $u_{\alpha}'$ has no multiplicities). Then we have an isomorphism of graded rings
\[
\Gr( \Symu{\alpha} ) \cong \Gr( \cS_{u_{\alpha}'} )
\]
\end{corollary}
\begin{proof}
If we set $P_i = \kk[x_1,\ldots, x_{i}]/(u_\alpha(x_1), \ldots u_\alpha(x_{i}))$ and $P'_i = \kk[x_1,\ldots, x_{i}]/(u_\alpha'(x_1), \ldots u_\alpha'(x_{i}))$,
then 
\[
P_i/\rad( P_i ) \cong P_i'.
\]
Thus, if $\mathcal{I}$ denotes the monoidal ideal of \Cref{corollary:family_mod_rad}, then
\[
\Symu{\alpha}/\mathcal{I} \simeq \cS_{u_{\alpha}'}
\]
as monoidal categories, which gives an isomorphism of graded rings
\[
\Gr( \Symu{\alpha}/\mathcal{I} ) \simeq \Gr( \cS_{u_{\alpha}'} ).
\]
Since $\mathcal{I} \subseteq \rad( \Symu{\alpha} )$ by \Cref{corollary:family_mod_rad}, the claim now follows from \Cref{cor:rad-quotient-K0}.
\end{proof}

\subsection{Idempotents of crossed products}

Let $\Mg$ be a finite left $G$-set.
Then $G$ acts from the left on the $\kk$-algebra
$\kk[\Mg] \cong \bigoplus_{u \in \Mg} \kk$
whose multiplication is given componentwise.
We write $T_G(\Mg)$ for the translation groupoid of the left $G$-set $\Mg$,
i.e., it denotes the following category:
objects in $T_G(\Mg)$ are given by elements in $\Mg$, and $\Hom_{T_G(\Mg)}(u,v) = \{ g \in G \mid gu = v \}$ for $u,v \in \Mg$.
We denote the category of $\kk$-linear functors from $T_G(\Mg)$ to the category $\lvec{\kk}$ of finitely dimensional $\kk$-vector spaces by $\lfmod{T_G(\Mg)}$.

\begin{lemma}\label{lemma:equiv_of_cats}
There is an equivalence of $\kk$-linear categories
\[
\lfmod{(\kk[\Mg] \rtimes G)} \simeq \lfmod{T_G(\Mg)}
\]
where $\lfmod{(\kk[\Mg] \rtimes G)}$ denotes
the category of finite-dimensional $\kk[\Mg] \rtimes G$-modules.
\end{lemma}
\begin{proof}
Let $F: T_G(\Mg) \rightarrow \lvec{\kk}$ be a functor in $\lfmod{T_G(\Mg)}$.
Set $M := \bigoplus_{u \in \Mg}F(u)$.
Then $M$ can be regarded as a left $(\kk[\Mg] \rtimes G)$-module where $u \in \Mg$
acts as the identity on $F(u)$ and as zero on $F(v)$ for $v \neq u$,
and $g \in G$ acts via $F(g): F(u) \rightarrow F(gu)$.

Conversely, let $M$ be a finite-dimensional left $(\kk[\Mg] \rtimes G)$-module.
Then we define a functor $F: T_G(\Mg) \rightarrow \lvec{\kk}$
by setting $F(u) := (u \otimes 1)M$ for $u \in \Mg$,
and for $g: u \rightarrow v$ a morphism in $T_G(\Mg)$,
we set $F(g): (u \otimes 1)M \rightarrow (v \otimes 1)M: x \mapsto (1 \otimes g)x$.
This is well-defined since for $m \in M$,
we have $(1 \otimes g)(u \otimes 1)m = (v \otimes 1)(1 \otimes g)m$
for all $m \in M$.
These two constructions define mutually inverse $\kk$-linear functors.
\end{proof}

\begin{corollary}\label{corollary:equiv_of_cats} Let $\Mgr$ be a set of representatives
of the orbits of the $G$-set $\Mg$.
Then 
$$
\lfmod{(\kk[\Mg] \rtimes G)} \simeq \bigoplus_{v\in \Mgr} \Rep_\kk(G_v)
$$
where $G_v\subset G$ denotes the stabilizer of $v$.
\end{corollary}

We denote the set of isomorphism classes of
irreducible objects in $\Rep_{\kk}(G)$ by $\Xk{G}$.

\begin{corollary}\label{corollary:abstract_classification}
Let $A$ be a finite-dimensional and commutative $G$-equivariant $\kk$-algebra
for a field $\kk$.
Let $\Mg$ be the left $G$-set of 
central primitive idempotents
of $B := \frac{A}{\rad(A)}$,
and let $\Mgr$ be a set of representatives
of its orbits.
If $B = \kk[\Mg]$,
then the primitive idempotents in $R := A \rtimes G$
up to conjugation
are in (constructive) bijection with the set
\[
\bigsqcup_{v \in \Mgr} \Xk{G_v}
\]
where $G_v$ denotes the stabilizer of $v$.
\end{corollary}
\begin{proof}
By \Cref{corollary:j_inclusion},
the primitive idempotents (up to conjugation) in $R$
are in bijection with
the primitive idempotents (up to conjugation)
in $\kk[\Mg] \rtimes G$,
which in turn are in bijection with the indecomposable projectives in 
$\bigoplus_{v\in \Mgr} \Rep_\kk(G_v)$
(and thus with the irreducibles)
by \Cref{corollary:equiv_of_cats}.
\end{proof}

To make the construction of the idempotents in \Cref{corollary:abstract_classification} concrete, we first note that the orbit decomposition $\Omega=\bigsqcup_{v\in\Omega'} Gv$ yields a direct sum decomposition $\kk[\Omega]\rtimes G\cong \bigoplus_{v\in\Omega'} \kk[Gv]\rtimes G$, and that $\kk[G_v]\to v\o \kk[G_v]\subset \kk[Gv]\rtimes G, h\mapsto v\otimes h,$ is an algebra embedding. So we have
$$
\bigoplus_{v\in\Omega'} \kk[G_v]
\hookrightarrow \bigoplus_{v\in\Omega'} \kk[Gv]\rtimes G
\cong \kk[\Omega]\rtimes G
.
$$

\begin{lemma} For any $v\in\Omega'$, $v\o 1\in\kk[Gv]\rtimes G$ is a full idempotent realizing the Morita equivalence of $\kk[Gv]\rtimes G$ and its subalgebra $v\o\kk[G_v]\cong \kk[G_v]$.
\end{lemma}

\begin{proof} It can be checked (using the componentwise multiplication in $\kk[\Omega]$) that with $e:=v\o 1$ and $R:=\kk[Gv]\rtimes G$, $e$ is idempotent, $ReR=R$, and $eRe = v\o\kk[G_v]$.
\end{proof}

Note that for $v\in \Omega'$ and $v'\in Gv$, the idempotents $v\o1$ and $v' \o 1$, and the subalgebras $v\o\kk[G_v]$ and $v'\o\kk[G_{v'}]$ from the lemma are conjugate.

\begin{corollary}\label{cor::form-of-idempotents} In particular, the primitive idempotents in $\kk[\Omega]\rtimes G$ up to conjugacy are given exactly by the elements $v\o e$, for $v\in\Omega'$ and $e$ ranging over a complete set of primitive idempotents (up to conjugation) in $\kk[G_v]$.
\end{corollary}

\subsection{Idempotents of \texorpdfstring{$P_n\rtimes \kk[S_n]$}{Pn rtimes Sn}} \label{sec::idempotents-Pn-Sn}

We now turn to our main example.
Recall that a Young diagram $\l$
is called \emph{$p$-regular} if 
it does not have $p$ (or more) rows of the same size, see \Cref{sec::Young}.

\begin{corollary} \label{cor::combinatorics-idempotents-Sn}
Let $\kk$ be a field and $u(x) \in \kk[x]$
a polynomial such that each irreducible factor over $\kk$
is linear.
We set
$A := ({\kk[x]}/{ u(x) })^{\otimes n}$ for $n \in \mZ_{\geq 0}$ and turn $A$ into an $S_n$-equivariant algebra by permuting the $n$ tensor factors.
Set $R := A \rtimes S_n$,
and let $r$ be the number of different zeros of the polynomial $u(x)$.
Then the primitive idempotents in $R$
up to conjugation
are in bijection with the set
\[
\bigsqcup_{\substack{n = (f_1 + \dots + f_r) \\ f_i \geq 0} } \prod_{i=1}^r \Xk{S_{f_i}}.
\]
In particular, if $\cha( \kk ) = 0$, then
this set is in bijection with
\[
\bigsqcup_{\substack{n = (f_1 + \dots + f_r) \\ f_i \geq 0} } \prod_{i=1}^r \{ \text{Young diagrams $\l$ with $|\l| = f_i$} \}
\]
and if $\cha( \kk ) = p$ for $p$ a prime, then
this set is in bijection with
\[
\bigsqcup_{\substack{n = (f_1 + \dots + f_r) \\ f_i \geq 0} } \prod_{i=1}^r \{ \text{$p$-regular Young diagrams $\l$ with $|\l| = f_i$} \}.
\]
\end{corollary}
\begin{proof}
Let $\Mg$ denote the set of 
primitive idempotents 
of $B := \frac{A}{\rad(A)}$.
Our assumption on $u(x)$
implies $B = \kk[\Mg]$.
The action of $S_n$ on $\Mg$
can be identified
with the action of
$S_n$ on the set of functions $f: \{1, \dots, n\} \rightarrow Z$
by precomposition, where
$Z$ denotes the set of distinct zeros of $u(x)$.
The orbits are classified by the fiber cardinalities $(|f^{-1}(z)|)_{z \in Z}$,
and the stabilizer of $f$
is given by
$\prod_{z \in Z} S_{|f^{-1}(z)|} \leq S_n$.
We apply \Cref{corollary:abstract_classification}
and use the fact that irreducible representations of symmetric groups are classified by
Young diagrams
if $\cha(\kk) = 0$,
and by $p$-regular Young diagrams
if $\cha(\kk) = p$ \cite{J78}*{Theorem 11.5}.
\end{proof}

In the following, we construct distinguished idempotents in the rings $A\rtimes S_n$.
Given Young diagrams $\l_1,\dots,\l_r$ such that $\sum_{i=1}^r|\l_i|=n$, we have a natural embedding $\prod_i \kk[S_{|\l_i|}]\hookrightarrow \kk[S_n]$, and for each $i$, we have fixed an associated primitive idempotent $e_{\l_i}$ in $\kk[S_{|\l_i|}]$ in \Cref{sec::Young}. This gives us an idempotent $(e_{\l_1},\dots,e_{\l_r})$ in $\kk[S_n]$.

We decompose $\kk[x]/(u(x))$ as a direct sum of local algebras $L_1\oplus\dots\oplus L_r$ where we fix the order of the summands. Correspondingly, we have an ordered list of idempotents $E_1,\dots,E_r$ where $E_i$ is the projection onto $L_i$. This way, we associate to the list of partitions $\l_1,\ldots, \l_r$ an idempotent $E_1^{\o|\l_1|}\o\dots\o E_r^{\o|\l_r|}$ in $A=(\kk[x]/(u(x)))^{\o n}$.

\begin{definition}\label{def:Elambda} We set
$E_{\l_1,\dots,\l_r}:=(E_1^{\o|\l_1|}\o\dots\o E_r^{\o|\l_r|})\o (e_{\l_1},\dots,e_{\l_r})
$ in $A \rtimes S_n$.
\end{definition}

\begin{lemma} \label{lem::Elambda} Then $E_{\l_1,\dots,\l_r}=(1\o (e_{\l_1},\dots,e_{\l_r}))
((E_1^{\o|\l_1|}\o\dots\o E_r^{\o|\l_r|})\o 1)
$, and $E_{\l_1,\dots,\l_r}$ is an idempotent.
\end{lemma}

\begin{proof} The formula follows from the fact that $1\o\dots\o 1\o E_i^{\o|\l_i|}\o 1\o\dots\o 1$ commutes with the image of $\kk[S_{|\l_i|}]$ in $A\rtimes G$. It directly implies idempotence.
\end{proof}

\begin{corollary} \label{cor::idempotents-concrete} In the situation of \Cref{cor::form-of-idempotents}, a complete list of primitive idempotents in $A\rtimes S_n$ up to conjugation is given by the elements $E_{\l_1,\dots,\l_r}$ for partitions $\l_1,\dots,\l_r$ satisfying $|\l_1|+\dots+|\l_r|=n$.
\end{corollary}

\begin{proof} Reducing modulo the Jacobson radical, it is enough to show that the images of the described elements in $B\rtimes S_n=\kk[\Omega]\rtimes S_n$ form a complete list of primitive idempotents up to conjugation. Now by \Cref{cor::form-of-idempotents}, such a list is given by the elements $v\otimes e$, where $v$ ranges over the set $\Omega'$ and $e$ ranges over a complete list of primitive idempotents in the group algebra of the stabilizer of $v$. In our situation, $\Omega'$ is described by tuples $f_1,\dots,f_r\geq 0$ satisfying $f_1+\dots+f_r=n$ and corresponding to the idempotents $E_1^{\o f_1}\o\dots\o E_r^{\o f_r}$, the stabilizer group is $S_{f_1}\times\dots\times S_{f_r}$, and a complete list of its primitive idempotents of this stabilizer group are given by the elements $(e_{\l_1},\dots,e_{\l_r})$. This proves the assertion.
\end{proof}

We record the following observation for later use.

\begin{lemma} \label{lem::radical-of-endo-DCob} In the situation of the previous corollary, the Jacobson radical of $R:=E_{\l_1,\dots,\l_r} (A\rtimes S_n) E_{\l_1,\dots,\l_r}$ satisfies $\dim R/\rad(R)=1$.
\end{lemma}

\begin{proof}  Let us introduce shorthand notations for the two factors of the idempotents,
$$
E := (E_1^{\o|\l_1|}\o\dots\o E_r^{\o|\l_r|}) \quad\text{ and }\quad
e := (e_{\l_1},\dots,e_{\l_r})
$$ for fixed $\l_1,\dots,\l_r$, so $E_{\l_1,\dots,\l_r}=Ee=eE$ (see \Cref{lem::Elambda}). Then 
$$
 E A 
 = L_1^{\o|\l_1|}\o\dots\o L_r^{\o|\l_r|}
\quad\text{ and }\quad
 E (\kk[S_n]) E
 = E (\kk[S_{|\l_1|}\times\dots\times S_{|\l_r|}]) 
$$
as subspaces of $A\rtimes G$. We now set $A' := e ( \kk[S_{|\l_1|}\times\dots\times S_{|\l_r|}] ) e$.

We note that $e$ commutes with $E$ (by \Cref{lem::Elambda}) and even with $E A 
 = L_1^{\o|\l_1|}\o\dots\o L_r^{\o|\l_r|}$, as those tensor factors $L_i$ of the latter which are permuted by conjugation with elements from $S_{|\l_1|}\times\dots\times S_{|\l_r|}$ are identical. Hence,
$$ E_{\l_1,\dots,\l_r} (A\rtimes S_n) E_{\l_1,\dots,\l_r}
 = eE (A\rtimes S_n) Ee
 = L_1^{\o|\l_1|}\o\dots\o L_r^{\o|\l_r|}\o A'
$$
as subspaces of $A\rtimes G$ and also as $\kk$-algebras.

$A'$ is the endomorphism algebra of the indecomposable projective $S_{|\l_1|}\times\dots\times S_{|\l_r|}$-module $\kk[S_{|\l_1|}\times\dots\times S_{|\l_r|}] (e_{\l_1},\dots,e_{\l_r})$. Hence, $A'/\rad(A')$ is isomorphic to the endomorphism algebra of an irreducible $S_{|\l_1|}\times\dots\times S_{|\l_r|}$-module, and as $\kk$ is a splitting field for each symmetric group (see \Cref{remark:sn_splitting_field}), this algebra is one-dimensional. As also $\kk$ is a splitting field of $u$, $L_i/\rad(L_i)$ is one-dimensional for each $1\leq i\leq r$. But then
$$
\rad(L_1)\o L_2\o\dots\ L_r\o A'
+ \dots + L_1\o\dots L_{r-1}\o\rad(L_r)\o A'
+ L_1\o\dots\o L_r\o\rad(A')
$$
is an ideal in $E_{\l_1,\dots,\l_r} (A\rtimes S_n) E_{\l_1,\dots,\l_r}$ yielding a one-dimensional quotient; hence, it is the Jacobson radical. 
\end{proof}

\begin{lemma}[The purely inseparable case] \label{lemma:insep_case}
Let $\kk$ be a field of characteristic $p > 0$ and $u(x) \in \kk[x]$
such that the splitting field $\KK$ of $u$ is purely inseparable.
Then the idempotents
$E_{\l_1,\dots,\l_r} \in (\KK[x]/(u(x)))^{\o n} \rtimes S_n$ of \Cref{cor::idempotents-concrete}
already lie in $(\kk[x]/(u(x)))^{\o n} \rtimes S_n$.
In particular, the classification of \Cref{cor::combinatorics-idempotents-Sn}
also holds in this case.
\end{lemma}
\begin{proof}
Since $\KK$ is finite and purely inseparable over $\kk$,
every $x \in \KK$ satisfies $x^q \in \kk$ where $q$ is some power of $p$.
In particular, if $A$ is a commutative $\kk$-algebra
and $R := \KK \otimes_{\kk} A$,
then any $r \in R$ satisfies $r^q \in A$ where $q$ is some power of $p$. Thus, any idempotent in $R$ already lies in $A$.
It follows that the idempotents
$\bigotimes_i E_i^{\o f_i}$ in $(\KK[x]/(u(x)))^{\o n}$
lie in $(\kk[x]/(u(x)))^{\o n}$ and consequently the
$E_{\l_1,\dots,\l_r}$ lie in $(\kk[x]/(u(x)))^{\o n}\rtimes S_n$, since the idempotents $e_{\lambda}$ are defined over $\mF_p \subseteq \kk$ by \Cref{remark:sn_splitting_field}.
\end{proof}

\subsection{Classification of Indecomposables of \texorpdfstring{$\DCob{\alpha}$}{DCob(alpha)}}
\label{sec::classification-splitting-case}

Let $u=u_\alpha$ be the polynomial associated to $\alpha$ as in \eqref{eq:ualpha}. In this section, we will assume that $\kk$ is a splitting field for $u=u_\alpha\in\kk[t]$.

As before, set $P_n:=(\kk[x]/(u(x)))^{\o n}$ for each $n\geq0$. In \Cref{sec::idempotents-Pn-Sn} we have obtained an explicit description of the primitive idempotents in $P_n\rtimes \kk[S_n]$ as follows: we fix a decomposition of $P_1=\kk[x]/(u(x))$ as a direct sum of local subalgebras,
$$
P_1 = L_1\oplus\dots\oplus L_r
$$
(where $r$ is the number of distinct roots of $u$ in $\kk$), and set
$$
\Lambda := \{(\l_1,\dots,\l_r): \l_i \text{ a $p$-regular Young diagram}, \cha\kk=p \} .
$$
Then by \Cref{cor::idempotents-concrete}, the elements of $\Lambda$ specify primitive idempotents $E_{\l_1,\dots,\l_r}$ in $P_n\rtimes\kk[S_n]$, where $n=|\l_1|+\dots+|\l_r|$, and all primitive idempotents can be obtained uniquely in this way (up to conjugation).

Recall that we have a filtration $\cD_\bullet$ on $\DCob{\alpha}$ (\Cref{def::filtration-DCob}), where $\cD_n$ can be characterized as the additive idempotent complete subcategory generated by the object $[n]$ in $\DCob{\alpha}$. The filtration $\cD_\bullet$ will allow us to describe the indecomposable objects in $\DCob{\alpha}$. 

\begin{lemma} \label{lem::indecomposables} For each $n\geq0$, $\End_{\cD_n/\cD_{n-1}}(\pi([n]))\cong P_n\rtimes S_n$, and a complete list of indecomposable objects in $\cD_n/\cD_{n-1}$ up to isomorphisms is given by the images of the primitive idempotents $E_{\l_1,\dots,\l_r}$ with $|\l_1|+\dots+|\l_r|=n$ viewed as endomorphisms of $\pi([n])$.
\end{lemma}

\begin{proof}
  By \Cref{thm:KoDCob} and its proof, $\cD_n/\cD_{n-1} \simeq \rfproj{\End_{\cD_n/\cD_{n-1}}([n])}$, and the desired classification of primitive idempotents in $P_n\rtimes S_n$ was provided in \Cref{cor::combinatorics-idempotents-Sn}.
\end{proof}

\begin{proposition} \label{prop::indecomposables} The indecomposable objects in $\DCob{\alpha}$ are in bijection with the set $\Lambda$ in the following way: Every indecomposable object corresponds uniquely to an indecomposable object in a quotient category $\cD_n/\cD_{n-1}$ for some $n\geq0$ which is the image of the primitive idempotent $E_{\l_1,\dots,\l_r}$ in $\End_{\cD_n/\cD_{n-1}}(\pi([n]))\cong P_n\rtimes \kk[S_n]$ for some partitions $\l_1,\dots,\l_r$ satisfying $|\l_1|+\dots+|\l_r|=n$.
\end{proposition}

\begin{proof}  
By \Cref{lemma:classify_via_filt}, the indecomposable objects in $\DCob{\alpha}$ correspond exactly to those in the quotients $\cD_n/\cD_{n-1}$. Hence, the statement follows from \Cref{lem::indecomposables}.
\end{proof}

\begin{corollary} \label{cor::splitting-field-DCob} If $\kk$ is a splitting field for $u=u_\alpha$, then $\kk$ is a splitting field for $\DCob{\alpha}$ in the sense of \Cref{def::splitting-field-for-C}.
\end{corollary}

\begin{proof} The defining criterion, the property that the quotient of the endomorphism algebra of an indecomposable object by its Jacobson radical is one-dimensional, is verified by \Cref{lem::radical-of-endo-DCob}.
\end{proof}

\begin{remark} If $\alpha=0$, then $u(t)=1$ and there are no zeros. Hence, $r=0$ and $P_n=0$ for all $n\geq 1$ , so there is exactly one indecomposable object, the tensor unit $[0]$.

If $\alpha=c\in\kk\setminus\{0\}$, then $r=1$ and the indecomposable objects are indexed by partitions $\l$, see \Cref{sec:constantcase}.
\end{remark}

\begin{example}\label{expl:RepStasDCob} For $T\in \kk$, $T\neq 0$, we recover $\uRep (S_T)$ as $\DCob{\alpha}$, where 
$$\alpha=(T,T,\ldots) = \frac{T}{1-t},$$ 
where $t$ is the formal variable, see
 \cite{KS}*{Section~4}. In this case, \Cref{prop::indecomposables} recovers the classification of indecomposable objects already obtained in \Cref{cor::gr-RepSt}. The case $\uRep (S_0)$ can be obtained using the more general setup of \Cref{sec:inflations}.
\end{example}

\begin{definition} \label{def::XY} Let $X_{\l_1,\dots,\l_r}$ be the indecomposable object (defined up to isomorphism) in $\DCob{\alpha}$ corresponding to the $(\cha\kk)$-regular partitions $\l_1,\dots,\l_r$ according to \Cref{prop::indecomposables} and denote
$$
Y_1:=X_{(1),\emptyset,\dots,\emptyset},\quad
\dots,\quad 
Y_r:=X_{\emptyset,\dots,\emptyset,(1)} .
$$
\end{definition}

\Cref{prop::indecomposables} implies that the indecomposable object $X_{\l_1,\dots,\l_r}$ occurs as a direct summand in $[n]$, where $n=|\l_1|+\dots+|\l_r|$, but not in $[n-1]$. In particular, the indecomposable objects occurring in the object $[1]$ in $\DCob{\alpha}$ are $Y_1,\dots,Y_r$. We can make this more precise.

\begin{lemma}\label{lem:Y1} There is a decomposition $[1] \cong Y_1\oplus\dots\oplus Y_r\oplus Y$ for an object $Y$ in $\cD_0$.
\end{lemma}

\begin{proof} $Y_1,\dots,Y_r$ are the images of primitive idempotents $\tilde E_1,\dots,\tilde E_r$ in $\End([1])$ such that $\tilde E_i$ is a lift (unique up to conjugation) of a primitive idempotent in $L_i$ (which is again unique up to conjugation) where $L_1\oplus\dots\oplus L_r$ is our fixed decomposition of $\End([1])/\cI_{<1}\cong\kk[x]/(u(x))=P_1$ into local algebras.
\end{proof}

Recall from \Cref{sec::Young} that to every $(\cha\kk)$-regular partition $\l$ of size $n$ we associated a functor $F^\l$ sending any object $X\in\DCob{\alpha}$ to a direct summand of its $n$-th tensor product $X^{\otimes n}$, which is the image of the primitive idempotent $e_\lambda$ in $\kk[S_n]$ corresponding to $\l$, interpreted as an endomorphism of $X^{\otimes n}$ via the categorical symmetric braiding. We can now ask how indecomposable objects can be located in the image of such functors. 

\begin{proposition} \label{prop::Flambda} For partitions $\l_1,\dots,\l_r$ satisfying $|\l_1|+\dots+|\l_r|=n$, 
$$ F^{\l_1}(Y_1)\otimes\dots\otimes F^{\l_r}(Y_r)\cong X_{\l_1,\dots,\l_r}\oplus Y,$$
where $Y$ is an object in $\cD_{n-1}$.
\end{proposition}

\begin{proof} We first note that the tensor product $F^{\l_1}(Y_1)\otimes\dots\otimes F^{\l_r}(Y_r)$ is the image of the idempotent $(\bigotimes_i \tilde E_i^{\o |\l_i|})\o (e_{\l_1},\dots,e_{\l_r})$ in $\End([1])^{\otimes n}\rtimes S_n\subseteq \End([n])$, with primitive idempotents $\tilde E_i$ as in the proof of \Cref{lem:Y1}. Note the idempotence of these elements can be verified as in \Cref{lem::Elambda}.  Hence, $\pi(F^{\l_1}(Y_1)\otimes\dots\otimes F^{\l_r}(Y_r))\cong \pi(X_{\l_1,\dots,\l_r})$ in $\cD_n/\cD_{n-1}$, with $\pi: \cD_n \rightarrow \cD_n/\cD_{n-1}$ the natural projection functor. This implies the assertion with \Cref{lem::indec-direct-summand}.
\end{proof}

\subsection{Tensor product decomposition}

In this section, $\alpha\in \kk[[t]]$ for $\kk$ a splitting field for $u_\alpha$. We will relate the classification of indecomposable objects in $\DCob{\alpha}$ to the tensor decomposition of these categories obtained in \cite{KKO}. If $u_\alpha$ has a unique root, we prove the following.

\begin{corollary}\label{Cor:uniqueroot}
Assume that $u_\alpha$ has a unique root in a splitting field $\kk$. Then there are isomorphisms of graded rings
$$\gr \Gr(\DCob{\alpha})\cong \Gr(\Symu{\alpha})\cong \Sym^p.$$
This isomorphism is given by  $[X_\lambda]\mapsto [V_\lambda]$ for the respective indecomposable objects labeled by a ($p$-regular) Young diagram $\lambda$.
\end{corollary}
\begin{proof}
The first isomorphism follows from \Cref{thm:KoDCob}. The second isomorphism follows from \Cref{cor:removemulti} and the definition of $\Sym^p$ in \Cref{def:Sym} since, by assumption, $u_\alpha$ has a unique root, hence the product of all its distinct irreducible factors $u'_{\alpha}$ is a linear polynomial, and $\Gr(\mathcal{S}_{u'_\alpha}) \cong \Sym^p$.
\end{proof}

To address the case when $u_\alpha$ has multiple roots, we recall the following result. 

\begin{theorem}[{\cite{KKO}*{Proposition 2.14, 2.16}}]\label{thm:KKO}
Let $\alpha$ be a rational function and choose a partial fraction decomposition 
\begin{equation}\label{eq:partialdec}
\alpha(t)=\sum_{i=1}^l\frac{p_i(t)}{q_i(t)}+\alpha_0(t), \qquad \alpha_0(t)\in \kk[t],\quad \deg p_i(t)<\deg q_i(t),
\end{equation}
where the $q_i(t)$ have a unique zero.
Denote $\alpha_i:=p_i(t)/q_i(t)$, for $i=1,\ldots, l$.
Then there is an equivalence of Karoubian tensor categories
\begin{equation}
    F_\alpha^{D}\colon \DCob{\alpha}\isomorph \boxtimes_{i=0}^l \DCob{\alpha_i}, \qquad [1]\longmapsto \bigoplus_i [0]^{\boxtimes (i-1)}\boxtimes[1]\boxtimes[0]^{\boxtimes(n-i)}.
\end{equation}
\end{theorem}

Note that if $\alpha_0= 0$, then $\DCob{\alpha_0}\simeq \lvec{\kk}$ and $\DCob{\alpha_0}$ may be omitted from the partial fraction and tensor product  decomposition. To explain how the classification in \Cref{prop::indecomposables} is compatible with the above tensor product decomposition, recall the following elementary observations on partial fraction decompositions.

\begin{lemma}\label{lemma:partialfrac} Given coprime $p,q\in \kk[t]$  such that $q(0)\neq 0$ and $\alpha=p/q$, define $u_\alpha$ as in Equation \eqref{eq:ualpha}.
\begin{enumerate}
\item[(i)]  $\rho=0$ is a zero of $u_\alpha(t)$ if and only if $\deg p\geq \deg q$. In this case, its multiplicity is
$$\mult_0 = \deg p-\deg q+1.$$
\item[(ii)] Let $\rho\neq 0$ be a root of $q(t)$ with multiplicity $\mult_\rho$. Then $\rho^{-1}$ is a root of $u_\alpha(t)$ with the same multiplicity.
\item[(iii)] The assignment $\rho \mapsto \rho^{-1}$ defines a bijection between roots of $q(t)$ and non-zero roots of $u_\alpha(t)$.
\item[(iv)] There is a bijection between the zeros $\rho\neq 0$ of $u_\alpha(t)$ and partial fractions $\alpha_i(t)$, $i=1,\ldots, l$, in Equation \eqref{eq:partialdec} sending $\rho$ to the unique partial faction with denominator $(t-\rho^{-1})^{\mult_\rho}$.%
\item[(v)] The polynomial $\alpha_0$ is zero if and only if $\deg p<\deg q$. Otherwise, $\deg\alpha_0=\deg p-\deg q$.
\end{enumerate}
\end{lemma}
\begin{proof}
Note that by definition, 
$$u_\alpha(t)=t^{k-m}q(t^{-1})t^m,$$
where $m=\deg q$ and $k=\max\{\deg q, \deg p+1\}$. Thus, $\rho=0$ is a zero of $u_\alpha(t)$ if and only if $\deg p+1-\deg q>0$. This proves Part (i) and the statement in Part (v) follows from Euclidean division. The form of $u_\alpha$, expressed in terms of $q(t^{-1})$ also directly implies the statements in Part (ii)--(iii).
Part (iv) follows from the above combined with the elementary fact that the denominators in the partial factions are in bijection with the zeros of $q(t)$. 
\end{proof}

We now equip $\boxtimes_{i=0}^l \DCob{\alpha_i}$ with the natural filtration on tensor products of filtered Krull--Schmidt categories from \Cref{rem::natural-filtration}.

\begin{lemma}\label{lem:filtcomp}
    The equivalence $F^D_\alpha$ from \Cref{thm:KKO} is compatible with the filtrations and hence induces an equivalence
    $$\gr F^D_\alpha\colon    \gr \DCob{\alpha}\isomorph \boxtimes_{i=0}^l \gr \DCob{\alpha_i}.$$
\end{lemma}
\begin{proof} We denote the tensor product filtration on $\DCob{\alpha}\simeq \boxtimes_i \DCob{\alpha_i}$ from \Cref{rem::natural-filtration} by $\cD'_\bullet$.
Note that the equivalence $F^D_\alpha$ sends
$$
[m] \longmapsto
\bigoplus_{i_0,\dots,i_l\geq0, i_0+\dots+i_l=m} \binom{m}{i_1,\dots,i_l} [i_0]\boxtimes\dots\boxtimes[i_l]
,
$$
for all $0\leq m\leq n$, so $F^D_\alpha (\cD_n)\subseteq\cD'_n$. Here, we use the multinomial coefficient $\binom{m}{i_1,\dots,i_l}=m!/(i_1!\dots i_l!)$, which is the number of ways to distribute $m$ objects into $l$ bins of sizes $i_1,\dots,i_l$.

However, by definition, $\cD'_n$ is the subcategory generated by $\cD_{i_0}\boxtimes\dots\boxtimes \cD_{i_l}$ for all tuples $(i_0,\dots,i_l)$ with $i_0+\dots+i_l\leq n$.
Hence, as any $[i_0]\boxtimes\dots\boxtimes[i_l]$ is a direct summand in the right-hand side, $\cD'_n\subseteq F^D_\alpha (\cD_n)$. This implies the equivalence $F^D_\alpha (\cD_n)$ is compatible with filtrations and $\gr F^D_\alpha (\cD_n)$ gives an equivalence of categories by \Cref{lem:filt-comp-equiv}. Now the result follows using \Cref{cor:tensorfiltrations}(1).
\end{proof}

The following proposition describes how the indecomposable objects behave under the tensor product decomposition from \Cref{thm:KKO}.

\begin{proposition}\label{prop:tensordec-indec} Assume that $u=u_\alpha$ splits over $\kk$. If $\alpha_0\neq 0$, then under the tensor equivalence $F_\alpha^D$ from \Cref{thm:KKO}, the indecomposable object $X_{\l_0,\dots,\l_l}$ of $\DCob{\alpha}$ from \Cref{prop::indecomposables} labeled by partitions $\l_0,\dots,\l_l$  is mapped to the object
$$X_{\l_0}\boxtimes \ldots \boxtimes X_{\l_l}.$$
Here, $X_{\l_i}$ is the indecomposable object of $\DCob{\alpha_i}$ labeled by $V_{\l_i}\in \Xk{S_{|\lambda_i|}}$.

If $\alpha_0=0$, we omit $\l_0$ and $X_{\l_0}$ as $\DCob{\alpha_0}\simeq \lvec{\kk}$.
\end{proposition}
\begin{proof} We first consider the case $\alpha_0\neq0$. By \Cref{lemma:partialfrac}, $u_{\alpha_i}$ has a unique zero. Hence, for all $0\leq i\leq l$, we denote by $Y_1^{(i)}$ the unique indecomposable summand of the tensor generator $[1]^{(i)}$ of $\DCob{\alpha_i}$ which does not lie in degree $0$ of the filtration of $\DCob{\alpha_i}$, cf. \Cref{lem:Y1} and \Cref{Cor:uniqueroot}. As $[1]\cong Y_0\oplus\dots\oplus Y_l\oplus Y$, where $Y$ is an object in degree $0$ of the filtration, is mapped to $\bigoplus_i [0]^{\boxtimes (i-1)}\boxtimes [1]^{(i)}\boxtimes[0]^{\boxtimes (n-i)}$, it follows that $Y_i\mapsto [0]^{\boxtimes (i-1)}\boxtimes Y_1^{(i)}\boxtimes[0]^{\boxtimes (n-i)}$ using \Cref{lem:filtcomp}.
Now by \Cref{prop::Flambda}, $X_{\l_0,\dots,\l_l}$ is the unique indecomposable summand of $F^{\l_0}(Y_0)\o\dots\o F^{\l_l}(Y_l)$ which does not lie in degree $n-1$. Hence, its image under $F_\alpha^D$ must be the unique indecomposable summand $V$ 
of 
$$
\bigotimes_{i}
F^{\l_i}([0]^{\boxtimes (i-1)}\boxtimes Y_1^{(i)}\boxtimes[0]^{\boxtimes (n-i)}) \cong F^{\l_0}(Y_1^{(0)})\boxtimes\dots\boxtimes F^{\l_l}(Y_1^{(l)})
$$
which does not lie in degree $n-1$ of the tensor product filtration $\cD_\bullet'$ on $\boxtimes_i \DCob{\alpha_i}$ by \Cref{lem:filtcomp}. By  \Cref{prop::Flambda}, the object $X_{\l_i}$ is the unique indecomposable summand of $F^{\l_i}(Y_1^{(i)})$ in $\DCob{\alpha_i}$ which does not lie in  degree $|\l_i|-1$. Hence,  $X_{\l_0}\boxtimes\dots\boxtimes X_{\l_l}$ is this unique indecomposable summand $V$ (using \Cref{theorem:grothendieck_dec} and \Cref{cor::splitting-field-DCob}), which proves the assertion.

The case $\alpha_0=0$ can be treated similarly, omitting the factor $\DCob{\alpha_0}\simeq \lvec{\kk}$ in the $\boxtimes$-decomposition.
\end{proof}

We can now describe the Grothendieck ring of $\DCob{\alpha}$ as a product of the graded rings $\Sym^p=\Sym_\bullet^p$ from \Cref{def:Sym}. Recall that the case of a unique root is addressed in \Cref{Cor:uniqueroot}.

\begin{corollary}\label{corollary:main_thm_split_case} Assume that $u=u_\alpha$ splits over the field $\kk$ of characteristic $p$.
Then the tensor equivalence $F_\alpha^D$ from \Cref{thm:KKO} induces an isomorphism of Grothendieck rings
\begin{align*}
     \gr \Gr (\DCob{\alpha})\isomorph (\Sym^p)^{\otimes k}.
\end{align*}
If $\alpha_0\neq 0$, then $k=l+1$ and $[X_{\l_0,\dots,\l_l}]\longmapsto \Grel{\l_0}\otimes \ldots \otimes \Grel{\l_l}$. If $\alpha_0=0$, then $k=l$ and 
$[X_{\l_1,\dots,\l_l}]\longmapsto \Grel{\l_1}\otimes \ldots \otimes \Grel{\l_l}$.
\end{corollary}
\begin{proof}
By assumption, $\kk$ is a splitting field for $u_\alpha$, and hence for the category $\DCob{\alpha}$ and all $\DCob{\alpha_i}$ by \Cref{cor::splitting-field-DCob}.
Hence, we derive the following composition of isomorphisms of graded rings 
\begin{align*}
\gr\Gr (\DCob{\alpha})&\xrightarrow[\text{\Cref{theorem:iso_graded_rings}}]{\cong} \Gr(\gr \DCob{\alpha})\\ &\xrightarrow[\text{~\Cref{lem:filtcomp}~}]{\cong}
\Gr(\boxtimes_{i=0}^l \gr \DCob{\alpha_i})\\
&\xrightarrow[\text{\Cref{cor:tensorfiltrations}}]{\cong}
\Gr(\gr(\boxtimes_{i=0}^l  \DCob{\alpha_i}))\\
&\xrightarrow[\text{\Cref{theorem:iso_graded_rings}}]{\cong}
\gr\Gr(\boxtimes_{i=0}^l \DCob{\alpha_i})\xrightarrow[\text{\Cref{corollary:grothendieck_dec}}]{\cong}
\bigotimes_{i=0}^l\gr \DCob{\alpha_i}.
\end{align*}
For $i=1,\ldots, l$, the tensor factors  of the latter are given by $\Sym^p$ using \Cref{Cor:uniqueroot} as  each $u_{\alpha_i}$ has a unique root. If $\alpha_0\neq 0$, then the same holds for $i=0$, but if $\alpha_0=0$, then $\DCob{\alpha_0}\cong \lvec{\kk}$ and its Grothendieck ring is $\mZ$ which maybe omitted in the tensor product.  

The assignment on indecomposable object follows directly by combining the assignments in \Cref{prop:tensordec-indec} and \Cref{Cor:uniqueroot}.
\end{proof}

\section{Galois descent for categories and \texorpdfstring{$\DCob{\alpha}$}{DCob(alpha)} over a general field}
\label{sec:Galois-general-kk}

In this section, we extend our classification of indecomposable objects of $\DCob{\alpha}$ provided in \Cref{sec:classification} for a splitting field of $u_\alpha$ to general fields. To this end, we study Galois descent for Krull--Schmidt categories.

\subsection{Motivation}\label{subsection:a_first_example}

In \Cref{cor::combinatorics-idempotents-Sn}, we classified the primitive idempotents of crossed products for the quotient of a polynomial algebra by a polynomial splitting into linear factors. This facilitated our description of the indecomposable objects of $\DCob{\alpha}$ in the case where the polynomial $u_\alpha$ splits over the base field $\kk$ in \Cref{sec::classification-splitting-case}. 

We demonstrate the subtleties in the case when $u_\alpha$ does not split over $\kk$ with the following motivational example. Let
$B := \frac{\mR[x]}{\langle x^2 + 1 \rangle} \otimes_{\mR} \frac{\mR[x]}{\langle x^2 + 1 \rangle}$ and $R := B \rtimes S_2$, where $S_2$ acts via permutation of the factors. We want to determine the primitive idempotents in $R$. First, we note that
\begin{align*}
    B &\cong_{\mR} \frac{\mR[x,y]}{\langle x^2 + 1, y^2 + 1 \rangle} 
    \cong_{\mR} \frac{\mC[y]}{\langle y^2 + 1 \rangle} 
    \cong_{\mC} \frac{\mC[y]}{\langle y - i \rangle} \oplus \frac{\mC[y]}{\langle y + i \rangle} \cong_{\mC} \mC \oplus \mC .
\end{align*}
The action of the generator $\sigma \in S_2$
on $B$ is fully described by
$\sigma( 1 \otimes x ) = x \otimes 1$,
and we can follow these two elements through all these isomorphisms:
\[
x \otimes 1 \mapsto x \mapsto i \mapsto (i,i) \mapsto (i,i)
\]
and
\[
1 \otimes x \mapsto y \mapsto y \mapsto (y,y) \mapsto (i,-i) .
\]
It follows that $\sigma$ acts on $\mC \oplus \mC$ as
\[
(a,b) \mapsto ({a}, \overline{b}) ,
\]

and that 
\[R \cong (\mC \otimes_{\mR} \mR[S_2]) \oplus (\mC \rtimes S_2)
\cong \mC[S_2] \oplus (\mC \rtimes S_2)
\] as algebras, where $S_2$ acts on $\mC$ via conjugation.
Thus, we get two primitive idempotents from
$\mC[S_2]$, and another one from
the other summand due to the equivalence
\[
\lfmod{(\mC \rtimes S_2)} \rightarrow \lfmod{\mR}: V \mapsto V^{S_2} .
\]
So we have $3$ primitive idempotents in this case, which can be interpreted as follows: we have $3$ indecomposable objects in the 2-nd graded piece of the category $\Symu{\alpha}$ for every $\alpha$ such that $u_{\alpha} = t^2 + 1$. 
This is remarkable since in the case of a splitting $u_{\alpha}$, the number of irreducible objects in the second graded piece of $\Symu{\alpha}$ cannot be $3$: it is $2$ if $\deg( u_{\alpha} ) = 1$, it is $5$ if $\deg( u_{\alpha} ) = 2$, and for even higher degrees, the number only increases. This suggests that a generalization of our previous results to an arbitrary field needs further tools, which in this paper will be the theory of group actions on categories and Galois descent.

\subsection{Group actions on categories}\label{sec:groupactions}

In this section, we discuss the theory of strict group actions on categories and their associated
categories of equivariant objects.
We restrict to the strict case in order to simplify the coherence conditions.
For the non-strict case, we refer to \cite{EGNO}*{Section 2.7}).
In this subsection, $G$ denotes a finite group.

\begin{definition}
A \emph{strict (left) action} of $G$ on a $\kk$-linear category $\cC$
is given by a monoid homomorphism
$G \rightarrow \End( \cC ): g \mapsto T_g$, where
$\End( \cC )$ denotes the monoid of $\kk$-linear functors $\cC \rightarrow \cC$ with multiplication given by composition.
\end{definition}

\begin{example}\label{example:action_on_Kar}
  Let $G$ act strictly on $\cC$. 
  Then $G$ acts strictly on $\Kar( \cC )$:
    for $X \in \cC$ and $e$ an idempotent in $\End_{\cC}(X)$, we set $T_g( (X, e) ) = (T_g(X), T_g(e))$
    for all $g \in G$.
\end{example}

\begin{example}\label{ex:fieldext-actiononC}
  Let $\kk \subseteq \KK$ be a field extension, 
  and let $G$ act from the left on $\KK$ by $\kk$-algebra automorphisms.
  For any $\kk$-linear category $\cC$, we get a strict $G$-action on
  $\KK \otimes_{\kk} \cC$ by setting $T_g(X)=X$ and (the linear extension of) the formula
  \[T_g( X \xrightarrow{a \otimes \alpha} Y ) 
  = (X \xrightarrow{g(a) \otimes \alpha} Y)\]
  for all $\alpha \in \Hom_{\cC}( X, Y )$, $a \in \KK$, $g \in G$.
  Combining this example with \Cref{example:action_on_Kar} yields a strict action on $\cC^\KK$.
\end{example}

\begin{definition}
A \emph{$G$-equivariant functor} between categories equipped with strict $G$-actions
$( \cC, (T_g)_g )$ and $(\cD, (T_g)_g)$ 
consists of a functor $F: \cC \rightarrow \cD$
and isomorphisms $\varphi_X^g: F T_g(X) \rightarrow T_g F(X)$ natural in $X \in \cC$ for all $g \in G$,
such that 
\[
  \big(F  T_{gh}(X) \xrightarrow{\varphi^g_{T_h(X)}} T_g  F  T_h(X) \xrightarrow{T_g(\varphi^h_{X})} T_{gh}  F(X)\big) = \varphi^{gh}_X.
\]
It is part of a \emph{$G$-equivariant equivalence} if $F$
is part of an equivalence of categories.
\end{definition}

\begin{definition}
A \emph{$G$-equivariant object} of $( \cC, (T_g)_g )$ is given by the following data:
\begin{enumerate}
    \item an (underlying) object $X \in \cC$,
    \item isomorphisms $T_g(X) \xrightarrow{\alpha_g} X$ for all $g \in G$
\end{enumerate}
subject to the compatibility conditions
\begin{enumerate}
  \item $\alpha_e = \id_X$ for $e \in G$ the neutral element,
  \item $\alpha_h \circ T_h( \alpha_g ) = \alpha_{hg}$ for all $h, g \in G$.
\end{enumerate}
A morphism between $G$-equivariant objects $(X, (\alpha_g)_g)$
and $(Y, (\beta_g)_g)$
consists of a morphism $\gamma: X \rightarrow Y$
such that 
\[\gamma \circ \alpha_g = \beta_g \circ T_g( \gamma )\]
for all $g \in G$. We denote the resulting category of $G$-equivariant objects by $\cC^G$.
\end{definition}

\begin{example}[$G$-indexed sums]\label{example:orbit_sums}
  Any object $X$
  of a category  $\cC$ equipped with a strict $G$-action
  gives rise to an object in $\cC^G$ by the following construction:
  we set 
  \[ Y := \bigoplus_{g \in G} T_g( X ) \] 
  as underlying object and
  \[
    \alpha_g: T_h( Y ) \cong \bigoplus_{g \in G} T_{hg}( X ) \rightarrow \bigoplus_{g \in G} T_{g}( X )
  \]
  is given by the permutation matrix sending the $g$-th summand of the source to the $(hg)$-th summand of the range.
\end{example}

\begin{remark}\label{remark:inclusion_Gro}
We have the canonical forgetful functor
\[
\mathrm{Forg}: \cC^G \rightarrow \cC:
(X, (\alpha_g)_g) \mapsto X.
\]
Moreover, since $T_g$ is $\kk$-linear, the action of $G$ on $\cC$
induces an action on $\Gr( \cC )$,
and
\[
\Gr( \mathrm{Forg} ):\Gr(\cC^G) \rightarrow \Gr(\cC)^G \hookrightarrow \Gr(\cC)
\]
factors over the subgroup of invariants $\Gr(\cC)^G$.
\end{remark}

\begin{remark}\label{remark:invariant_orbit_sums}
  When $\cC$ is a Krull--Schmidt category with a strict $G$-action, then $G$
  acts on $\Gr( \cC )$ be permutation of the basis
  vectors given by $\Indec( \cC )$.
  Its group of invariants $\Gr( \cC )^G$ is spanned by orbit sums, i.e.,
  \[
    \Gr( \cC )^G = \Big\langle \sum_{m \in B} m \mid \text{$B \subseteq \Indec( \cC )$ is a $G$-orbit} \Big\rangle \subseteq \Gr( \cC ).
  \]
  Clearly, every sum of elements in an orbit is an invariant. Conversely, let
  $b = \sum_{m \in \Indec( \cC )}a_m \cdot m$
  be an invariant with $a_m \in \mZ$
  (all but finitely many equal to zero).
  Then $gb = b$ implies $a_m = a_{gm}$
  for all $m \in \Indec( \cC )$, $g \in G$. Thus, the coefficients of $b$ are constant on orbits.
\end{remark}

\begin{lemma}\label{lemma:equivariant_equiv}
  A $G$-equivariant functor $F: \cC \rightarrow \cD$ 
  induces a functor 
  \[ F^G: \cC^G \rightarrow \cD^G: (X, (\alpha_g)_g) \mapsto \big(FX, ( T_gF(X) \cong FT_g(X) \xrightarrow{F(\alpha_g)} FX  )_g \big) \]
  that is an equivalence if $F$ is a $G$-equivariant equivalence.
\end{lemma}
\begin{proof}
  By a direct computation using the coherence conditions.
\end{proof}

\begin{example}\label{example:G_action_on_mod}
Let $G$ act from the left on a $\kk$-algebra $R$ via $\kk$-algebra automorphisms.
Then we have a strict left $G$-action on the category of right modules $\rmod{R}$
via restriction of scalars along $R \xrightarrow{g^{-1}} R$, i.e.,
for $M \in \rmod{R}$, the module $T_g(M)$ has the same underlying $\kk$-vector space as $M$, and the 
right action $\ast$ of $R$ on $T_g(M)$ is given by $m\ast r = mg^{-1}(r)$ for all $m \in T_g(M), r \in R$.
For a morphism $\mu \in \Hom_R(M,N)$,
we have $T_g(\mu): T_g(M) \rightarrow T_g(N), m \mapsto \mu(m)$, i.e.,
the underlying map of $\kk$-vector spaces remains the same.
Note that we have an isomorphism of right $R$-modules $R \to T_g(R), r \mapsto g^{-1}(r)$,
and it follows that
the action restricts to an action on $\rfproj{R}$.
Unwrapping its defining data, an object in $(\rmod{R})^G$ is given by $M \in \rmod{R}$
together with $\kk$-linear automorphisms $M \xrightarrow{\alpha_g} M$
such that $\alpha_g \circ \alpha_h = \alpha_{gh}$ and
$\alpha_g( mr ) = \alpha_g(m)g(r)$ for all $g, h \in G$, $m \in M$. 
In other words, $M$ comes equipped with a semi-linear action of $G$ from the left, and thus
\[
  (\rmod{R})^G \simeq \rmod{(R \rtimes G)}.
\]
\end{example}

\begin{lemma}\label{lemma:KS_and_projectives_equiv}
Let $\cC$ be a $\kk$-linear Krull--Schmidt category equipped with a strict $G$-action.
Let $X \in \cC$ be a strict fixed point of the action, i.e., $T_g(X) = X$ for all $g \in G$.
Then the strict $G$-action on $\cC$
restricts to a strict $G$-action
on $\Kar( \{ X \}^{\oplus} )$,
and $R := \End_{\cC}( X )$ is a $G$-equivariant $\kk$-algebra.
Moreover, we obtain a $G$-equivariant equivalence
\[
\Kar( \{ X \}^{\oplus} ) \rightarrow \rfproj{R}:
Y \mapsto \Hom_{\cC}(X,Y)
\]
with the strict $G$-action on $\rfproj{R}$
given by the one of \Cref{example:G_action_on_mod}.
In particular, we get
\[
\Kar( \{ X \}^{\oplus} )^G
\simeq
(\rfproj{R})^G
\subseteq 
(\rmod{R})^G
\simeq
\rmod{(R \rtimes G)},
\]
i.e., we can identify 
$\Kar( \{ X \}^{\oplus} )^G$
with those $G$-equivariant right $R$-modules $M$
which are projective as $R$-modules.
\end{lemma}
\begin{proof}
Since $X$ is a strict fixed point,
the $G$-action restricts to the full subcategory
$\{X \} \subseteq \cC$
and consequently $G \rightarrow \End_{\kk}(R): g \mapsto (\alpha \mapsto T_g( \alpha ))$
defines a left $G$-action via $\kk$-algebra automorphisms on $R$.
Moreover, by the $\kk$-linearity of the $T_g$ for $g \in G$,
the action also restricts to $\Kar( \{ X \}^{\oplus} ) \subseteq \cC$.

Let $R_R$ denote the free  right $R$-module of rank $1$.
By \Cref{example:G_action_on_mod},
$T_g( R_R )$ has the same underlying $\kk$-vector space
as $R_R$, but $x \ast r := xT_{g^{-1}}(r)$
as its module structure for $r,x \in R$, $g \in G$.
The maps
\[
\varphi^g: R_R \rightarrow T_g( R_R ): x \mapsto T_{g^{-1}}(x) = 1 \ast x
\]
are isomorphisms of $R$-modules
such that the diagram
\begin{center}
        \begin{tikzpicture}
            \coordinate (r) at (6,0);
            \coordinate (d) at (0,-1.5);
            \node (A) {$T_g(R_R)$};
            \node (B) at ($(A)+(r)$) {$T_g( R_R )$};
            \node (C) at ($(A)+(d)$) {$R_R$};
            \node (D) at ($(C)+(r)$) {$R_R$};
            \draw[->,thick] (A) to node[above]{$x \mapsto \alpha x$} (B);
            \draw[->,thick] (C) to node[below]{$x \mapsto T_g(\alpha)x$} (D);
            \draw[->,thick] (C) to node[left]{$\varphi^g$} (A);
            \draw[->,thick] (D) to node[right]{$\varphi^g$} (B);
        \end{tikzpicture}
    \end{center}
commutes for all $\alpha \in R$ and $g \in G$.
In other words, the $\varphi^g$ define the components of
natural transformations that turn the fully faithful inclusion functor
\[
\{ X \} \rightarrow \rfproj{R}:
X \mapsto R_R
\]
into a $G$-equivariant functor.
By the universality of additive closures and Karoubian envelopes, we obtain the desired $G$-equivariance of
\[
\Kar( \{ X \}^{\oplus} ) \rightarrow \rfproj{R}:
Y \mapsto \Hom_{\cC}(X,Y).
\]
Now, \Cref{lemma:equivariant_equiv} yields the desired equivalence of categories of equivariant objects.
\end{proof}

\subsection{Galois descent for categories}\label{sec:galoisdecent}

In this subsection, we introduce Galois descent for categories.
The idea is to recover $\cC$ from $\cC^{\KK}$ as a category of $G$-equivariant objects,
where $G$ is the Galois group of a field extension $\kk \subseteq \KK$.
Galois descent allows us to generalize our investigation of $\DCob{{\alpha}}$ over a splitting field to an arbitrary field.

\begin{lemma}\label{lemma:gal_des}
  Let $\kk \subseteq \KK$ be a field extension, and let $G$ be a finite group acting on $\KK$ via $\kk$-algebra automorphisms.
  Then we have a faithful functor
  \[
  \cC \rightarrow ( \cC^\KK )^G: X \mapsto (X^{\KK}, (\id_{X^{\KK}})_{g\in G})
  \]
  which is full if $\kk \subseteq \KK$ is a finite Galois extension and $G=\Gal(\KK|\kk)$.
\end{lemma}
\begin{proof}
  Clearly, we have a well-defined functor.
  Moreover, we compute
  \begin{align*}
    \Hom_{( \cC^\KK )^G}( (X^\KK, (\id_{X^\KK})_g), (Y^\KK, (\id_{Y^\KK})_g) )
    &=
    \{
      \alpha \in \Hom_{\cC^{\KK}}( X^{\KK}, Y^{\KK}) \mid \text{$g(\alpha) = \alpha$ for all $g \in G$}
    \} \\
    &=  F \otimes_{\kk} \Hom_{\cC}( X, Y )
  \end{align*}
  where $\kk \subseteq F \subseteq \KK$ is the field fixed by the action of $G$. The claim follows.
\end{proof}

In the following, we let $G$ be the Galois group of a Galois extension $\KK$ of $\kk$. The group $G$ acts naturally on $\KK$ by automorphisms which fix $\kk\subseteq \KK$. 
The well-known theory of Galois descent states that
the scalar extension functor
\[
\KK \otimes_\kk - : \lvec{\kk} \rightarrow \lvec{(\KK \rtimes G)}
\]
defines an equivalence between finite-dimensional $\kk$-vector spaces and finite-dimensional $G$-equivariant $\KK$-vector spaces,
i.e., $\KK$-vector spaces $W$ such that
\[
g(aw) = g(a)g(w)
\]
for all $g \in G$, $a \in \KK$, $w \in W$,
with a quasi-inverse functor given by taking $G$-invariants.
The only non-trivial part of this equivalence is the essential surjectivity, which can be found in \cite{Bourbaki}*{Chapter V.62, Proposition 7}.

Moreover, since $(\KK \otimes_\kk -)$ 
actually defines an equivalence of \emph{monoidal} categories, for any finite-dimensional $\kk$-algebra $A$, we obtain a $G$-equivariant $\KK$-algebra
$R := \KK \otimes_{\kk} A$,
and an equivalence between categories of finite-dimensional modules
\[
\rfmod{A} \simeq \rfmod{( R \rtimes G)}
\]
that restricts to finitely generated projectives
\[
\rfproj{A} \simeq \rfproj{( R \rtimes G)}.
\]

\begin{lemma}\label{lemma:projectives_descent}
Let $\kk \subseteq \KK$ be a finite Galois extension with Galois group $G$, let $A$ be a finite-dimensional $\kk$-algebra,
and let $R := \KK \otimes_{\kk} A$
be the corresponding $G$-equivariant $\KK$-algebra.
Then $M \in \rfmod{( R \rtimes G)}$
is a projective object if and only if
$M$ is projective as an $R$-module.
\end{lemma}
\begin{proof}
By Galois descent, the statement is equivalent to the following:
$N \in \rfmod{A}$ is projective if and only if
$\KK \otimes_{\kk} N \in \rfmod{R}$ is projective.
Let $N$ be projective, i.e., a summand of $A^n$ for some $n \in \mN$. Then $\KK \otimes_{\kk} N$
is a summand of $R^n$.
Conversely, if $\KK \otimes_{\kk} N$
is a summand of $R^n$
for some $n \in \mN$, then
\[
N 
\hookrightarrow 
\KK \otimes_{\kk} N
\hookrightarrow
R^n
\cong_{A}
A^{[\KK:\kk]\cdot n}
\]
is a composition of split $A$-module monomorphisms and thus split.
\end{proof}

\begin{theorem}[Galois descent]\label{theorem:galois_descent}
  Let $\kk \subseteq \KK$ be a finite Galois extension with Galois group $G$,
and let $\cC$ be a $\kk$-linear hom-finite Krull--Schmidt category.
Then the functor
\[
  \cC \rightarrow ( \cC^\KK )^G: X \mapsto (X^{\KK}, (\id_{X^{\KK}})_g).
\]
is an equivalence of categories.
\end{theorem}
\begin{proof}
  By \Cref{lemma:gal_des} it is left to show essential surjectivity.
  Let $X \in \cC$,
  $A := \End_{\cC}( X )$,
  $R := \End_{\cC^{\KK}}( X^{\KK} ) = \KK \otimes_{\kk} A$.
  Then
  \begin{align*}
      \Kar( \{ X \}^{\oplus} )
      \simeq
      \rfproj{A}
      \simeq
      \rfproj{(R \rtimes G)}
  \end{align*}
  by Galois descent for finite-dimensional $\kk$-algebras,
  and
  \begin{align*}
      \rfproj{(R \rtimes G)}
      \simeq
      \{ M \in \rfmod{(R \rtimes G)} \mid \text{$M$ is a projective $R$-module} \}
  \end{align*}
  by \Cref{lemma:projectives_descent}.
  The right-hand side category in turn is equivalent to
  \begin{align*}
      (\rfproj{R})^G \simeq \Kar( \{ X^{\KK} \}^{\oplus} )^G
  \end{align*}
  by \Cref{lemma:KS_and_projectives_equiv}.
  Combining these equivalences yields the equivalence
  \[
  \cC \supseteq \Kar( \{ X \}^{\oplus} ) \xrightarrow{\sim} \Kar( \{ X^{\KK} \}^{\oplus} )^G: Y \mapsto (Y^{\KK}, (\id_{Y^{\KK}})_g).
  \]
  Since every $F \in ( \cC^\KK )^G$
  lies in $\Kar( \{ Y^{\KK} \}^{\oplus} )^G \subseteq ( \cC^\KK )^G$
  for some $Y \in \cC$, we obtain the desired essential surjectivity.
\end{proof}

\begin{corollary}\label{corollary:gr_in_invariants}
Let $\kk \subseteq \KK$ be a finite Galois extension with Galois group $G$,
and let $\cC$ be a $\kk$-linear hom-finite Krull--Schmidt category.
Then
$\Gr( \cC ) \subseteq \Gr( \cC^\KK )^G$,
and for each $X \in \cC^\KK $,
the invariant element $\sum_{g \in G} [T_g(X)]$
lies in $\Gr( \cC )$.
\end{corollary}
\begin{proof}
We have ring morphisms
\[
\Gr( \cC ) 
\cong
\Gr( (\cC^{\KK})^G )
\rightarrow
\Gr( \cC^{\KK} )^G
\hookrightarrow
\Gr( \cC^{\KK} ): [X] \mapsto [X^{\KK}]
\]
by \Cref{theorem:galois_descent}
and \Cref{remark:inclusion_Gro}.
On the other hand, we know that
$[X] \mapsto [X^{\KK}]$ is a monomorphism by
\Cref{lemma:mono_Gro}, which gives us
\[
\Gr( \cC ) \subseteq \Gr( \cC^\KK )^G.
\]
Moreover, 
each object $X \in \cC^\KK$
gives rise to a $G$-equivariant object with underlying object
$\bigoplus_{g \in G} T_g(X)$
by \Cref{example:orbit_sums}.
\end{proof}

Recall that $\Gr( \cC^\KK )^G$ is spanned by the orbit sums of the $G$-action on 
$\Indec( \cC^\KK )$, see \Cref{remark:invariant_orbit_sums}. In general,
the inclusion $\Gr( \cC ) \subseteq \Gr( \cC^\KK )^G$ is strict.

\begin{example}\label{example:quaternions}
  Let $\mH$ be the $\mR$-algebra of quaternions, and set $\cC := \lfmod{\mH}$.
  Since $\mH \otimes_{\mR} \mC \cong \mC^{2 \times 2}$, we have $\cC^{\mC} \simeq \lfmod{\mC^{2 \times 2}} \simeq \lfmod{\mC}$.
  Using these identifications on the level of Grothendieck groups, we obtain
  \[
    \Gr( \cC ) \rightarrow \Gr( \cC^{\mC}): [\mH] \mapsto 2[ \mC ].
  \]
  The Galois group $G$ of $\mC$ over $\mR$ acts trivially on $\Gr( \cC^{\mC})$, hence $\Gr( \cC^{\mC})^G = \Gr( \cC^{\mC})$.
  Thus, the inclusion
  \[
    \Gr( \cC ) \cong \langle 2[\mC] \rangle \subseteq \Gr( \cC^{\mC})^G = \langle [\mC] \rangle
  \]
  is strict.
\end{example}

However, we will prove that in the case of $\DCob{\alpha}$, the phenomenon of \Cref{example:quaternions} does not occur.
To this end, we first need to construct equivariant objects that correspond to orbit sums.

\subsection{Construction of equivariant objects}\label{sec:orbitsumobjects}

Let $\alpha$  be a sequence with coefficients in $\kk$ given by a rational function.
Let $\KK'$ be a splitting field for $u(t)=u_\alpha(t)\in \kk[t]$. Then $\KK'$ is normal over $\kk$. We choose $\KK$ to be the \emph{separable closure} (i.e.~the maximal separable subextension, which is uniquely determined) of $\kk$ in $\KK'$. Then $\KK$ is also normal over $\kk$ \cite{Lan}*{Corollary~V.6.8} and hence $\kk\subseteq \KK$ is a Galois extension. Further, $\KK\subset \KK'$ is a purely inseparable extension \cite{Lan}*{Proposition~V.6.6}.

 We want to  describe $\Gr(\Symu{\alpha})$ as Galois group invariants in the larger Grothendieck ring of $\cS_{u_{\alpha^\KK}}$. In this section, we construct equivariant objects corresponding to orbit sums.
\begin{remark}\label{rem:fieldsetup}
Every irreducible factor of $u(t)$ over $\KK$ is either linear or of the form $t^{p^n} - a$ for some $a \in \KK$, $n \geq 0$, $p$ the characteristic of $\KK$. If $b \in \KK'$ denotes a root of $t^{p^n} - a$, then $t^{p^n} - a = (t-b)^{p^n}$.
Moreover, we have an isomorphism $\Aut(\KK'|\kk)\cong \Aut(\KK|\kk)$ given by restriction along $\KK \subseteq \KK'$.
\end{remark}

Recall the setup and the idempotents $E_{\l_1,\dots,\l_r}$ from \Cref{def:Elambda}.
That is, we consider Young diagrams $\l_1,\dots,\l_r$ such that $\sum_{i=1}^r|\l_i|=n$ and primitive idempotents $e_{\l_i}$ in $\KK [S_{|\l_i|}]\subset \KK [S_n]$  as in \Cref{sec::Young}, and fix a  decomposition of $\KK[x]/(u(x))$ as an ordered direct sum of local algebras $L_1\oplus\dots\oplus L_r$. If $\KK=\KK'$, these correspond to an ordered list of distinct zeros of $u(x)$ but in general to irreducible factors (ignoring multiplicities) of $u(x)\in \KK[x]$. We denote the corresponding orthogonal idempotents projecting onto $L_i$ by  $E_i$. Slightly generalizing \Cref{def:Elambda}, we consider idempotents of the form 
\begin{align}\label{Elambdaperm}
    E_{\l_1,\ldots,\l_r}^\s:=E_{\s(1)}^{\o|\l_1|}\o\dots\o E_{\s(r)}^{\o|\l_r|}\otimes (e_{\l_1},\ldots , e_{\l_r})\,\in \, B^{\KK}=(\KK[x]/u(x))^{\otimes n}\rtimes \KK [S_n],
\end{align}
for $\s$ a permutation of $r$. %
We can extend the permutation $\s$ to a permutation $\beta(\s)$ of $S_n=S_{|\l_1|+\ldots+|\l_r|}$ by permuting the subgroups $S_{\{1,\ldots,|\l_1|\}}$, $S_{\{|\l_1|+1,\ldots,|\l_1|+|\l_2|\}}$, \ldots, $S_{\{n-|\l_n|,\ldots,n\}}$ block-wise, i.e.~through
$$\beta(\s)(|\l_1|+\ldots+|\l_k|+i)=|\l_{\s^{-1}(1)}|+\ldots+|\l_{\s^{-1}(k)}|+i, \qquad \text{for }1\leq i\leq |\l_{k+1}|.$$ 
Then conjugation $\beta(\s)^{-1}(-)\beta(\s)$ transforms the idempotent $E_{\l_1,\ldots,\l_r}^\s$
to the equivalent idempotent
\begin{equation} E_{\l_{\s^{-1}(1)},\ldots,\l_{\s^{-1}(r)}}= E_{1}^{\o|\l_{\s^{-1}(1)}|}\o\dots\o E_{r}^{\o|\l_{\s^{-1}(r)}|}\otimes (e_{\l_{\s^{-1}(1)}},\ldots, e_{\l_{\s^{-1}(r)}})\,\in \, B^\KK,
\end{equation}
matching the form used in \Cref{def:Elambda}.
We denote the isomorphism of $B^\KK$-modules given by the left action of $1\otimes \beta(\s)^{-1}\in B^\KK$ by 
$$\ord_\s \colon  E_{\l_1,\ldots,\l_r}^\s B^\KK\xrightarrow{\sim} E_{\l_{\s^{-1}(1)},\ldots,\l_{\s^{-1}(r)}}B^\KK.$$

As in the proof of \Cref{lemma:gal_des}, the Galois group $G=\Gal(\KK|\kk)$ acts on the $\KK$-algebra $B^\KK$ via $\KK$-algebra automorphisms. Thus, as in \Cref{example:G_action_on_mod}, $\rmod{B^\KK}$ obtains a strict $G$-action such that  $(\rmod{B^\KK})^G\simeq \rmod{B^\KK\rtimes G}$.

\smallskip

The action of $g\in G$ on $B^\KK$ permutes the idempotents $E_i$ by virtue of permuting the summands $L_1,\ldots, L_r$ of $\KK[x]/u(x)$ but does not change the idempotents $e_{\l_i}\in \KK [S_n]\subset B^\KK$. As $G$ acts faithfully on the set of idempotents $E_i$, we regard $G$ as a subgroup of $S_r$ and write $g(E_{i})=E_{g(i)}$. Consider the action of $g$ on $EB^{\KK}$, where $E=E_{\l_1,\ldots,\l_r}^\s$ is an idempotent of the form \eqref{Elambdaperm}. Then $T_g(E B^{\KK})=g(E)B^{\KK}$, where $g(E)=E_{\l_1,\ldots,\l_r}^{g \s}$.

\begin{lemma}\label{lem:ordmultiplicative}
The isomorphisms $\ord_\s$ are compatible with products in $G$ in the sense that 
$$\ord_g \circ T_g(\ord_\s)=\ord_{g\s} \colon  E_{\l_1,\ldots,\l_r}^{g\s}\xrightarrow{\sim} E_{\l_{\s^{-1}(1)},\ldots,\l_{\s^{-1}(r)}}^g\xrightarrow{\sim} E_{\l_{\s^{-1}g^{-1}(1)},\ldots,\l_{\s^{-1}g^{-1}(r)}}.$$
\end{lemma}

We want to construct a $G$-equivariant objects giving the $G$-orbit sum of elements of $\Gr(\rfproj{B^\KK})\subset \Gr(\Symu{\alpha^\KK})$. To this end, we will construct $G$-equivariant $B^\KK$-modules. 

\begin{lemma}\label{lemma:stabilizer}
Given an idempotent $E=E_{\l_1,\ldots,\l_r}$ with corresponding projective module $V=E B^\KK$, we observe that $G_E:=\Stab_{[V]}=\Stab_{(\l_1,\ldots, \l_r)}$, where the former is the stabilizer of the $G$-action on $[V]\in \Gr(\Symu{\alpha^\KK})$ and the latter is the stabilizer of $G$ acting on the tuple $(\l_1,\ldots,\l_r)$ by permutation according to $G\hookrightarrow S_r$. 
\end{lemma}
\begin{proof}
This statement follows, as the idempotents $E_{\l_1,\ldots,\l_r}$ give non-isomorphic modules for distinct tuples $(\l_1,\ldots,\l_r)$ using \Cref{cor::idempotents-concrete} and \Cref{lemma:insep_case}.
\end{proof}

We now choose a set of $G_E$-coset representatives $g_1,\ldots, g_s$ such that $G=\sqcup_{i=1}^s g_iG_E$ and define 
\begin{equation}\label{def:F}
F:= \bigoplus_{i=1}^s E_{g_i} B^\KK, \qquad \text{where}\qquad  
E_\s=E_{\l_{\s^{-1}(1)},\ldots, \l_{\s^{-1}(r)}},\qquad \text{ for $\s\in S_r$}.
\end{equation}
With the above observations we compute for $g\in G$
$$
T_g(F)=\bigoplus_{i=1}^s T_g(E_{g_i}B^\KK)=\bigoplus_{i=1}^s g(E_{g_i})B^\KK=\bigoplus_{i=1}^s E_{g_i}^gB^\KK.
$$
For each $i$, we have isomorphisms 
$$\ord_g\colon E_{g_i}^g B^\KK\isomorph E_{gg_i}B^\KK.$$
As $\{g_1,\ldots, g_s\}$ was chosen to be a set of $G_E$-coset representatives, the idempotents $\{E_{gg_1},\ldots, E_{gg_s}\}$ are a rearrangement of the idempotents $\{E_{g_1},\ldots, E_{g_s}\}$ following the rule that $E_{gg_i}=E_{g_j}$ if and only if $gg_i\in g_j G_E$. We write $j=\ov{g}(i)$ in this case and observe that $\ov{gh}=\ov{g}\ov{h}$. Thus, we define isomorphisms 
$$\alpha_g:=\Pi_{\ov{g} }(\ord_{g},\ldots, \ord_{g})\colon T_g(F)\isomorph F,$$
where $\Pi_{\ov{g}}$ is the permutation matrix permuting the direct summands in $F$ according to index changes $i\mapsto \ov{g}(i)$.

\begin{proposition}\label{lem:FisGequivariant}
The pair $(F,(\alpha_g)_g)$ defines a $G$-equivariant object in $\rmod{B^\KK}$.
\end{proposition}
\begin{proof}
We observe that 
$$\alpha_g=\Pi_{\ov{g} }(\ord_{g},\ldots, \ord_{g})=(\ord_{g},\ldots, \ord_{g})\Pi_{\ov{g} }.$$
As we have observed  that $\ov{gh}=\ov{g}\ov{h}$, it follows that $\Pi_{\ov{gh} }=\Pi_{\ov{g} } \Pi_{\ov{h} }$ and using \Cref{lem:ordmultiplicative}, we have that $\alpha_{gh}=\alpha_g\circ\alpha_h$. Using the description of the data of a $G$-equivariant object in $\rmod{B^\KK}$ from \Cref{example:G_action_on_mod}, the claim follows. 
\end{proof}

\begin{corollary}\label{corollary:constructed_orbit_sum_obj}
 For the $G$-equivariant object $F$ from \Cref{lem:FisGequivariant}, $[F]$ is the orbit sum of $[V]$ in $\Gr(\rmod{B^\KK})^G$, for the object $V=EB^\KK$. 
\end{corollary}
\begin{proof}
Using the classification from \Cref{cor::idempotents-concrete}, the direct summands that $F$ comprises are pairwise non-isomorphic indecomposable objects of $\rmod{B^{\KK'}}$. By \Cref{lemma:insep_case}, the idempotents used are already defined over $\KK$ and $F$ is a direct sum of pairwise non-isomorphic indecomposable objects over $\rmod{B^{\KK}}$. We observed in \Cref{lemma:stabilizer} that the stabilizer $G_E$ is the stabilizer of $[V]$ under the $G$-action on $\Gr(\rmod{B^\KK})$ given by $g[V]=[T_g(V)]$. Thus, the orbit of $[V]$ is given precisely by $\{g_1[V],\ldots, g_s[V]\}$ and the sum over these symbols equals $[F]$.
\end{proof}

\subsection{The Grothendieck ring of \texorpdfstring{$\DCob{\alpha}$}{DCobalpha} for general fields}
We retain the setup from \Cref{sec:orbitsumobjects}, where $\alpha$ is a rational function with coefficients in $\kk$ and $\kk \subseteq \KK\subseteq \KK'$ are field extensions where $\kk \subseteq \KK$ is Galois and $\KK\subseteq \KK'$ is purely inseparable.

\begin{remark}\label{remark:scalar_ext_of_dcob}
We denote by $\alpha^{\KK}$ the sequence $\alpha$ with entries regarded as elements in $\KK$ and observe that $(\Symu{\alpha})^{\KK} \simeq \Symu{\alpha^{\KK}}$.
\end{remark}

\begin{proposition}\label{prop:grDCob-gen}
For $G=\Gal(\KK|\kk)$, there is an equivalence of monoidal categories 
$$\gr \DCob{\alpha}\simeq \Symu{\alpha^{\KK}}^{~G}$$
compatible with gradings.
\end{proposition}
\begin{proof}
The equivalence follows by composing the equivalences 
$$\gr \DCob{\alpha}\simeq \Symu{\alpha}\simeq \Symu{\alpha^\KK}^{~G}$$
from \Cref{thm:KoDCob} and \Cref{theorem:galois_descent}.
\end{proof}

\begin{theorem}\label{theorem:Gr_compatible_with_invarinats}
  Let $G$ be the Galois group of $\KK$ over $\kk$.
  Then
  \[
  \Gr( (\Symu{\alpha^{\KK}})^G ) = \Gr( \Symu{\alpha^{\KK}} )^G.
  \]
  \end{theorem}
  \begin{proof}
    By \Cref{corollary:gr_in_invariants}
  and \Cref{remark:scalar_ext_of_dcob} we have 
  \[
  \Gr( (\Symu{\alpha^{\KK}})^G ) \subseteq \Gr( \Symu{\alpha^{\KK}} )^G.
  \]
  Conversely, by \Cref{corollary:constructed_orbit_sum_obj}, we can construct for each orbit sum a $G$-equivariant object
  and thus get the desired equality by \Cref{remark:invariant_orbit_sums}.
  \end{proof}

\begin{lemma}\label{lemma:Gr_purely_insep}
    The natural functor $\Symu{\alpha^{\KK}} \rightarrow \Symu{\alpha^{\KK'}}$ 
    induces an isomorphism of graded rings
    \[
    \Gr( \Symu{\alpha^{\KK}} ) \cong \Gr( \Symu{\alpha^{\KK'}} ).
    \]
\end{lemma}
\begin{proof}
  This is a consequence of \Cref{lemma:insep_case}.
\end{proof}

\begin{theorem}[Main classification theorem] \label{theorem:main_classification}
Let $\kk$ be a field of characteristic $p$ (possibly zero) and $\KK'$ a splitting field of $u_\alpha\in\kk[x]$. Denote $G=\Aut(\KK'|\kk)$, the automorphism group of $\KK'$ fixing $\kk$, and 
  let $Z$ be the $G$-set of (distinct) zeros of $u_{\alpha}$ in $\KK'$.
  Then we have a graded ring homomorphism
  \[\gr( \Gr( \DCob{\alpha} ) ) \cong \bigg(\bigotimes_{z \in Z} \Sym^p\bigg)^G,\]
  where $G$ acts on $(\bigotimes_{z \in Z} \Sym^p)$ by permutation of the factors.
\end{theorem}
\begin{proof}
Recall the setup from \Cref{rem:fieldsetup}. In particular, we can identify  $G$ with $\Gal(\KK|\kk)$ through restriction.
We need to collect our previous results:
\begin{align*}
    \gr( \Gr( \DCob{\alpha} ) ) %
    &\cong \Gr( (\Symu{\alpha^{\KK} })^G ) ~& %
    \text{\Cref{prop:grDCob-gen} (using Galois descent)} \\
    &\cong \Gr( \Symu{\alpha^{\KK} } )^G ~& \text{\Cref{theorem:Gr_compatible_with_invarinats}} \\
    &\cong \Gr( \Symu{\alpha^{\KK'} } )^G ~& \text{\Cref{lemma:Gr_purely_insep}} \\
    &\cong \bigg(\bigotimes_{z \in Z} \Sym^p\bigg)^G ~& \text{\Cref{corollary:main_thm_split_case}}
\end{align*}
It is easy to see that the $G$-action corresponds to a permutation of the factors.
\end{proof}

\subsection{Examples}
\label{sec:nonsplitting examples}

To demonstrate the main result \Cref{theorem:main_classification} we provide some examples.

\begin{example}
Consider an irreducible polynomial 
$$u_\alpha(t)=\beta_0-\beta_1t +t^2, \qquad \text{with }\beta_0,\beta_1\in \mR,$$
of degree two.
Such a polynomial $u_\alpha$ arises from rational functions of the form
$$\alpha(t)=\frac{c_0+c_1 t}{1-\beta_1t+\beta_0 t^2}\quad \in\quad  \mR[[t]].$$
Using the splitting field $\mC\supset \mR$ we have $G=\Gal(\mC|\mR) \cong \mZ/2\mZ$, 
$$u_{\alpha^\KK}(t)=(t-\rho)(t-\ov{\rho}), \qquad \text{for some }\rho\in \mC\setminus \mR,$$
and the unique generator $\sigma$ of $G$ acts by complex conjugation $\sigma\rho=\ov{\rho}$.

Using the results from \Cref{sec:classification}, the associated graded category $$\gr\DCob{\alpha^\mC}\simeq \cS_{(t-\rho)(t-\ov{\rho})}$$ is given by projective modules over the algebras
$$R_n:= \left(\mC[x]/(x-\rho)\otimes \mC[x]/(x-\ov{\rho})\right)^{\otimes n} \rtimes \mC S_n, \qquad n\geq 0.$$
The indecomposable objects of $\DCob{\alpha}$ of degree $n$ in the filtration are parametrized by pairs $(\lambda_1,\lambda_2)$ of Young diagrams with $n=|\lambda_1|+|\lambda_2|$. From \Cref{theorem:main_classification} we know that
$$\gr\Gr(\DCob{\alpha})\cong \left( \Sym\otimes \Sym \right)^{\mZ/2\mZ},$$
where $G=\mZ/2\mZ$ acts by swapping the tensor factors $\Sym$. Hence, the induced $G$-action on $\Gr (\DCob{\alpha})$ is given by
$$\sigma\cdot [X_{\lambda_1,\lambda_2}]=[X_{\lambda_2,\lambda_1}]$$
for any indecomposable object $X_{\lambda_1,\lambda_2}$.

This implies that the objects $[X_{\lambda,\lambda}]$ are contained in $\gr\Gr(\DCob{\alpha})$. The other elements are the orbit sums 
$$[X_{\lambda_1,\lambda_2}]+[X_{\lambda_2,\lambda_1}],\qquad \text{for } \lambda_1\neq \lambda_2.$$
These symbols correspond to the orbit sum objects constructed in \Cref{lem:FisGequivariant}.

\end{example}

\begin{example}
\Cref{theorem:main_classification} allows us to recover our findings of \Cref{subsection:a_first_example}. In the case $u_{\alpha} := x^2 + 1$ over $\mR$,
the second degree part of $\gr( \Gr( \DCob{\alpha} ) )$
is freely generated by the three elements
$\Grel{(0)} \otimes \Grel{(2)} + \Grel{(2)} \otimes \Grel{(0)}$,
$\Grel{(0)} \otimes \Grel{(1,1)} + \Grel{(1,1)} \otimes \Grel{(0)}$, and
$\Grel{(1)} \otimes \Grel{(1)}$.
\end{example}

\begin{example}
If $u_{\alpha}$ is irreducible, it does not suffice to simply know its number of zeros over a splitting field in order to determine $\gr( \Gr( \DCob{\alpha} ) )$.
For example, let $u_{\alpha_1} := x^3 - x + 1$ and $u_{\alpha_2} := x^3 -3x + 1$.
Then both polynomials are irreducible over $\mQ$, but the Galois group $G_1$ of $u_{\alpha_1}$ is $S_3$, while the Galois group $G_2$ of $u_{\alpha_2}$ is the cyclic group $C_3$ \cite{Lan}*{Example~VI.§2.2}.
It follows from \Cref{theorem:main_classification} that we may regard both $\gr( \Gr( \DCob{\alpha_1} ) )$ and $\gr( \Gr( \DCob{\alpha_2} ) )$ as subrings of
$(\bigotimes_{i = 1}^3 \Sym)$.
It also follows that regarded as such subrings, we have a strict inclusion
\[
\gr( \Gr( \DCob{\alpha_1} ) ) \subsetneq \gr( \Gr( \DCob{\alpha_2} ) ).
\]
For example, the element
\[\Grel{(0)} \otimes \Grel{(1)} \otimes \Grel{(2)} + \Grel{(2)} \otimes \Grel{(0)} \otimes \Grel{(1)} +
\Grel{(1)} \otimes \Grel{(2)} \otimes \Grel{(0)}\]
lies in $\gr( \Gr( \DCob{\alpha_2} ) )$
but not in $\gr( \Gr( \DCob{\alpha_1} ) )$.
\end{example}

\section{Special cases and examples}\label{sec:examples}

In \Cref{sec:grRepSt} we derived a classification of indecomposable objects in $\uRep_{\kk}(S_T)$ where $\kk$ is any field, generalizing \cite{Del}*{Proposition~5.11}. In \Cref{expl:RepStasDCob} we recalled the result of \cite{KS} that, for $T\neq 0$, $\uRep_\kk(S_T)$ arises as the special case of the categories $\DCob{\alpha}$ with $\alpha=(T,T,\ldots)$. In this section, we specify our results to other classes of examples associated to certain rational series $\alpha$. In addition, we provide a slight generalization of $\DCob{\alpha}$ allowing the polynomial $u_\alpha$ which governs the endomorphisms of $[1]$ to be inflated by a polynomial factor in \Cref{sec:inflations}. For example, $\uRep_\kk(S_0)$ appears as an example of the inflated construction. 

\subsection{Representations of the orthosymplectic Lie superalgebra \texorpdfstring{$\mathfrak{osp}(1|2)$}{osp(1|2)}}
\label{sect:Liesup}
Assume that $\cha \kk=0$ during this subsection.
The Lie superalgebra $\mathfrak{osp}(1|2)$ is defined as the Lie super-subalgebra of $\mathfrak{gl}(1|2)$ of matrices of the form
$$\begin{pmatrix}
0&a&b\\b&c&d\\-a&e&-c
\end{pmatrix},$$
see e.g. \cite{Sta}*{Section~4}.
We denote by $\Rep \mathfrak{osp}(1|2)$ the category of finite-dimensional modules over $\mathfrak{osp}(1|2)$. 
We regard this category as a $\kk$-linear category with morphisms that preserve the $(\mZ/2\mZ)$-grading.
The universal enveloping algebra $U(\mathfrak{osp}(1|2))$ is a primitively generated Hopf superalgebra, all $x\in \mathfrak{osp}(1|2)$ are primitive, i.e.~$\Delta(x)=x\otimes 1+1\otimes x$, $\varepsilon(x)=0$. This gives $\Rep \mathfrak{osp}(1|2)$ the structure of a symmetric (abelian) tensor category.  

We denote by $V$ the standard representation, of superdimension $(1|2)$, of $\mathfrak{osp}(1|2)$ given by the action on column vectors.
For every integer $i\geq 0$ there is an irreducible representation $V_i$ over $\mathfrak{osp}(1|2)$, which is the unique $(2i+1)$-dimensional irreducible summand of $V^{\o i}$. Tensoring with $\kk^{0|1}$, the unique representation of superdimension $(0|1)$, gives an autoequivalence (merely as an abelian category) of $\Rep \mathfrak{osp}(1|2)$ and we denote $\ov{W}=\kk^{0|1}\o W$ for any $W\in \Rep \mathfrak{osp}(1|2)$.

We denote the Karoubian tensor subcategory of $\Rep \mathfrak{osp}(1|2)$ generated by the object $V=V_1$ by $\Rep^+\mathfrak{osp}(1|2)$. This category is semisimple (see e.g.~\cite{Sta}*{Theorem~4.1}) and contains the irreducible representations $V_i$ but not the irreducible representations $\ov{V_i}$.

\begin{remark} \label{rem::rep-osp} In $\Rep^+\osp(1|2)$,
$$ V_i \o V_j \cong V_{|i-j|}\oplus V_{|i-j|+1}\oplus\dots\oplus V_{i+j}
.
$$
This implies
$V^{\o k}\cong V_k\oplus Y$, where $Y$ is a direct sum of indecomposable objects each of which occurs as a direct summand in $V^{\o(k-1)}$, for all $k\geq 1$. 
Even more concretely, as $V$ is generated by an odd highest weight vector $v$, $V_k$ is the submodule of $V^{\o k}$ which is generated by $v^{\o k}$. In particular, $V_k$ is a direct summand in the subspace of alternating tensors (in the super sense). In terms of Schur functors, this space of alternating tensors is exactly $F^\lambda(V)$ for $\l=1^k$, i.e., $\lambda$ is a single column. From the tensor product decomposition above, it follows that $V_k$ is the unique such summand which is not a direct summand of $V^{\o i}$ for any $i<k$.
\end{remark}

Next, we turn to the representation category of the orthosymplectic supergroup $\OSp(1|2)$.
We recall that the category $\Rep(\OSp(1|2))$ can be described using the super Harish--Chandra pair\footnote{A \emph{super Harish--Chandra pair} is a pair $(G_0,\g)$ where $G_0$ is an (even) Lie group acting algebraically on a Lie superalgebra $\g$ whose even part $\g_0$ is the Lie algebra of $G_0$ such that the differential of the action of $G_0$ on $\g$ is the adjoint action of $\g_0$ on $\g$ (see \cite{Ser}*{Section~3}).} with $\OSp(1|2)_0=\sfO(1)\times\Sp(2)=\{1,-1\}\times\SL(2)$, which can be thought of as block-diagonal matrices. It is then the category whose objects are simultaneously finite-dimensional $\osp(1|2)$-modules and $\OSp(1|2)_0$-modules such that the two actions of $\osp(1|2)_0$ obtained by restricting the Lie superalgebra action or by differentiating the Lie group action coincide. Morphisms in this category have to respect both module structures.

For the orthosymplectic supergroup $\OSp(1|2)$, we fix $\epsilon\in\OSp(1|2)_0$ to be the diagonal matrix with entries $(1,-1, -1)$ and denote by $\Rep(\OSp(1|2),\epsilon)$ the category of algebraic representations of $\OSp(1|2)$ such that the action of $\epsilon$ is given by the parity endomorphism (i.e., the involution which is the identity on the even part and multiplication with $-1$ on the odd part of a superspace), see \cites{Ser,CH}. It can be verified that $\Rep(\OSp(1|2),\epsilon)$ is generated as a Karoubian tensor category by the $(1|2)$-dimensional defining $\OSp(1|2)$-module, which we also denote $V$, but does not contain $\overline{V}$. Hence, there is a faithful (forgetful) functor $I:\Rep(\OSp(1|2),\epsilon)\to\Rep^+\osp(1|2)$ sending $V\mapsto V$. It is known that $I$ is not full, see for instance \cite{LZ}*{Sec.~5} or \Cref{rem::I-not-full} below.

\subsection{Constant generating functions}\label{sec:constantcase}

In this section, we specify our results to the case of $\DCob{c}$ for $c\in\kk\setminus\{0\}$ and identify the images of indecomposable objects under the semisimplification functor. To this end, we recall that the semisimplification is given by $\Rep^+ \mathfrak{osp}(1|2)$ by a theorem of \cite{KKO}. We further clarify the relationship of $\DCob{c}$ to Deligne's category $\uRep(\sfO_{-1})$ and its semisimplification.

The category $\DCob{c}$, for $c\in \kk\setminus \{0\}$ a constant---which corresponds to the sequence $\alpha=(c,0,0,\ldots)$---is subject to the relations
\begin{align} \label{eq::relations-dcob-c}
    x&=\sm\sdel=0, & s_0&=\scap\scup=c,
\end{align}
i.e., surfaces of positive genus evaluate to $0$, and the genus zero surface evaluates to the constant $c$.
Thus, a basis for $\Hom_{\DCob{c}}([n],[m])$ is given by the set of partitions of $\{1,\ldots, n,1',\ldots, m'\}$.

The relations in \Cref{eq::relations-dcob-c} imply that the object $[1]$ has dimension $0$ in $\DCob{c}$, while at the same time, $e:=\tfrac1c \scup \scap$ is an idempotent endomorphism of $[1]$ with trace $1$. Hence, $e$ decomposes $[1]$ into a direct sum of two objects, $([1],e)$ and $([1],1-e)$, of dimension $1$ and $-1$, respectively.

Within the following proposition, we collect known facts from the literature and present them in a unified diagram.
\begin{proposition} \label{prop::commutative-diagram-osp} Let $\kk$ be a field of characteristic zero and $c\in \kk\setminus \{0\}$.
We have the following commutative diagram of symmetric Karoubian tensor categories
\begin{align}
    \xymatrix{
    &\uRep (\sfO_{-1})\ar@{->>}[ld]\ar@{->>}[d]\ar@{^{(}->}[rr]^{\un{I}}&& \DCob{c}\ar@{->>}[d]^{Q}\ar@{->>}[rd]&\\
    \uRep (\sfO_{-1})/\cN\ar@{=}[r]^-{\sim}&\Rep (\OSp(1|2),\epsilon)\ar@{->}[rr]^{I}&& \Rep^+ \mathfrak{osp}(1|2)\ar@{=}[r]^-{\sim}&\DCob{c}/\cN.
    }
\end{align}
Here, the horizontal arrows are faithful monoidal functors, the downward arrows are full and essentially surjective monoidal functors, and $\cN$ denotes the ideal of negligible morphisms.
\begin{itemize}
    \item The functor $\un{I}$ is uniquely determined by the universal property of $\uRep (\sfO_{-1})$ by sending the generating object $[1]$ to the self-dual object $([1],1-e)$, where 
    \begin{equation}
        e:= \frac{1}{c}\scup\scap = \frac{1}{c}\,
        \vcenter{\hbox{
\begin{tikzpicture}[
  tqft,
  every outgoing boundary component/.style={draw=black,fill=black!20},
  every incoming boundary component/.style={draw=black, fill=black!20},
  every lower boundary component/.style={draw},
  every upper boundary component/.style={draw},
  cobordism/.style={fill=black!25},
  cobordism edge/.style={draw},
  view from=incoming,
  cobordism height=2cm,
]
\begin{scope}[every node/.style={rotate=0}]
\pic[tqft/cup,
  at={(0,0)}
  ];
  \pic[tqft/cap,
  at={(0,0.5)}
  ];
\draw node[label=$1'$](incoming boundary 1) at (0,0) {};
\draw node[label=$1$](incoming boundary 1) at (0,-2.4) {};
\end{scope}
\end{tikzpicture}}}\, .
    \end{equation}
    \item The functor $I$ is the canonical embedding from \Cref{sect:Liesup}.
    \item The  functor $\uRep  (\sfO_{-1})\to\Rep(\OSp(1|2),\epsilon)$ is obtained by the universal property of $\uRep (\sfO_{-1})$ by sending $[1]$ to the self-dual object $V$ in $\Rep(\OSp(1|2),\epsilon)$  and the kernel of this functor consists of the ideal of negligible morphisms $\cN$ \cite{Del}*{Th\'eor\`eme~9.6}. Thus, the category $\Rep(\OSp(1|2),\epsilon)$ is the semisimplification of $\uRep(\sfO_{-1})$, see \cite{EO}*{Section~2.2}.
    \item The functor $Q$ is obtained by sending $[1]$ to the Frobenius algebra $A=V_0\oplus V_1$ %
    defined in \cite{KKO}*{Section 5.5} (note that those objects are called $V_0$ and $V_2$ in loc.cit.) via the universal property stated in \Cref{prop:universal}. It was shown in \cite{KKO}*{Theorem~5.5} that $\Rep^+ \mathfrak{osp}(1|2)$ is indeed equivalent to the semisimplification of $\DCob{c}$.
\end{itemize}
\end{proposition}

\begin{remark} \label{rem::I-not-full} On a side note, one can use the commutative diagram in \Cref{prop::commutative-diagram-osp} to see that the functor $I$ is not full as follows: We observe that the endomorphism space between $[0]$ and $[3]$ in $\uRep(\sfO_{-1})$ is zero, as there are no Brauer diagrams with an odd number of points (see \Cref{sec::Delignes-categories}). If $I$ was full, this would imply that all morphisms in $\DCob{c}$ between $[0]$ and $([1],1-e)^{\o 3}$ are negligible. This is not the case: let $s,s'$ be cobordisms with one component of genus $0$, where $s$ has $0$ incoming and $3$ outgoing boundary circles, and $s'$ has $3$ incoming and $0$ outgoing boundary circles. Then a short computation using that any closed surface of positive genus is zero in $\DCob{c}$ shows that $s'(1-e)^{\o3} s= 2c^{-1} \neq0$. Hence, $(1-e)^{\o3}s$ is not negligible, and $I$ cannot be full.
\end{remark}

Our classification of indecomposable objects of $\DCob{\alpha}$ specifies to the following result.
\begin{corollary}\label{prop:constant-irrep}
Let $\kk$ be an arbitrary field, $c\in \kk\setminus\{0\}$. 
A complete list of  indecomposable objects of $\DCob{c}$ is given by $X_\l$    , where $\lambda$ corresponds to an irreducible representation of $S_n$ over $\kk$. The object $X_\lambda$ is a direct summand of $[n]$.
\end{corollary}
\begin{proof}
We observe that $u_c(t)=t$ in the constant case of $\alpha=c=(c,0,0,\ldots)$. Thus, for any field, $u_c$ is a linear polynomial with unique root $0$. By \Cref{prop::indecomposables}, the indecomposable objects are in bijection with $\bigsqcup_{n\geq 0}\Xk{S_n}$. 
\end{proof}

The following is a direct consequence of \Cref{thm:KoDCob} as any field $\kk$ is a splitting field for the polynomial $u_c(t)=t$.

\begin{corollary}\label{prop:constant-grK0}
Let $\kk$ be a field of characteristic $p \geq 0$ and $c\neq 0$. Then there is an  equivalence of symmetric monoidal categories
$$\gr (\DCob{c})\simeq \cS_{t}=\bigoplus_{n\geq 0} \rfproj{\kk[S_n]}.$$
Consequently, there is an isomorphism of graded commutative rings
$$\gr \Gr(\DCob{c})\cong \Sym^p.$$
\end{corollary}

Now assume that $\cha \kk=0$. Recall that the isomorphism classes of irreducible objects in $\Rep(\osp(1|2))$ are given by $\{V_k\}_{k\geq 0}$, where $V_k$ is the unique indecomposable $\osp(1|2)$-submodule of $V^{\o k}$ which does not occur in $V^{\o(k-1)}$.

\begin{proposition} The semisimplification functor $Q\colon\DCob{c}\to\Rep^+\osp(1|2)$ (for $c\neq0$) sends $X_\l$ to an object isomorphic to $V_k$ if $\l=1^k$ for some $k\geq 0$, and to the object $0$ otherwise.
\end{proposition}

\begin{proof} Let $n:=|\l|$. By \Cref{prop::Flambda}, $X_\lambda$ is the unique indecomposable object occurring as a direct summand in $F^{\lambda}([1])$ which does not occur in any of the objects $[0],\dots,[n-1]$.  Since we are in characteristic $0$, $F^\lambda$ is just the ordinary Schur functor, and as $Q$ is a symmetric monoidal functor sending $[1]\mapsto V_0\oplus V_1$, $Q(X_\lambda)$ is an indecomposable direct summand of $F^\lambda(V_1)$ not occurring in $V_1^{\o 0},\dots,V_1^{\o(n-1)}$, or zero by \Cref{lemma:ks_indec}. However, in $\Rep^+\osp(1|2)$, $F^\lambda(V_1)$ is a direct summand of $V_1^{\o n}$, so only one object of the form $F^\l(V_1)$ contains an indecomposable summand which does not occur in $V_1^{\o(n-1)}$, namely the object for $\l=1^n$ by \Cref{rem::rep-osp}. This proves the assertion.
\end{proof}

\subsection{Polynomial generating functions}\label{sec:polynomialcase}
Let $\kk$ be a field. In this section we consider the case when
$$\alpha(t)=\beta_0+\beta_1 t+\ldots + \beta_nt^n \in \kk[t]$$
is a polynomial, i.e.~$\alpha=\alpha_0$ in with respect to the partial fraction  decomposition of \eqref{eq:partialdec}. In this case, we find 
$u_{\alpha}(t)=t^{n+1}$, where $n=\deg \alpha$.

\begin{corollary}
For any field $\kk$, and $\alpha\in \kk[t]$, a full list of irreducible objects of $\DCob{\alpha}$ is given by $X_\lambda$, where $\lambda$ corresponds to an irreducible representation of $S_n$ over $\kk$.
\end{corollary}
\begin{proof}
As $u_\alpha(t)=t^{n+1}$ factors over any field and has a unique zero, the corollary follows from \Cref{prop::indecomposables}.
\end{proof}
As a direct consequence of \Cref{thm:KoDCob} we record the following.
\begin{corollary}
For $\alpha\in \kk[t]$, there is an equivalence of symmetric monoidal categories 
$$\gr(\DCob{\alpha})\simeq \cS_{t^{n+1}}$$ which induces
 an isomorphism of graded rings $$\gr\Gr(\DCob{\alpha})\cong \Sym^p,$$ 
 where $\cha \kk =p$.
\end{corollary}

\subsection{Inflations of \texorpdfstring{$\DCob{\alpha}$}{DCob(alpha)}}\label{sec:inflations}

The following class of $\kk$-linear tensor categories generalizes the class of monoidal categories $\DCob{\alpha}$, for $\kk$ a field.
\begin{definition}
Let $\alpha(t)$ be a rational function as in \Cref{sect:DCob} and let $f(t)$ be a monic polynomial. In analogy with \Cref{def:DCob}, we define $\SCob{\alpha,f}$ to be the quotient category of $\kk \Cob{2}$ by the ideal generated under $\kk$-linear combinations, two-sided composition, and tensor product by the relations 
\begin{align}
s_i&=\alpha_i \id_{[0]},\quad \forall i\geq0, & f(x)u_\alpha(x)&=0,
\end{align}
and $\DCob{\alpha,f}$ to be the Karoubian envelope of $\SCob{\alpha,f}^\oplus$.
We call $\DCob{\alpha,f}$ the \emph{inflation} of $\DCob{\alpha}$ by $f$, as $\DCob{\alpha}=\DCob{\alpha,1}$.
\end{definition}

The symmetric tensor categories just defined are intermediate quotients 
$$\kk \Cob{2}\twoheadrightarrow \DCob{\alpha,f}\twoheadrightarrow \DCob{\alpha}$$
that clearly have finite-dimensional morphism spaces. 

The results from \Cref{sec:K0DCob} can be extended to descriptions of the associated graded category of the inflation categories $\DCob{\alpha,f}$. To this end, recall the categories $\cS_u$ from \Cref{definition:symu}, which depend on a monic polynomial $u=u(t)$. The categories $\DCob{\alpha,f}$ have $[1]$ as a tensor generator in the sense of \Cref{sec::tensor-generator} and, hence, a filtration, similarly to $\DCob{\alpha}$, where $\cD_n$ consists of objects appearing as direct summands of sums of copies of $[n]$.

\begin{proposition} There is an equivalence of symmetric monoidal categories
$$\gr (\DCob{\alpha, f}) \simeq \cS_{u_{\alpha}{f}},$$
where $u_\alpha$ is associated to $\alpha$ as in \eqref{eq:ualpha}. In particular, if $\kk$ is a splitting field for $u_\alpha f$ with $\cha \kk=p$, and $u_\alpha f$ has $l$ distinct roots in $\kk$, then we obtain an isomorphism of graded rings
$$\gr \Gr (\DCob{\alpha, f}) \cong (\Sym^p)^{\otimes l}.$$
\end{proposition}
\begin{proof} As in \Cref{prop::quotient-is-crossed-product}, we find that
  \begin{align*}
    \End_{\cD_n/\cD_{n-1}}([n]) \cong \widetilde P_n\rtimes S_n, && \text{for $n\geq 0$},
  \end{align*}
where  $\widetilde P_n=\kk[x_1,\ldots, x_n]/(u_\alpha(x_1)f(x_1), \ldots, u_\alpha(x_n)f(x_n))$.
This implies the equivalence of symmetric monoidal categories by \Cref{prop:tensor-generator}. The remaining assertions follow from the same arguments used in the proof of \Cref{corollary:main_thm_split_case}.
\end{proof}

The inflated categories $\DCob{\alpha,f}$ are useful for extending families of interpolation categories to the special value $\alpha=0$ in a natural way. We give two examples of this idea.

\begin{example}
Deligne's category $\uRep S_0$ is equivalent to  $\DCob{0,f}$, where $f(t)=t-1$. Indeed, for $\alpha(t)=0$, we have $u_\alpha(t)=1$. Inflating with $f(t)=t-1$ gives that the genus $k$ morphism associated to a partition of $n+m$ reduces to the same partition with genus $0$.
\end{example}

\begin{example}
The categories $\DCob{c}$, for $c\neq 0$, can be extended to a family of categories at $c=0$ by $\DCob{0,t}$. The morphism spaces $\Hom_{\DCob{0,t}}([n],[m])$ have bases given by all partitions of $n+m$. In comparison, the morphism spaces $\Hom_{\DCob{0}}([n],[m])$ are zero for $n+m\geq 0$. The statements from \Cref{prop:constant-irrep} and \Cref{prop:constant-grK0} also hold true for the category $\DCob{0,t}$.
\end{example}

\bibliography{biblio}
\bibliographystyle{amsrefs}%

\end{document}